\documentclass[11pt]{amsart}
\usepackage{amssymb,amsfonts,amsmath}
\usepackage{hyperref}

\usepackage[margin=1in]{geometry}

\usepackage[all]{xy}

\newcommand{\bB}{\mathbf{B}}

\newcommand{\cA}{\mathcal{A}}

\newcommand{\cD}{\mathcal{D}}
\newcommand{\cO}{\mathcal{O}}

\newcommand{\cS}{\mathcal{S}}
\newcommand{\cH}{\mathcal{H}}
\newcommand{\cJ}{\mathcal{J}}
\newcommand{\cL}{\mathcal{L}}

\newcommand{\cX}{\mathcal{X}}
\newcommand{\cG}{\mathcal{G}}

\newcommand{\fp}{\mathfrak{p}}

\newcommand{\CC}{\mathbb{C}}

\newcommand{\QQ}{\mathbb{Q}}

\newcommand{\ZZ}{\mathbb{Z}}

\newcommand{\bc}{\mathbf{c}}
\newcommand{\bH}{\mathbf{H}}

\newcommand{\bL}{\mathbf{L}}
\newcommand{\be}{\mathbf{e}}

\newcommand{\fc}{\mathfrak{c}}
\newcommand{\fS}{\mathfrak{S}}
\newcommand{\hh}{\mathfrak{h}}
\newcommand{\fg}{\mathfrak{g}}

\newcommand{\oA}{\overline{A}}

\DeclareMathOperator{\Ind}{Ind}
\DeclareMathOperator{\modd}{-mod}
\DeclareMathOperator{\KZ}{KZ}
\DeclareMathOperator{\supp}{supp}

\DeclareMathOperator{\Mat}{Mat}

\DeclareMathOperator{\Spec}{Spec}
\DeclareMathOperator{\Hom}{Hom}

\DeclareMathOperator{\gr}{gr}
\DeclareMathOperator{\End}{End}
\DeclareMathOperator{\triv}{triv}
\DeclareMathOperator{\Ext}{Ext}

\DeclareMathOperator{\HH}{HH}

\DeclareMathOperator{\ad}{ad}
\DeclareMathOperator{\Res}{Res}

\DeclareMathOperator{\sign}{sign}

\DeclareMathOperator{\R}{R}
\DeclareMathOperator{\Diff}{Diff}
\DeclareMathOperator{\Soc}{Soc}
\DeclareMathOperator{\HC}{HC}

\DeclareMathOperator{\Loc}{Loc}
\DeclareMathOperator{\Fun}{Fun}
\DeclareMathOperator{\opp}{opp}

\DeclareMathOperator{\Supp}{SS}
\DeclareMathOperator{\lann}{LAnn}
\DeclareMathOperator{\rann}{RAnn}
\DeclareMathOperator{\ann}{Ann}
\DeclareMathOperator{\bimod}{-bimod}
\DeclareMathOperator{\repp}{-rep}
\DeclareMathOperator{\codim}{codim}

\DeclareMathOperator{\nam}{Nam}

\newcommand{\homf}[2]{\Hom_{\mathtt{fin}}(#1,#2)}

\newtheorem{theorem}{Theorem}[section]
\newtheorem{rmk}[theorem]{Remark}
\newtheorem{prop}[theorem]{Proposition}
\newtheorem{definition}[theorem]{Definition}
\newtheorem{cor}[theorem]{Corollary}
\newtheorem{lemma}[theorem]{Lemma}

\newenvironment{proofmain}[1][\it{Proof of Theorem \ref{thm_main}}]{\textbf{#1. } }{$\square$}
\newenvironment{proofmainA}[1][\it{Proof of Theorem} \ref{thm:maintypeA}]{\textbf{#1. } }{$\square$}
\newenvironment{prooffull}[1][\it{Proof of Theorem} \ref{thm:main full support}]{\textbf{#1.} }{$\square$}
\newenvironment{proofthmvery}[1][\it{Proof of Theorem} \ref{thm:veryverytechnical}]{\textbf{#1.} }{$\square$}

\author{Jos\'e Simental}
\address{Department of Mathematics, Northeastern University. Boston, MA 02115. USA.}
\email{simentalrodriguez.j@husky.neu.edu}

\title{Harish-Chandra bimodules over rational Cherednik algebras}

\begin{document}

\begin{abstract}
We study Harish-Chandra bimodules over the rational Cherednik algebra $H_{c}(W)$ associated to a complex reflection group $W$ with parameter $c$. Our results allow us to partially reduce the study of these bimodules to smaller algebras. We classify those pairs of parameters $(c,c')$ for which there exist fully supported Harish-Chandra bimodules, and give a description of the category of all Harish-Chandra bimodules modulo those without full support. When $W$ is a symmetric group we are able to classify all irreducible Harish-Chandra bimodules. Our proofs are based on localization techniques, the action of the Namikawa-Weyl group on the set of parameters, and the study of partial KZ functors.
\end{abstract}
\maketitle 
\section{Introduction}\label{sect_intro}

In this paper we study Harish-Chandra bimodules over rational Cherednik algebras. Recall that a rational Cherednik algebra is an associative algebra $H_c := H_c(W, \hh)$ associated to a complex reflection group $W$ and its reflection representation $\hh$, see Subsection \ref{subsection:RCA} for a precise definition. This algebra depends on a parameter $c$, which is a conjugation invariant function $c: \cS \rightarrow \CC$, where $\cS$ is the set of reflections of $W$. The algebra $H_c$ is filtered, with associated graded $\gr H_c = \CC[\hh\oplus\hh^*]\# W$, the smash-product algebra. It follows that $H_c$ has a triangular decomposition $H_c = \CC[\hh]\otimes \CC W \otimes \CC[\hh^*]$, where $\CC[\hh], \CC[\hh^*] $ and $\CC W$ sit inside $H_c$ as subalgebras, similar to that of the universal enveloping algebra of a semisimple Lie algebra $\fg$. Then, the representation theory of the rational Cherednik algebra has many similarities with that of semisimple Lie algebras. For example, one has a category $\cO_{c}$, \cite{GGOR}, to be recalled in Subsection \ref{subsect:catO} below, that has been extensively studied in recent years, see e.g. \cite{BE, GGOR, gordon-losev, losev_cyclotomic, losev_derived, rouquier, RSVV, shan, webster, wilcox}. One also has a notion of Harish-Chandra bimodules, \cite{BEG}. These are the main object of study of this paper. Unlike category $\cO$, Harish-Chandra bimodules have not been extensively studied in the literature. Let us remark that, while in the Lie algebra setting category $\cO$ and the category of Harish-Chandra bimodules are very similar, cf. \cite[Section 5]{BG}, this is no longer the case in the Cherednik algebra setting see, for example, Subsection \ref{subsect:c=r/n} below. 

An $H_c\text{-}H_{c'}$-bimodule $B$ is said to be \emph{Harish-Chandra} (HC, for short) if it is finitely generated and the adjoint action of every element from $\CC[\hh]^W$ or $\CC[\hh^*]^W$ is locally nilpotent, \cite[Definition 3.2]{BEG}. Just as in the Lie algebra case, HC bimodules form a full Serre subcategory of the category of $H_{c}\text{-}H_{c'}$-bimodules, and the tensor product with a HC $H_{c}\text{-}H_{c'}$-bimodule $B$ induces a functor $B \otimes_{H_{c'}} \bullet : \cO_{c'} \rightarrow \cO_{c}$. Recently, these functors have been used to construct derived equivalences between categories $\cO$ for different parameters $c, c'$, see Subsection 5.4 in \cite{losev_derived}. 

An interesting problem, then, is to describe the category $\HC(c, c')$ of HC $H_{c}\text{-}H_{c'}$-bimodules. In this paper, we address the problem of classifying its irreducible objects. A classification of irreducible HC $H_c\text{-}H_{c'}$-bimodules has been carried out in \cite{BEG} for the case when $W$ is a Coxeter group and the parameters $c, c'$ are integral. Namely, in this case we have that the category $\HC(c,c')$ is semisimple and isomorphic to the category of finite dimensional representations of the group $W$. An explicit construction of its irreducibles is given in terms of spaces of locally finite maps $\homf{M}{N}$ for irreducible objects $M, N \in \cO_{c}$. Here, a map $f \in \Hom_{\CC}(M, N)$ is said to be \emph{locally finite} if it is locally nilpotent with respect to the operator $\ad(a)$ for all $a \in \CC[\hh]^{W}\cup\CC[\hh^{*}]^{W}$. Moreover, for $B_1 \in \HC(c, c')$, $B_2 \in \HC(c', c'')$ one has that $B_1 \otimes_{H_{c'}} B_2 \in \HC(c, c'')$, Proposition \ref{prop:HCproperties}. So the category $\HC(c,c)$ becomes a monoidal category and the equivalence $\HC(c,c) \cong W\!\repp$ is that of monoidal categories. 

In this paper, we generalize the previous result by giving a description of the category $\overline{\HC}(c,c')$ of all HC $H_{c}\text{-}H_{c'}$-bimodules modulo the full subcategory whose objects are HC bimodules whose singular support is strictly contained in $(\hh\oplus\hh^{*})/W$, see Subsection \ref{subsection:HC} for the definition of the singular support of a HC bimodule. Here, we allow $W$ to be \emph{any} complex reflection group, and the parameters $c, c'$ are arbitrary. First of all, if $\fp \cong \CC^{|\cS/W|}$ denotes the space of all parameters for the Cherednik algebra, there is a lattice $\fp_{\ZZ} \subseteq \fp$ that generalizes the notion of integrality mentioned in the previous paragraph, see Subsection \ref{subsect:lattice} for a precise definition. Associated to $W$, there is a product of symmetric groups (the Namikawa-Weyl group) that acts on $\fp$ in an affine way, see Subsection \ref{subsect:nami}. For a parameter $c$, we construct in Subsection \ref{subsect:subgroup} a normal subgroup $W_{c} \subseteq W$ satisfying the following properties:

\begin{enumerate}
\item $W_{c}$ is a reflection group.
\item $W_{c} = \{1\}$ if and only if $c$ is in the Namikawa-Weyl group orbit of an integral parameter. If $c$ is generic, then $W_{c} = W$.
\item $W_{c} = W_{c'}$ provided that either (i) $c - c'$ is integral, or (ii) $c$ and $c'$ are conjugate under the Namikawa-Weyl group action.
\end{enumerate}

\begin{theorem}\label{thm:main1}
Let $W$ be a complex reflection group, and let $c, c' \in \CC[\cS]$ be conjugation invariant functions. The following is true.
\begin{enumerate}
\item The category $\overline{\HC}(c,c')$ is nonzero if and only if there exists $\sigma$ in the Namikawa-Weyl group such that $c - \sigma c' \in \fp_{\ZZ}$.
\item If $\overline{\HC}(c,c')$ is nonzero, then it is equivalent to the category of representations of the group $W/W_{c}$. If $c = c'$, this is an equivalence of monoidal categories.
\end{enumerate}
\end{theorem} 

Let us remark that the Namikawa-Weyl group does not appear in the results of \cite{BEG2} since, when $W$ is a real reflection group, the Namikawa-Weyl group preserves the latice $\fp_{\ZZ}$ see, for example, Lemma \ref{lemma:namintegral}.



We study more closely the case where $W = \fS_n$, the symmetric group in $n$ elements. In this case, the possible supports of modules in category $\cO$ have been classified, \cite{BE}, and it is known which irreducibles have a given support, \cite{wilcox}. Using the results of {\it loc. cit.} together with Theorem \ref{thm_main}, we are able to describe all irreducible HC bimodules over $H_{c}(\fS_n)$, see Section \ref{sect:A}. Our main result for type A is the following.

\begin{theorem}\label{thm:main2}
Consider the rational Cherednik algebra $H_{c}(\fS_{n})$, where $c = r/m$ with $\gcd(r;m) = 1$ and $1 < m \leq n$. The possible supports for HC bimodules form a chain $(\hh\oplus\hh^{*})/\fS_{n} = \overline{\cL}_{0} \supsetneq \overline{\cL}_{1} \supsetneq \cdots \supsetneq \overline{\cL}_{\lfloor n/m\rfloor}$. For each $i = 0, \dots, \lfloor n/m\rfloor$, the category of HC $H_{c}$-bimodules supported on $\overline{\cL}_{i}$ modulo those supported on $\overline{\cL}_{i+1}$ is equivalent, as a monoidal category, to the category of representations of $\fS_{i}$.
\end{theorem} 

Theorem \ref{thm:main2} is proved in Subsection \ref{subsect:allbimodA}, see Theorem \ref{thm:maintypeA}. We also describe the category of HC bimodules over the algebra $H_{c}(\fS_n)$, $c = r/n$, with $\gcd(r;n) = 1$. In this case, the algebra $H_c(\fS_n)$ admits a finite dimensional representation, and the category of HC bimodules has two irreducible objects: a finite dimensional one, $M$, and the unique proper nonzero ideal in the algebra $H_c$, $\cJ$. Obviously, we have a nonsplit exact sequence $0 \rightarrow \cJ \rightarrow H_c \rightarrow M \rightarrow 0$. On the other hand, \cite[Theorem 7.7]{BL} constructs a non-split exact sequence $0 \rightarrow M \rightarrow D \rightarrow \cJ \rightarrow 0$. We provide a more simple-minded construction of the bimodule $D$ that the one given in \cite{BL}, and show that $M, \cJ, H_c$ and $D$ exhaust the indecomposable HC $H_c$-bimodules. 

The proofs of Theorems \ref{thm:main1} and \ref{thm:main2} depend on a reduction result that we give in Section \ref{sect_reduction}, see Theorem \ref{thm:veryverytechnical}. The statement is a bit technical, roughly speaking it gives us a criterium for extending a bimodule defined on an open subvariety of $\hh$ to the closure of that subvariety. This is used in Theorem \ref{thm:main1} to find a HC bimodule of a particular form whenever $\overline{\HC}(c,c') \neq 0$. In Theorem \ref{thm:main2}, it is used to compute the image of the restriction functors from \cite{losev_completions}, which are recalled in Subsection \ref{subsection:induction}.

To finish this section, let us sketch the structure of the paper. In Section 2 we recall the definition of the rational Cherednik algebra $H_{c}$ and its homogeneous version $\bH$, the spherical subalgebra $eH_{c}e$ and category $\cO_c$. We also recall finite Hecke algebras $\cH_{q}$, and give an overview of the functor $\KZ: \cO_{c} \rightarrow  \cH_{q}\modd$. This functor will be very important in our arguments. At the end of the section, we review known results about HC bimodules, including a useful alternative definition given in \cite{losev_completions}, as well as induction and restriction functors introduced in {\it loc. cit.} In Section 3 we state and prove our reduction theorem, Theorem \ref{thm:veryverytechnical}, that can be interpreted to give a lower bound on the number of irreducible HC bimodules with a given support. We would like to remark that this section is inspired by \cite{losev_hecke}, where a similar result is obtained on the level of category $\cO_{c}$. Section 4 uses localization techniques to study HC bimodules with full support. We remark that the main technical result of this section, Lemma \ref{lemma:mainKZ}, has already appeared in a slightly weaker form in \cite{spencer}. In Section 5 we introduce the Namikawa-Weyl group, study its action on the set of parameters, and apply this together with results of Sections 3 and 4 to give a description of the category $\overline{\HC}(c,c')$. Finally, in Section 6 we use the results of Sections 3 and 4, together with results from \cite{wilcox} to give a description of all (not only fully supported) irreducible HC bimodules over rational Cherednik algebras associated to symmetric groups. We also give a complete description of the category $\HC(H_{r/n}(\fS_n))$, where $\gcd(r;n) = 1$. The proof of this is based on the vanishing of several extension groups.
\vspace{0.25cm}

{\bf Acknowledgments.} This paper would've never appeared without Ivan Losev's help. I am very grateful to him for many interesting discussions and his countless remarks and suggestions. I'd also like to thank Iain Gordon for pointing out the reference \cite{spencer}, Seth Shelley-Abrahamson for stimulating discussions and comments, and the anonymous referee for helpful suggestions that allowed me to improve the exposition. 

\section{Preliminaries}\label{sect_prelim}
\subsection{Rational Cherednik Algebras.}\label{subsection:RCA}  Fix a complex reflection group $W$, let $\cS \subseteq W$ be the set of reflections and let $c: \cS \rightarrow \CC$ be a conjugation invariant function. For each reflection $s \in \cS$, let $\alpha_{s} \in \hh^{*}$ be an eigenvector with eigenvalue $\lambda_{s} \neq 1$. Also, let $\alpha_{s}^{\vee} \in \hh$ be an eigenvector of $s$ with eigenvalue $\lambda_{s}^{-1}$. We remark that $\alpha_{s}^{\vee}$, $\alpha_{s}$ are unique up to multiplication by a nonzero scalar, and we normalize so that $\langle \alpha_{s}, \alpha_{s}^{\vee}\rangle = 2$. The rational Cherednik algebra $H_{c} := H_{c}(W, \hh)$ is the quotient of the smash product algebra $T(\hh \oplus \hh^{*})\# W$, where $T(\bullet)$ denotes the tensor algebra, by the relations: 

\begin{equation}\label{eqn:relations}
\begin{array}{ll}
 [x, x'] = 0, \; \; \; \; \; \; \; \; \; \; \; \; \; \; \; [y, y'] = 0 & x, x' \in \hh^{*}, y, y' \in \hh \\
\hbox{} [y, x] = \langle y, x\rangle - \sum_{s \in \cS} c(s) \langle\alpha_{s}, y\rangle \langle x, \alpha_{s}^{\vee}\rangle s & x \in \hh^{*}, y \in \hh.
\end{array}
\end{equation}

The algebra $H_{c}$ is filtered, with $W, \hh^{*}$ in filtration degree $0$ and $\hh$ in filtration degree $1$. Its associated graded is the smash product $\CC[\hh\oplus\hh^{*}]\#W$. Thus, we have a triangular decomposition $H_{c} = \CC[\hh]\otimes \CC W \otimes \CC[\hh^{*}]$ given by the multiplication map, see e.g. \cite[Theorem 1.3]{EG}.

Let $\delta:= \displaystyle{\prod_{s \in \cS}\alpha_{s}}$. This is a $W$-semiinvariant element of $\CC[\hh]$. Let $\hh^{reg}$ be the principal open set in $\hh$ determined by $\delta$. Note that $\hh^{reg}$ coincides with the locus where the $W$-action is free. For some positive integer $m$, the element $\delta^{m}$ is $W$-invariant and the operator $[\delta^{m}, \cdot]$ is locally nilpotent, so the localization $H_{c}[\delta^{-1}]$ makes sense, and it is isomorphic to the algebra $\cD(\hh^{reg})\# W$.  Here, $\cD(\bullet)$ denotes the algebra of algebraic differential operators.

We will need the spherical rational Cherednik algebra which is constructed as follows. Let $e:= \frac{1}{|W|}\sum_{w \in W}w$ be the trivial idempotent in $\CC W$. Note that $e \in H_{c}$ is still an idempotent. The associated spherical subalgebra of $H_{c}$ is $eH_{c}e$. We remark that $eH_{c}e$ inherits a filtration from that of $H_{c}$, and $\gr(eH_{c}e) = \CC[\hh\oplus\hh^{*}]^{W}$. Note that $e\delta^{n} \in eH_{c}e$ for some $n > 0$, so that the localization $eH_{c}e[(e\delta^{n})^{-1}]$ makes sense and, moreover, it is isomorphic to $\cD(\hh^{reg}/W)$. If $c$ is a parameter such that the algebras $H_{c}$ and $eH_{c}e$ are Morita equivalent (that is, $H_{c} = H_{c}eH_{c}$) then we say that $c$ is \emph{spherical}.

\subsection{Homogeneous rational Cherednik algebras.} Sometimes, see e.g. Section \ref{sect_reduction} below, it will be more convenient to work with the homogeneous version of the rational Cherednik algebra. Namely, let $\cS = \bigsqcup_{i = 1}^{r} \cS_{i}$ be the decomposition of $\cS$ into conjugacy classes, and let $\hbar, \bc_1, \dots, \bc_r$ be independent variables. For $s \in \cS_{i}$, define $\bc(s) := \bc_{i}$.  Let $\fc$ be the vector space with basis $\hbar, \bc_1, \dots, \bc_r$. Then, $\bH$ is the $S(\fc)$-algebra defined by generators and relations analogous to those of the previous subsection, with the commutation relation between $\hh$ and $\hh^{*}$ replaced by

$$
[y, x] = \hbar\langle y, x\rangle - \sum_{s \in \cS} \bc(s)\langle\alpha_{s}, y\rangle \langle x, \alpha_{s}^{\vee}\rangle s ,
$$

\noindent note that the algebra $\bH$ is graded, with $W, \hh^{*}$ in degree $0$ and $\hh, \fc$ in degree $1$. We remark that $\bH$ is a flat $S(\fc)$-algebra, see e.g. \cite[Proposition 1.1.1]{losev_completions}, and that $\bH/\fc\bH = \CC[\hh\oplus\hh^{*}]\# W$. 

Let $\R_{\hbar}(H_{c})$ denote the Rees algebra of $H_{c}$ with respect to the filtration described in the previous subsection. We have a quotient map $\bH \longrightarrow \R_{\hbar}(H_{c})$, given by $w \mapsto w$, $x \mapsto x$, $y \mapsto \hbar y$, $\hbar \mapsto \hbar$, $\bc_{i} \mapsto \hbar c_{i}$, $w \in W, x \in \hh^{*}, y \in \hh$ and $c_{i} := c(s)$ for $s \in \cS_{i}$. So we can pass from $H_{c}$-modules to $\bH$-modules using the Rees construction with respect to some filtration, see for example Subsection \ref{subsection:HC}.

\subsection{Category $\cO_{c}$}\label{subsect:catO} The triangular decomposition $H_{c} = \CC[\hh] \otimes \CC W \otimes \CC[\hh^{*}]$ allows us to define a category $\cO_{c}$ of modules over $H_{c}$, \cite{GGOR}. By definition, $\cO_{c}$ is the full subcategory of the category of finitely generated $H_{c}$-modules consisting of those modules for which $\hh$ acts by locally nilpotent endomorphisms. For example, finite dimensional representations, when they exist, belong to $\cO_{c}$. Also, for an irreducible representation $\tau$ of $W$, consider $\tau$ as an $S(\hh)\#W$-module by letting $\hh$ act on $\tau$ by $0$. Then, the \emph{Verma module} $\Delta_{c}(\tau) := \Ind_{S(\hh)\#W}^{H_{c}}\tau = H_{c}\otimes_{S(\hh)\#W}\tau$ belongs to $\cO_{c}$. The Verma module $\Delta_{c}(\tau)$ admits a unique irreducible quotient, $L_{c}(\tau)$. Moreover,  the set $\{L_c(\tau): \tau \; \text{an irreducible representation of} \; W\}$ forms a complete and irredundant collection of irreducible objects in $\cO_{c}$. 

A module $M \in \cO_{c}$ is finitely generated over the subalgebra $\CC[\hh] \subseteq H_{c}$, so it can be viewed as a coherent sheaf on $\hh$. Hence, we may define its support $\supp(M)$ as the support of $M$ as a $\CC[\hh]$-module. This is a $W$-invariant subvariety of $\hh$. In fact, it is a union of strata of the stratification of $\hh$ with respect to stabilizers in $W$ of points in $\hh$. We remark that the support of an irreducible module $N \in \cO_{c}$ is irreducible when viewed as a subvariety of $\hh/W$, see \cite[Proposition 3.22]{BE}, and that a module $M \in \cO_{c}$ is finite dimensional if and only if its support consists of only the 0 point.

\subsection{Hecke algebras.} Let $B_W:= \pi_1(\hh^{reg}/W)$ be the generalized braid group associated to $W$. The group $B_{W}$ admits a system of generators indexed by the set $\cA$ of reflection hyperplanes on $\hh$. For each $\Gamma \in \cA$, the pointwise stabilizer $W_\Gamma$ is cyclic, of order say $\ell_\Gamma$. Let $s_{\Gamma} \in \cS\cap W_\Gamma$ be the element with determinant $\exp(2\pi\sqrt{-1}/\ell_\Gamma)$, and let $T_{\Gamma}$ be a generator of the monodromy around $\Gamma$ such that a lift of $T_\Gamma$ to $\hh^{reg}$ is represented by a path from $x_0$ to $s_\Gamma(x_0)$, see \cite[Appendix 1]{BMR} for a precise definition. The set $\{T_\Gamma\}_{\Gamma \in \cA}$ is a generating set for the group $B_W$. 

To define the Hecke algebra, for each reflection hyperplane $\Gamma \in \cA$, fix nonzero complex numbers $q_{\Gamma,0}, \dots$, $q_{\Gamma, \ell_{\Gamma} -1}$, in such a way that if $\Gamma, \Gamma'$ are $W$-conjugate then $q_{\Gamma, i} = q_{\Gamma', i}$ for each $i = 0, \dots, \ell_{\Gamma} - 1 = \ell_{\Gamma'} - 1$. We denote this collection of complex numbers by $q$. The Hecke algebra $\cH_{q}$ is, by definition, the quotient of the group algebra $\CC B_{W}$ by the relations $\prod_{i = 0}^{\ell_{\Gamma}-1}(T_{\Gamma} - q_{\Gamma, i})$, one for each $\Gamma \in \cA$. For example, setting $q_{\Gamma,i} = \exp(2\pi\sqrt{-1}i/\ell_{\Gamma})$ we recover the group algebra $\CC W$. We note that for each map $a: \cA \longrightarrow \CC^{\times}$ that is constant on $W$-conjugacy classes we can rescale the parameters $q_{\Gamma, i} \mapsto a(\Gamma)q_{\Gamma,i}$ without changing the algebra, cf. \cite[Subsection 3.3.3]{rouquier}. Most of the time below, we will normalize so that $q_{\Gamma,0} = 1$ for all $\Gamma \in \cA$.

\begin{rmk}
For the Hecke algebra of a Coxeter group, we use the normalization of the Hecke algebra $\cH_q$ whose quadratic relations read $(T_i - 1)(T_i + q) = 0$, which is not standard. Hence, our $T_i$'s differ from those of \cite{DiJa} by a factor of $q$, and our $q$ corresponds to $q^{-1}$ in \cite{DiJa}. Using this normalization, we say that an $\mathcal{H}_q$-module is trivial if all $T_{i}$'s act by $1$. See also remark after Theorem 2.6 in \cite{DiJa}. 
\end{rmk}

\subsection{KZ functor} In \cite{GGOR} it is shown that there exists a quotient functor $\KZ: \cO_{c} \rightarrow \cH_{q}(W)\modd$, the category of finite dimensional $\cH_{q}(W)$-modules, where the parameter $q$ explicitly depends on $c$ as follows. Recall that for each reflection $s \in \cS, \lambda_{s}$ denotes the unique non-trivial eigenvalue for the action of $s$ on $\hh^{*}$. For each reflection hyperplane $\Gamma \in \cA$, define 

\begin{equation}\label{eqn:h}
k_{\Gamma, i} := \sum_{s \in \cS\cap W_{\Gamma}} \frac{2c(s)}{1 - \lambda_{s}}(\lambda_{s}^{-i} - 1), \; \; \; \; i = 0, \dots, \ell_{\Gamma} - 1.
\end{equation}

Note that $k_{\Gamma, i}$ depends only on the conjugacy class of $\Gamma$, and that $k_{\Gamma, 0} = 0$. Now the parameter $q$ is computed as follows:

\begin{equation}\label{eqn:q}
q_{\Gamma, i} := \exp(2\pi\sqrt{-1}(k_{\Gamma, i} - i)/\ell_{\Gamma}).
\end{equation}

Note that $q_{\Gamma, 0} = 1$. Now we give the construction of the KZ functor. Start with a module $M \in \cO_{c}$. Then, the localization $M[\delta^{-1}]$ is a $\cD(\hh^{reg})\#W$-module. Since $W$ acts freely on $\hh^{reg}$, there is an equivalence of categories $\cD(\hh^{reg})\#W\!\modd \rightarrow \cD(\hh^{reg}/W)\modd$, given by $M \mapsto eM$. Since $M \in \cO_{c}$ is a $\CC[\hh]$-coherent module, we have that $eM[\delta^{-1}]$ must be a local system. Hence, $DR(eM[\delta^{-1}])$ is a representation of the braid group $B_W$, where $DR$ stands for the de Rham functor, see e.g. \cite[Chapter 7]{HTT}. By \cite[Theorem 5.13]{GGOR}, the action of $\CC B_W$ on $DR(eM[\delta^{-1}])$ factors through the Hecke algebra $\cH_{q}$. We abbreviate $\KZ(M) := DR(eM[\delta^{-1}])$.

For example, let $\tau$ be a 1-dimensional representation of $W$. For a hyperplane $\Gamma \in \cA$, let $s_{\Gamma}$ be a generator of $W_{\Gamma}$ with $\lambda_{s_{\Gamma}}^{-1} = \exp(2\pi\sqrt{-1}/\ell_{\Gamma})$. Denote by $\CC_{\Gamma, i}$ the 1-dimensional representation of $W_{\Gamma}$ where $s_{\Gamma}$ acts by $\lambda_{s_{\Gamma}}^{-i}$. In particular, we have that $\Res^{W}_{W_{\Gamma}}\tau = \CC_{\Gamma, \tau(\Gamma)}$ for $\tau(\Gamma) \in \{0, \dots, \ell_{\Gamma} - 1\}$.  Then, we have that $\KZ(\Delta(\tau))$ is the 1-dimensional $\cH_{q}$-module where $T_{\Gamma}$ acts by $q_{\Gamma, \Gamma(\tau)}$. In particular, $\KZ(\Delta(\triv))$ is the trivial representation of $\cH_{q}$. All claims in this paragraph follow easily from an explicit isomorphism $H_{c}[\delta^{-1}] \cong \cD(\hh^{reg})\#W$, which is given in terms of Dunkl-Opdam operators, see e.g. \cite[Proposition 4.5]{EG}. 

We remark that the local system $eM[\delta^{-1}]$ has regular singularities, see e.g. \cite[Section 5.3]{GGOR}. It follows, in particular, that the $\KZ$ functor is exact. Moreover, the essential image of $\KZ$ consists of all finite dimensional modules over the algebra $\cH_{q}$, \cite{losev_hecke}. The kernel of $\KZ$, $\cO_{c}^{tor}$, consists of those modules $M$ for which $\supp(M)$ is properly contained in $\hh$. This is a Serre subcategory of $\cO_{c}$. Then, $\KZ$ induces an equivalence $\cO_{c}/\cO_{c}^{tor} \cong \cH_{q}\modd$ where, recall, the latter category is that of finite dimensional $\cH_{q}$-modules.

\subsection{Lattice $\fp_{\ZZ}$}\label{subsect:lattice} Let $\fp^{*}:= \fc/\hbar \cong \CC^{|\cS/W|}$, a vector space with basis $\bc_{1}, \dots, \bc_{r}$. We remark that the set of parameters \lq$c$\rq\; for the Cherednik algebra $H_{c}$ can be naturally identified with $\fp$, the dual of $\fp^{*}$. So we can view $k_{\Gamma, i}$ as an element of $\fp^{*}$, its value on a parameter $c \in \fp$ is given by the formula (\ref{eqn:h}), in particular, $k_{\Gamma, 0}$ is the constant function $0$. Define $\fp^{*}_{\ZZ}$ to be the $\ZZ$-lattice inside $\fp^{*}$ spanned by elements $\ell_{\Gamma}^{-1}k_{\Gamma, i}$, and let $\fp_{\ZZ} \subseteq \fp$ be the dual lattice. Note that the lattice $\fp_{\ZZ}$ consists of all parameters $c$ such that $q_{\Gamma, i} = \eta_{\Gamma}^{-i}$ for every $\Gamma \in \cA$ and $i = 0, \dots, \ell_{\Gamma} - 1$. Moreover, we have the following.

\begin{rmk}\label{rmk:integral}
	For $c, c' \in \fp$, we have $c - c' \in \fp_{\ZZ}$ if and only if $q(c)_{\Gamma, i} = q(c')_{\Gamma, i}$ for every $\Gamma \in \cA$ and $i = 0, \dots, \ell_{\Gamma} - 1$. 
\end{rmk} 

Thus, the set of parameters for the Hecke algebra can be identified with $\fp/\fp_{\ZZ}$, see \cite[Subsection 2.6]{losev_derived}.  For example, if $W$ is a Coxeter group, then $\fp_{\ZZ}$ coincides with the set of parameters for which $c(s) \in \ZZ$ for all $s \in \cS$.

We will need a spanning set for $\fp_{\ZZ}$. First, let us introduce some notation. For each $\Gamma \in \cA$, the set of characters of $W_{\Gamma}$ is identified with $\ZZ/\ell_{\Gamma}\ZZ$, an isomorphism $\ZZ/\ell_{\Gamma}\ZZ \buildrel \cong \over \longrightarrow \Hom(W_{\Gamma}, \CC^{\times})$ is given by $m \mapsto (s \mapsto \det(s)^{m})$. We have a morphism $\Hom(W, \CC^{\times}) \longrightarrow \prod_{\Gamma \in \cA/W} \Hom(W_{\Gamma}, \CC^{\times})$, given by restriction. This is an isomorphism, \cite[Subsection 3.3.1]{rouquier}. Thus, we have a correspondence between 1-dimensional characters $\chi$ of $W$ and $|\cA/W|$-tuples of integers $(m_{\Gamma})$ with $0 \leq m_{\Gamma} \leq \ell_{\Gamma} - 1$.  So, to a character $\chi \in \Hom(W, \CC^{\times})$ associated to the tuple $(m_{\Gamma})$ we assign $\overline{\chi} \in \fp$, given by 

$$
\overline{\chi}(k_{\Gamma, i}) = \begin{cases} \ell_{\Gamma}  & \text{if} \; i \geq \ell - m_{\Gamma} \\ 0 & \text{if} i < \ell - m_{\Gamma}.\end{cases}
$$

\noindent Clearly, $\overline{\chi} \in \fp_{\ZZ}$, and the elements $\overline{\chi}$ form a spanning set for $\fp_{\ZZ}$. 

Let us explain the reason why we are interested in the elements $\overline{\chi}$. Recall that we have an embedding $H_{c} \hookrightarrow \cD(\hh^{reg})\#W$. We have an automorphism of $\cD(\hh^{reg})\#W$, given by $w \mapsto \chi(w)w$, $x \mapsto x$, $\partial_{y} \mapsto \partial_{y}$, $w \in W$, $x \in \hh^{*}$, $y \in \hh$. Then, according to \cite[Section 5.1]{BC}, under this automorphism $H_{c}$ transforms to $H_{c + \overline{\chi}}$, while $eH_{c}e$ transforms to $e_{\chi}H_{c + \overline{\chi}}e_{\chi}$. Here, $e_{\chi}$ denotes the idempotent corresponding to $\chi$, $e_{\chi} = |W|^{-1}\sum_{w \in W}\chi(w^{-1})w$. We will use this in Subsection \ref{subsect:proof full support}, see in particular Lemma \ref{lemma:naminonzero}.
 
\subsection{Restriction functors for category $\cO$}\label{subsect:resO} We remark that, if $W'$ is a parabolic subgroup of $W$ (that is, the stabilizer of a point in $\hh$) then there is a natural inclusion of algebras $\iota: \cH_{q}(W') \hookrightarrow \cH_{q}(W)$ where, abusing the notation, we also denote by $q$ its restriction to $\cS \cap W'$. This allows us to define a restriction functor $^{\cH}\!\Res^{W}_{W'} := \iota^{*}: \cH_{q}(W)\modd \rightarrow \cH_{q}(W')\modd$. 

There is also a restriction functor on the level of category $\cO$, \cite{BE}. This functor depends on the choice of a point $b \in \hh$ whose stabilizer $W_{b}$ coincides with $W'$. For distinct $b, b'$ with this property, the functors are isomorphic (but not canonically so) so we will just denote this functor by $\Res^{W}_{W'}$. This functor is defined as follows. Let $b \in \hh$ be a point with $W_{b} = W'$. We will denote also by $b$ its projection to $\hh/W$. Consider $\CC[\hh/W]^{\wedge b}$, the completion of $\CC[\hh/W]$ at the maximal ideal of $b$. Then it can be easily checked that $H_{c}^{\wedge b} := \CC[\hh/W]^{\wedge b}\otimes_{\CC[\hh/W]}H_{c}$ is naturally an algebra. Bezrukavnikov and Etingof showed in \cite[Theorem 3.2]{BE} that $H_{c}^{\wedge b}$ is isomorphic to a matrix algebra of size $|W/W'|$ with coefficients in $H_{c}(W', \hh)^{\wedge 0}$, see Subsection \ref{subsect:BE} for a more precise statement. For a module $M$ in category $\cO$, let $E(M^{\wedge b})$ denote the $\hh$-locally nilpotent part of $M^{\wedge b}$ where, abusing the notation, we denote by $M^{\wedge b}$ the image of the completion $\CC[\hh/W]^{\wedge b}\otimes_{\CC[\hh/W]} M$ under an equivalence $H_{c}^{\wedge b}\modd \longrightarrow H_{c}(W', \hh)^{\wedge 0}\modd$. Then, $\Res^{W}_{W'}(M)$ is defined to be $\{v \in E(M^{\wedge b}) : yv = 0 \; \text{for all} \; y \in \hh^{W'}\}$. This is a module in category $\cO$ for $H_{c}(W', \hh/\hh^{W'})$, cf. \cite[Section 2]{BE}. 

We will also need a certain compatibility between the restriction functors and the KZ functor that was stablished in \cite[Section 2]{shan}. Namely, let $W'$ be a parabolic subgroup of $W$. Then, as is checked in \cite[Theorem 2.1]{shan}, we have $\underline{\KZ}\circ \Res^{W}_{W'} = \, ^{\cH}\!\Res^{W}_{W'}\circ\KZ$, where $\underline{\KZ}$ denotes the KZ functor from category $\cO_{c}(W', \hh/\hh^{W'})$ to $\cH_{q}(W')$. 

\subsection{Harish-Chandra bimodules.}\label{subsection:HC} By a HC $H_{c}\text{-}H_{c'}$-bimodule we mean a finitely generated $H_{c}\text{-}H_{c'}$-bimodule such that the adjoint action of every element of $\CC[\hh]^W\cup\CC[\hh^*]^W$ is locally nilpotent. For example, the algebra $H_c$ is always a HC $H_c$-bimodule. Clearly, the category $\HC(c,c')$ of HC $H_{c}\text{-}H_{c'}$-bimodules is a Serre subcategory of the category of all $H_{c}\text{-}H_{c'}$-bimodules. The following proposition gives basic properties of HC bimodules. For its proof, see e.g. Lemma 3.3 in \cite{BEG}, or Proposition 3.1 in \cite{losev_derived}.

\begin{prop}\label{prop:HCproperties}
\begin{enumerate}
\item Any $B \in \HC(c,c')$ is finitely generated as a left $H_{c}$-module, as a right $H_{c'}$-module, and as a $\CC[\hh]^W\otimes\CC[\hh^*]^W$-module (here, $\CC[\hh]^W$ is considered inside $H_c$, while $\CC[\hh^*]^W$ is considered inside $H_{c'}$).
\item If $B \in \HC(c,c')$, $B' \in \HC(c',c'')$ then $B \otimes_{H_{c'}} B' \in \HC(c,c'')$.
\item If $B \in \HC(c,c')$ and $M \in \cO_{c'}$, then $B \otimes_{H_{c'}} M \in \cO_{c}$.
\end{enumerate}
\end{prop}

A way to construct HC bimodules is as follows. Consider modules $N \in \cO_{c}$, $M \in \cO_{c'}$. Then, $\Hom_{\CC}(M, N)$ is an $H_{c}\text{-}H_{c'}$-bimodule. By $\homf{M}{N}$ we mean the sub-bimodule of $\Hom_{\CC}(M, N)$ consisting of all those vectors that are locally nilpotent under the adjoint action of $\CC[\hh]^W\cup\CC[\hh^*]^W$. Clearly, $\homf{M}{N}$ is the direct limit (= union) of its HC sub-bimodules. It was checked in \cite[Proposition 5.7.1]{losev_completions} that $\homf{M}{N}$ is actually HC. We remark, \cite[Lemma 5.7.2]{losev_completions}, that if $M$ and $N$ are irreducible then $\homf{M}{N} = 0$ unless $\supp(M) = \supp(N)$. Also, note that $\homf{M}{N} \neq 0$ if and only if there exists a HC $H_{c}\text{-}H_{c'}$-bimodule $B$ and a nonzero morphism of left $H_{c}$-modules $B \otimes_{H_{c'}} M \rightarrow N$.

\begin{prop}\label{prop:inclocfin}
Let $B$ be an irreducible HC $H_{c}\text{-}H_{c'}$-bimodule. Then, there exist irreducible modules $M \in \cO_{c'}$, $N \in \cO_{c}$ and a monomorphism $B \hookrightarrow \homf{M}{N}$.
\end{prop}
\begin{proof}
By \cite[Lemma 3.10]{losev_derived}, there exists an irreducible module $M \in \cO_{c'}$ with $B \otimes_{H_{c'}} M \neq 0$. Since the latter module is in category $\cO_{c}$, there exists an irreducible module $N \in \cO_{c}$ and a nonzero map $f: B \otimes_{H_{c'}} M \rightarrow N$. Then, $v \mapsto (m \mapsto f(v\otimes_{H_{c'}} m))$ defines a nonzero morphism $B \rightarrow \homf{M}{N}$.
\end{proof}

An equivalent definition of HC bimodules was found in \cite[Section 3]{losev_completions}. Namely, recall that the algebras $H_{c}, H_{c'}$ are filtered and there is a natural identification $\gr H_{c} = \gr H_{c'} = \CC[\hh\oplus\hh^{*}]\#W$, see Subsection \ref{subsection:RCA}. Let $B$ be a filtered $H_{c}\text{-}H_{c'}$-bimodule. Then, $\gr B$ is a $\CC[\hh\oplus\hh^{*}]\#W$-bimodule. It is proved in \cite[Subsection 5.4]{losev_completions}, that $B$ is HC if and only if it admits a bimodule filtration such that $\gr B$ is a finitely generated $\CC[\hh\oplus\hh^{*}]\#W$-bimodule and, moreover, the left and right actions of $\CC[\hh\oplus\hh^{*}]^{W} = Z(\CC[\hh\oplus\hh^{*}]\#W)$ on $\gr B$ coincide. We call such a filtration on $B$ {\it good}. Note that this implies that $\gr B$ is a finitely generated \emph{module} over $\CC[\hh\oplus\hh^{*}]^{W}$. Then, we define the (singular) support of $B$, $\Supp(B) \subseteq (\hh\oplus\hh^{*})/W$ to be the support of $\gr B$ as a $\CC[\hh\oplus\hh^{*}]^{W}$-module. We remark that, as usual, $\gr B$ depends on the filtration, but its support does not. We also remark that $\Supp(B)$ is a Poisson subvariety of $(\hh\oplus\hh^{*})/W$ and, moreover, it is a union of symplectic leaves. The variety $(\hh\oplus\hh^{*})/W$ has finitely many symplectic leaves, see e.g. \cite[Subsection 7.4]{BrGo}.

\begin{lemma}[\cite{losev_bernstein}, Lemma 4.2]\label{lemma:annihilators}
Let $B$ be a HC $H_{c}\text{-}H_{c'}$-bimodule. Then, $\Supp(B) = \Supp(H_{c}/\lann(B)) = \Supp(H_{c'}/\rann(B))$, where $\lann(B), \rann(B)$ denote the left and right annihilator of $B$, respectively, and $H_{c}/\lann(V)$ (resp. $H_{c'}/\rann(V)$) is viewed as a HC $H_{c}$-bimodule (resp. HC $H_{c'}$-bimodule).
\end{lemma}

\begin{lemma}
Let $B$ be an irreducible HC $H_{c}\text{-}H_{c'}$-bimodule, and let $N \in \cO_{c'}$ be irreducible. Then, $B \otimes_{H_{c'}} N = 0$ unless $\rann(B) = \ann(N)$. If $B \otimes_{H_{c'}}N \neq 0$, then the annihilator of every irreducible quotient of $B \otimes_{H_{c'}} N$ coincides with $\lann(B)$. Moreover, $\lann(B) = \ann(B \otimes_{H_{c'}}N)$.
\end{lemma}
\begin{proof}
Assume $B \otimes_{H_{c'}} N \neq 0$, and let $M$ be an irreducible quotient of $B \otimes_{H_{c'}} N$. Then, as in Proposition \ref{prop:inclocfin}, we have an inclusion $B \hookrightarrow \homf{N}{M}$. First, we check that $\rann(B) = \ann(N)$. Note that $\bigcap_{f \in B}\ker(f)$ is a proper submodule of $N$. Since $N$ is irreducible, we must have $\bigcap_{f \in B}\ker(f) = 0$. Now, if $a \in \rann(B)$, then $aN \subseteq \bigcap_{f \in B}\ker(f)$, so $a \in \ann(N)$. The other inclusion is clear. To show that $\lann(B) = \ann(M)$, we observe that $\sum_{f \in B} f(N)$ is a nonzero submodule of $M$, so we must have $\sum_{f \in B}f(N) = M$. From here it follows easily that $\lann(B) = \ann(M)$. To prove the last statement of the lemma, note that we clearly have $\lann(B) \subseteq \ann(B \otimes_{H_{c'}}N)$. On the other hand, by what we just proved $\ann(B \otimes_{H_{c'}}N) \subseteq \ann(M) = \lann(B)$. So $\lann(B) = \ann(B \otimes_{H_{c'}} N)$.
\end{proof}

\begin{cor}
Let $B_1$ be an irreducible HC $H_{c}\text{-}H_{c'}$-bimodule, and $B_2$ an irreducible HC $H_{c'}\text{-}H_{c''}$-bimodule. Then, $B_1 \otimes_{H_{c'}} B_2 = 0$ unless $\Supp(B_1) = \Supp(B_2)$.
\end{cor}
\begin{proof}
Assume that $\Supp(B_1) \neq \Supp(B_2)$, and denote $B := B_1 \otimes_{H_{c'}} B_2$. First, we assume that $\Supp(B_1) \not\subseteq \Supp(B_2)$. By \cite[Lemma 3.10]{losev_derived} it is enough to show that $B \otimes_{H_{c''}} N = 0$ for all irreducible modules $N \in \cO_{c''}$. If $B_2 \otimes_{H_{c''}} N = 0$ we are done. So we may assume that $B_2 \otimes_{H_{c''}} N \neq 0$. By the previous lemma this implies that $\ann(B_2 \otimes_{H_{c''}} N) = \lann(B_2)$. If $B_1 \otimes_{H_{c'}} (B_2 \otimes_{H_{c''}} N) \neq 0$, then $B_1 \otimes_{H_{c'}} M \neq 0$ for some irreducible subquotient $M$ of $B_2 \otimes_{H_{c''}} N$. So $\rann(B_1) = \ann(M) \supseteq \ann(B_2 \otimes_{H_{c}} N) = \lann(B_2)$. Thus, $\Supp(B_1) = \Supp(H_{c'}/\rann(B_1)) \subseteq \Supp(H_{c'}/\lann(B_2)) = \Supp(B_2)$, a contradiction with our assumption. We conclude that $B_1 \otimes_{H_{c'}} B_2 = 0$. 

Now assume that $\Supp(B_2) \not\subseteq \Supp(B_1)$. Let $c^{\opp}$ be defined by $c^{\opp}(s) := -c(s^{-1})$. Then, it is easy to check that we have an isomorphism $H_{c}(W, \hh) \rightarrow H_{c^{\opp}}(W, \hh^{*})^{\opp}$ given by $x \mapsto x, y \mapsto y, w \mapsto w^{-1}$, $x \in \hh^{*}, y \in \hh, w \in W$. We get an equivalence $\rho_{c,c'}: H_{c}\text{-}H_{c'}\bimod \rightarrow H_{(c')^{\opp}}(W, \hh^{*})\text{-}H_{c^{\opp}}(W,\hh^{*})\bimod$. Similarly, we get equivalences $\rho_{c',c''}, \rho_{c,c''}$. Note that these equivalences preserve the categories of HC bimodules as well as the support of a HC bimodule. We have that $\rho_{c,c''}(B_1 \otimes_{H_{c'}} B_2) = \rho_{c',c''}(B_2)\otimes_{H_{(c')^{\opp}}(W, \hh^{*})} \rho_{c,c'}(B_1)$. Thus, the result in this case follows from the previous paragraph.
\end{proof}

The definition of a HC bimodule given in \cite[Section 3]{losev_completions} also allows us to give a definition of a HC $\bH$-bimodule in such a way that, if $B$ is a HC $H_{c}\text{-}H_{c'}$-bimodule with a good filtration, then $\R_{\hbar}(B)$ is a HC $\bH$-bimodule (recall that $\R_{\hbar}(H_{c})$, $\R_{\hbar}(H_{c'})$ are quotients of $\bH$). 

\begin{definition}\label{def_HCbimod}
A HC $\bH$-bimodule $\bB$ is a finitely generated, graded $\bH$-bimodule satisfying the following conditions:
\begin{enumerate}
\item[(i)] The left and right actions of $\hbar$ on $\bB$ coincide.
\item[(ii)] $\bB$ is flat as a $\CC[\hbar]$-module.
\item[(iii)] The left and right actions of $Z(\bH/\hbar\bH)$ on $\bB/\hbar\bB$ coincide.
\end{enumerate}
\end{definition}

Note that, since the grading on $\bH$ is positive, it follows that the grading on any HC $\bH$-bimodule $\bB$ is bounded below. 

We remark that we also have a notion of Harish-Chandra bimodule for spherical rational Cherednik algebras. In fact, the definition does not change, because $\CC[\hh]^{W}, \CC[\hh^{*}]^{W} \subseteq eH_{c}e$. Equivalently, an $eH_{c}e\text{-}eH_{c'}e$-bimodule $B$ is HC if it admits a bimodule filtration so that $\gr B$ is a finitely generated $\CC[\hh\oplus\hh^{*}]^{W}$-module. If the parameter $c$ is spherical, then the categories of HC $H_{c}$-bimodules and HC $eH_{c}e$-bimodules are equivalent, an equivalence is given by $B \mapsto eBe$. This equivalence intertwines the tensor products of HC bimodules.
 
\subsection{Induction and restriction functors for HC bimodules.}\label{subsection:induction} We will need induction and restriction functors for HC bimodules, introduced in \cite[Section 3]{losev_completions}, see also \cite[Section 3]{losev_derived}. Namely, for $x \in (\hh \oplus \hh^{*})/W$, let $W'$ be the stabilizer of a corresponding point $\overline{x} \in (\hh \oplus \hh^{*})$, and let $\Xi:= N_{W}(W')/W'$. There is a decomposition $\hh = \hh^{W'} \oplus \hh_{W'}$, where $\hh_{W'}$ is a unique $W'$-invariant complement of $\hh^{W'}$ in $\hh$. Let $\underline{H}_{c}$ be the Cherednik algebra for the pair $(W', \hh_{W'})$. The group $N_{W}(\underline{W})$ acts on $\underline{H}_c$ by algebra automorphisms in such a way that the action of $\underline{W} \subset N_{W}(\underline{W})$ is the adjoint one, this follows because the group $N_{W}(W')$ setwise stabilizes $\hh_{W'}$. A $\Xi$-equivariant HC $\underline{H}_{c}\text{-}\underline{H}_{c'}$-bimodule is, by definition, a HC $\underline{H}_c\text{-}\underline{H}_{c'}$-bimodule $B$ with a $N_{W}(W')$-action such that the action of $W' \subseteq N_{W}(W')$ coincides with the adjoint action of $W'$, $w'.b = w'bw'^{-1}$, and the structure map $\underline{H}_{c} \otimes B \otimes \underline{H}_{c'} \rightarrow B$ is $N_{W}(W')$-equivariant. Denote by $\HC^{\Xi}(\underline{H}_c)$ the category of $\Xi$-equivariant HC $\underline{H}_c$-bimodules. Then, \cite[Theorem 3.4.6]{losev_completions}, there exists an exact $\CC$-linear functor $\bullet_{\dagger, W'}: \HC(H_c, H_{c'}) \rightarrow \HC^{\Xi}(\underline{H}_c, \underline{H}_{c'})$ that admits a right adjoint $\bullet^{\dagger, W'}: \HC^{\Xi}(\underline{H}_c, \underline{H}_{c'}) \rightarrow \widehat{\HC}(H_c, H_{c'})$, the category of bimodules that are the direct limit of their HC sub-bimodules. The functor $\bullet_{\dagger,W'}$ intertwines the tensor product functors. Moreover, it satisfies the following: for a symplectic leaf $\cL \subseteq (\hh\oplus\hh^*)/W$, let $\HC_{\overline{\cL}}(H_c, H_{c'})$ be the full subcategory of $\HC(H_c, H_{c'})$ consisting of all HC bimodules $B$ with $\Supp(B) \subseteq \overline{\cL}$. Define the category $\HC_{\partial\cL}(H_c,H_{c'})$ in a similar way, where $\partial\cL := \overline{\cL}\setminus \cL$. This is a Serre subcategory of $\HC_{\overline{\cL}}(H_c, H_{c'})$. Form the quotient category $\HC_{\cL}(H_c, H_{c'}) := \HC_{\overline{\cL}}(H_c, H_{c'})/\HC_{\partial\cL}(H_c, H_{c'})$. Then, the above construction for a point in the symplectic leaf $\cL$ restricts to a functor $\bullet_{\dagger,W'}: \HC_{\overline{\cL}}(H_c, H_{c'}) \rightarrow \HC_0^{\Xi}(\underline{H}_c, \underline{H}_{c'})$, where the latter category consists of $\Xi$-equivariant finite dimensional HC $\underline{H}_{c}\text{-}\underline{H}_{c'}$ bimodules. Moreover, this functor factors through the quotient $\HC_{\cL}(H_c, H_{c'})$, and identifies this category with a full subcategory of $\HC_0^{\Xi}(\underline{H}_c,\underline{H}_{c'})$ closed under taking subquotients, \cite[Theorem 3.4.6]{losev_completions}. If $c = c'$, then this subcategory is also closed under tensor products. We remark that the functor $\bullet_{\dagger,W'}$ can be upgraded to a functor between HC bimodules for the homogenized rational Cherednik algebras, see \cite[Section 3]{losev_completions}.

Using restriction functors and the fact that $(\hh\oplus\hh^{*})/W$ has finitely many symplectic leaves, \cite[Subsection 4.1]{losev_bernstein} proves the following.

\begin{prop}\label{prop:finitelength}
Any HC $H_{c}\text{-}H_{c'}$-bimodule has finite length.
\end{prop}

Let us see some more applications of the restriction functors. The first one of these tells us that a certain quotient of the category of all HC $H_{c}$-bimodules is semisimple. We remark that the Poisson variety $(\hh\oplus \hh^{*})/W$ has a unique dense symplectic leaf $\cL$, which coincides with the projection of the set of points in $\hh\oplus\hh^{*}$ with trivial stabilizer. The category $\HC_{\cL}(H_{c})$ is the quotient of the category of all HC $H_{c}$-bimodules modulo the Serre subcategory formed by HC bimodules with proper support, and we denote $\overline{\HC}(H_{c}) := \HC_{\cL}(H_{c})$. 

\begin{prop}\label{prop:ss1param}
The category $\overline{\HC}(c,c)$ is semisimple.
\end{prop}
\begin{proof}
Pick a point $x \in \hh\oplus\hh^{*}$ whose stabilizer in $W$ is trivial, its projection to $(\hh\oplus\hh^{*})/W$ is a point in the open symplectic leaf $\cL$. Note that $\Xi = W$, and $\underline{H}_{c} = \CC$, so $\HC^{\Xi}_{0}(\underline{H}_{c})$ is precisely the category of finite dimensional representations of $W$. The results from \cite{losev_completions} mentioned above imply that $\overline{\HC}(c,c)$ can be embedded as a full subcategory of the category of representations of $W$. Moreover, this subcategory is closed under subquotients and tensor products. It follows that $\overline{\HC}(c,c)$ is equivalent to the category of representations of $W/N$ for a normal subgroup $N \subseteq W$. In particular, it is a semisimple category. 
\end{proof}

In Section \ref{sect:fullsupport}, we will find an explicit description of the subgroup $N$ that appears in the proof of Proposition \ref{prop:ss1param}, see Subsection \ref{subsect:subgroup}.. We will also see, Corollary \ref{cor:semisimple}, that $\overline{\HC}(H_{c}, H_{c'})$ is semisimple for \emph{different} parameters $c, c'$. Another application of restriction functors is the following result.

\begin{prop}\label{prop:injective}
The regular bimodule $H_{c}$ is injective in the category of HC $H_{c}$-bimodules.
\end{prop}
\begin{proof}
In view of Proposition \ref{prop:finitelength}, we need to show that $\Ext(B, H_{c}) = 0$ for any irreducible HC bimodule $B$, where $\Ext$ denotes $\Ext^{1}_{H_{c}\bimod}$. We separate in two cases.

{\it Case 1: $B$ has proper support}. This case is contained in an old version of the paper \cite{BL}. We provide a proof for the reader's convenience. Consider an exact sequence $0 \rightarrow H_{c} \rightarrow X \rightarrow B \rightarrow 0$. Let $x$ be in the open symplectic leaf of $(\hh\oplus\hh^{*})/W$, and consider the corresponding restriction functor $\bullet_{\dagger} := \bullet_{\dagger, x}$. Note that $B_{\dagger} = 0$. Since the restriction functor is exact, we must then have $((H_{c})_{\dagger})^{\dagger} = (X_{\dagger})^{\dagger}$. We have the adjunction map $X \rightarrow ((H_{c})_{\dagger})^{\dagger}$. The latter bimodule admits a filtration whose associated graded is contained in $((\CC[\hh\oplus\hh^{*}]\#W)_{\dagger})^{\dagger}$, see \cite[Subsection 3.6]{losev_completions} for a construction of the functor $\bullet_{\dagger}$ for the algebra $\CC[\hh\oplus\hh^{*}]\#W$. By construction, $((\CC[\hh\oplus\hh^{*}]\#W)_{\dagger})^{\dagger}$ is the global sections of the restriction of $\CC[\hh\oplus\hh^{*}]\#W$ to the open symplectic leaf of $(\hh\oplus\hh^{*})/W$. But the complement of this leaf has codimension 2. Hence, $ ((\CC[\hh\oplus\hh^{*}]\#W)_{\dagger})^{\dagger} = \CC[\hh\oplus\hh^*]\#W$, and this implies that $((H_{c})_{\dagger})^{\dagger} = H_{c}$. Now the adjunction map $X \rightarrow H_{c}$ is a splitting of the exact sequence $0 \rightarrow H_{c} \rightarrow X \rightarrow B \rightarrow 0$.

{\it Case 2: $B$ has full support}. Assume $0 \rightarrow H_c \buildrel \varphi \over \longrightarrow X  \longrightarrow B \rightarrow 0$ is an exact sequence. Pick $x$ in the open symplectic leaf of $(\hh \oplus \hh^{*})/W$, and consider the corresponding restriction functor $\bullet_{\dagger}$. We have an exact sequence $0 \rightarrow (H_c)_{\dagger} \rightarrow X_\dagger \rightarrow B_{\dagger} \rightarrow 0$. Since the category of $\Xi$-equivariant Harish-Chandra $\underline{H}_c$-bimodules is semisimple, Proposition \ref{prop:ss1param}, this exact sequence splits, $X_{\dagger} = (H_c)_{\dagger} \oplus B_{\dagger}$. Now, recall from the previous case that $(H_{c, \dagger})^{\dagger} = H_c$, and that we have the adjunction morphism $X \rightarrow (X_{\dagger})^{\dagger} = H_c \oplus (B_{\dagger})^{\dagger}$. By \cite[Theorem 3.7.3]{losev_completions}, the kernel of this morphism is a HC bimodule with proper support. We claim that $X$ does not have sub-bimodules with proper support. Since $B$ has full support, our claim will follow if we check that $H_{c}$ has no sub-bimodules (= ideals) with proper support. This is an immediate consequence of the claim that any ideal of the algebra $\CC[\hh\oplus\hh^{*}]\#W$ has full support as a $\CC[\hh\oplus\hh^{*}]^{W}$-module, which is clear. Thus, we can consider $X \subseteq H_{c} \oplus (B_{\dagger})^{\dagger}$, and $\varphi = (\varphi_1, \varphi_2)$, where $\varphi_1: H_{c} \rightarrow H_{c}$, $\varphi_{2}: H_{c} \rightarrow (B_{\dagger})^{\dagger}$. Now, $\End_{H_{c}\bimod}(H_{c})$ coincides with the center of the algebra $H_{c}$ and, since this is trivial, it follows that every nonzero endomorphism of $H_{c}$ is an automorphism. So, if $\varphi_1 \neq 0$, we can find a splitting for $\varphi$. Thus, we may assume $\varphi_1 = 0$, and $\varphi_2: H_{c} \rightarrow (B_{\dagger})^{\dagger}$ is an inclusion. 

Now recall that we have the adjunction morphism $B \rightarrow (B_{\dagger})^{\dagger}$. Since $B$ is irreducible, this is actually an injection. The cokernel of this morphism is a HC bimodule with proper support, this follows from \cite[Proposition 3.7.3]{losev_completions}. Thus, we must have $B \subseteq \varphi_{2}(H_{c})$, so $B$ is isomorphic to an ideal of $H_{c}$. But this implies that $B_{\dagger} = (H_{c})_{\dagger}$.  So the exact sequence $0 \rightarrow H_{c} \rightarrow X \rightarrow B \rightarrow 0$ induces an inclusion $X \subseteq (X_{\dagger})^{\dagger} = H_{c} \oplus H_{c}$ and, reasoning as in the previous paragraph, we can find a splitting for $\varphi$. Thus, $\Ext(H_{c}, B) = 0$.
\end{proof}

We now explain a way to construct $\bullet_{\dagger}$ that is convenient for us. We follow \cite[Section 3.2]{losev_derived}. First, we explain how to construct $\bullet_{\dagger}$ for the homogeneous rational Cherednik algebra. Let $\underline{\cL} \subseteq \hh^{reg-W'}/W'$ be the projection of $\{x \in \hh : W_{x} = W'\} \subseteq \hh$. Here, $\hh^{reg-W'}$ denotes the principal open set $\{x \in \hh : W_{x} \subseteq W'\} = \hh \setminus \bigcup_{s \not \in W'} \Gamma_{s}$. Note that $\underline{\cL}$ is closed in $\hh^{reg-W'}/W'$. We consider the completion $\bH_{reg-W'}^{\wedge\underline{\cL}} := \CC[\hh^{reg-W'}/W']^{\wedge\underline{\cL}} \otimes_{\CC[\hh/W]} \bH$, which is naturally an algebra. For a HC $\bH$-bimodule $\bB$, let $\bB_{reg-W'}^{\wedge	\underline{\cL}} :=  \CC[\hh^{reg-W'}/W']^{\wedge\underline{\cL}} \otimes_{\CC[\hh/W]} \bB$. This is a $\Xi$-equivariant $\bH_{reg-W'}^{\wedge\underline{\cL}}$-bimodule where, recall, we denote $\Xi = N_W(W')/W'$. The latter algebra is isomorphic to the algebra of $|W/W'|\times|W/W'|$-matrices over $\bH(W', \hh)^{\wedge\underline{\cL}}_{reg-W'}$, so we have a Morita equivalence between $\bH_{reg-W'}^{\wedge \underline{\cL}}$ and $\bH(W',\hh)^{\wedge\underline{\cL}}$. Abusing the notation, let us denote also by $\bB_{reg-W'}^{\wedge\underline{\cL}}$ the corresponding $\bH(W', \hh)^{\wedge\underline{\cL}}_{reg-W'}$-bimodule. Now let $\bB_{\diamondsuit}$ be the subspace of $\bB_{reg-W'}^{\wedge\underline{\cL}}$ consisting of elements that commute with $\hh^{W'}, (\hh^*)^{W'}$ and for which the action of $S(\hh_{W'})^{W'}, S(\hh^{*}_{W'})^{W'}$ is locally nilpotent. Then, according to \cite[Lemma 3.7]{losev_derived}, we have an isomorphism of functors $\bullet_{\dagger} \cong \bullet_{\diamondsuit}$. A more precise statement is as follows. Consider the functor $\cG: \HC^{\Xi}(\bH(W', \hh_{W'})) \longrightarrow \HC^{\Xi}(\bH(W', \hh)^{\wedge\underline{\cL}}_{reg-W'})$ given by $\cG(\bB') = \CC[\underline{\cL} \times \hh_{W'}/W']^{\wedge\underline{\cL}}\otimes_{\CC[\underline{\cL} \times \hh_{W'}/W']} (\cD_{\hbar}(\underline{\cL}) \otimes \bB')$. Then, $\cG$ is a fully faithful embedding and $\bullet_{\dagger}$ coincides with $\cG^{-1}(\bullet_{reg-W'}^{\wedge\underline{\cL}})$ (the statement here includes that this latter functor is well-defined).

To construct $\bullet_{\dagger}$ for rational Cherednik algebras $H_{c}$, $H_{c'}$ we use the Rees construction. Namely, let $B$ be a HC $H_{c}\text{-}H_{c'}$-bimodule, and pick a bimodule filtration on $B$ in a way that $\R_{\hbar}(B)$ is a HC $\bH$-bimodule. Then, $B_{\dagger} := (R_{\hbar}(B)_{\dagger})/(\hbar - 1)$ is a HC $\underline{H}_{c}\text{-}\underline{H}_{c'}$-bimodule. The construction of $B_{\dagger}$ does not depend, up to a distinguished isomorphism, on the choice of a filtration on $B$, see e.g. \cite[Subsection 3.9]{losev_completions}.

Finally, let us state a compatibility result between restriction functors for category $\cO$ and restriction functors for HC bimodules. Recall that for a HC $H_{c}\text{-}H_{c'}$-bimodule $B$ and $N \in \cO_{c'}$, we have that $B \otimes_{H_{c'}} N \in \cO_{c}$. Then, \cite[Subsection 3.3]{losev_derived}, we have

\begin{lemma}\label{lemma:compatibility of res}
There is a natural isomorphism $\Res^{W}_{W'}(B \otimes_{H_{c'}} N) \buildrel \cong \over \longrightarrow B_{\dagger, W'} \otimes_{\underline{H}_{c'}} \Res^{W}_{W'}(N)$.
\end{lemma}

\section{Reduction to corank 1}\label{sect_reduction}

Let $\underline{W}$ be a parabolic subgroup of $W$. Associated to it, we have the symplectic leaf $\cL := \pi(\{(x, y) \in \hh \oplus \hh^{*} : W_{(x,y)} = \underline{W}\}) \subseteq (\hh \oplus \hh^{*})/W$, where $\pi: \hh \oplus \hh^{*} \rightarrow (\hh \oplus \hh^{*})/W$ is the quotient by the $W$-action, see \cite[7.4]{BrGo}. Let us denote by $\underline{H}_{c}$ the rational Cherednik algebra for the group $\underline{W}$ acting on $\hh_{\underline{W}}$ where, recall, $\hh = \hh^{\underline{W}} \oplus \hh_{\underline{W}}$ is a unique $\underline{W}$-invariant decomposition. We have the restriction functor

\[
\bullet_{\dagger}: \HC_{\cL}(H_{c}, H_{c'}) \rightarrow \HC_{0}^{\Xi}(\underline{H}_{c}, \underline{H}_{c'})
\]

\noindent where, as before, $\Xi = N_{W}(\underline{W})/\underline{W}$. A natural question here is to describe the image of this functor or, equivalently, of the functor $\bullet_{\dagger}: \HC_{\overline{\cL}}(H_{c}, H_{c'}) \rightarrow \HC_{0}^{\Xi}(\underline{H}_{c}, \underline{H}_{c'})$. Here, we reduce the study of this question to the following situation.

\begin{definition}
	Let $\underline{W} \subseteq W' \subseteq W$ be parabolic subgroups. In particular, $\hh^{W'} \subseteq \hh^{\underline{W}}$. We say that $\underline{W}$ sits inside $W'$ \emph{in corank $1$} if $\codim_{\hh^{\underline{W}}}(\hh^{W'}) = 1$. 
\end{definition} 

More precisely, we have the following result.

\begin{theorem}\label{thm:veryverytechnical}
	Let $B \in \HC^{\Xi}_{0}(\underline{H}_{c}, \underline{H}_{c'})$. Assume that, for every parabolic subgroup $W' \subseteq W$ containing $\underline{W}$ in corank 1, there exists a HC $H'_{c}\text{-}H'_{c'}$-bimodule $B'$ such that $(B')_{\dagger^{W'}_{\underline{W}}} = B$. Here, $H'_{c} := H_{c}(W', \hh_{W'})$ and the $N_{W'}(\underline{W})/\underline{W}$-equivariant structure on $B$ is restricted from the $\Xi$-equivariant structure. Then, there exists a HC $H_{c}\text{-}H_{c'}$-bimodule $\overline{B}$ such that $\overline{B}_{\dagger^{W}_{\underline{W}}} = B$.
\end{theorem}

The proof of Theorem \ref{thm:veryverytechnical} passes through its homogeneous version, which is not surprising given the construction of the funtor $\bullet_{\dagger}$. In Subsection \ref{subsect:localization} we give basic facts concerning the localization of HC bimodules to open sets in $\hh/W$ and, more generally, on \'etale lifts of HC bimdodules. Subsection \ref{subsect:BE} recalls the isomorphisms of completions of \cite{BE} and, more generally, the isomorphisms of \'etale lifts that appeared in \cite{losev_derived}. The next two subsections, Subsection \ref{subsect:supports} and Subsection \ref{subsect:annihilators} are technical, in them we study supports and annihilators of HC bimodules over the localized Cherednik algebras. In Subsection \ref{subsect:main homogeneous} we state and prove an homogeneous variant of Theorem \ref{thm:veryverytechnical}. Finally, in Subsection \ref{subsect:specpar} we use the Rees construction to prove Theorem \ref{thm:veryverytechnical}.

\subsection{Localization}\label{subsect:localization} Let $f \in \CC[\hh]^{W}$. Since the adjoint action of $f$ on $\bH$ is locally nilpotent, the localization $\bH[f^{-1}]$ makes sense as an algebra. We remark that $\bH[f^{-1}]/\fc\bH[f^{-1}] = \CC[\pi^{-1}(U) \times \hh^{*}]\#W$. Here, $U$ denotes the principal open set in $\hh/W$ determined by $f$, and $\pi: \hh \longrightarrow \hh/W$ is the natural projection. Note that the algebra $\bH[f^{-1}]$ is graded, this follows because $f \in \bH$ is in degree $0$. Also, note that $\bH[f^{-1}] = \CC[U] \otimes_{\CC[\hh/W]}\bH$. 

\begin{lemma}\label{lemma_loc}
Let $\bB$ be a HC $\bH$-bimodule, and $f \in \CC[\hh]^{W}$. Then, all localizations $\CC[U] \otimes_{\CC[\hh/W]} \bB$, $\bB \otimes_{\CC[\hh/W]} \CC[U]$, $\CC[U] \otimes_{\CC[\hh/W]} \bB \otimes_{\CC[\hh/W]} \CC[U]$ coincide.
\end{lemma} 
\begin{proof}
Recall that $\bB$ is graded and, since it is finitely generated, the grading is bounded below. Now, since the adjoint action of $f$ on $\bB/\hbar\bB$ is $0$, we have that $fv - vf \in \hbar\bB$ for every $v \in \bB$. So we can define the operator $\frac{1}{\hbar}[f, \cdot]$ because $\bB$ is $\CC[\hbar]$-flat. Since $f$ has degree $0$, this operator has degree $-1$. So the operator $\frac{1}{\hbar}[f, \cdot]$, and hence $[f, \cdot]$, is locally nilpotent. The result follows. 
\end{proof}

For a HC $\bH$-bimodule, $\bB$, we will denote by $\bB[f^{-1}]$ any of the localizations of Lemma \ref{lemma_loc}. Note that we can define the notion of a HC bimodule over $\bH[f^{-1}]$ similarly to Definition \ref{def_HCbimod}. Clearly, $\bB[f^{-1}]$ is a HC $\bH[f^{-1}]$-bimodule. 

We remark that, more generally, for a smooth affine algebraic variety $U$ with an \'etale map $U \longrightarrow \hh/W$, the space $\CC[U] \otimes_{\CC[\hh/W]} \bH$ is actually an algebra. Indeed, it can be identified with the $S(\fc)$-subalgebra of $\cD_{\hbar}(U \times_{\hh/W} \hh^{reg})[\bc_1, \dots, \bc_n]\#W$ generated by $\CC[U \times_{\hh/W}\hh]$, $\CC W$, and the Dunkl-Opdam operators. We will denote this algebra by $\bH_{U}$. Note that $\bH_{U}$ is graded, with $\CC[U \times_{\hh/W}\hh]$ in degree $0$. We can define the notion of a HC $\bH_{U}$-bimodule as before. Note that we still have a notion of the support of a HC $\bH_U$-bimodule $\bB$: this is the support of $\bB/\fc\bB$ as a $Z(\bH_{U}/\fc\bH_{U})$-module. If, moreover, $U \longrightarrow \hh/W$ is an inclusion then, similarly to Lemma \ref{lemma_loc}, for a HC $\bH$-bimodule $\bB$, all localizations $\CC[U] \otimes_{\CC[\hh/W]} \bB$, $\bB \otimes_{\CC[\hh/W]} \CC[U]$ and $\CC[U] \otimes_{\CC[\hh/W]} \bB \otimes_{\CC[\hh/W]} \CC[U]$ coincide. We will denote this by $\bB_U$. This is a HC $\bH_U$-bimodule.

\subsection{Bezrukavnikov-Etingof isomorphisms}\label{subsect:BE} We will need some isomorphisms of \'etale lifts, that are essentially due to Bezrukavnikov and Etingof, \cite[Theorem 3.2]{BE}, see also \cite[Lemma 2.1]{losev_hecke}, \cite[Proposition 2.6]{wilcox}.  

Let $W' \subseteq W$ be a parabolic subgroup. Recall that we have defined $\hh^{reg-W'} := \{ b \in \hh : W_{b} \subseteq W'\} = \hh \setminus \bigcup_{s \not\in W'}\ker\alpha_{s}$. In particular, $\hh^{reg-W'}$ is a principal open set in $\hh$. We remark that the unramified locus of the map $\hh/W' \rightarrow \hh/W$ coincides with $\hh^{reg-W'}/W'$, so we have an \'etale morphism $\hh^{reg-W'}/W' \longrightarrow \hh/W$ and the algebra $\bH_{reg-W'} := \bH_{\hh^{reg-W'}/W'}$ makes sense. 

On the other hand, let $\underline{\bH}$ be the rational Cherednik algebra for the action of $W'$ on $\hh$. We remark that we take this as an algebra over $S(\fc)$, even if the defining relations do not involve all the variables $\bc_1, \dots, \bc_r$. We have a decomposition $\hh = \hh^{W'} \oplus \hh_{W'}$. Recall that here, $\hh^{W'}$ denotes the subspace of $W'$-invariants, and $\hh_{W'}$ a unique $W'$-stable complement to $\hh^{W'}$ in $\hh$. This induces a decomposition $\underline{\bH} = \underline{\bH}^{+} \otimes_{\CC[\hbar]}\cD_{\hbar}(\hh^{W})$, where $\underline{\bH}^{+}$ is the rational Cherednik algebra for the action of $W'$ on $\hh_{W'}$ (again, we include all variables $\bc_1, \dots, \bc_r)$. 

We can form the centralizer algebra $Z(W, W', \underline{\bH}_{reg-W'})$. Recall that, by definition, for an algebra $A$ and a monomorphism $\CC W' \hookrightarrow A$, we can form the right $A$-module $\Fun_{W'}(W, A) := \{f: W \rightarrow A : f(w'w) = w'f(w),\; \text{for all} \; w' \in W'\}$. This is a free right $A$-module of rank $|W/W'|$. Define $Z(W, W', A) := \End_{A}(\Fun_{W'}(W,A))$. We remark that $Z(W,W',A) \cong \Mat_{|W/W'|}(A)$, but this isomorphism is not canonical. There is, however, a canonical way to recover $A$ from $Z(W,W',A)$. Namely, consider the element $e(W') \in Z(W,W',A)$ defined by

$$
[e(W')f](w) = \begin{cases} f(w) & \text{if} \; w \in W' \\ 0 & \text{else}. \end{cases}
$$

\noindent Then, $e(W')Z(W,W',A)e(W')$ is naturally identified with $A$.

\begin{lemma}[\cite{BE}, Theorem 3.2]\label{lemma_be}
There is an isomorphism

$$
\Theta: \bH_{reg-W'} \longrightarrow Z(W, W', \underline{\bH}_{reg-W'}).
$$
\end{lemma}

\noindent We remark that \cite[Theorem 3.2]{BE} works with completions rather than \'etale lifts. However, $\hh^{reg-W'}$ is defined as the complement of the reflection hyperplanes for reflections not in $W'$, so the proof of \cite[Theorem 3.2]{BE} works in the setting of Lemma \ref{lemma_be}. Alternatively, the existence of the isomorphism $\Theta$ can be seen from the description of the variety $\hh^{reg-W'}/W' \times_{\hh/W} \hh$ given in Subsection \ref{subsect:supports}. Moreover, the isomorphism in Lemma \ref{lemma_be} can be further restricted as follows. Recall that we have set $\underline{\cL} := \{x \in \hh : W_{x} = W'\}$, which is a closed subvariety inside of $\hh^{reg-W'}/W'$. We can form the completion along the closed subvariety $\underline{\cL} \subseteq \hh^{reg-W'}/W'$, and we define the algebras $\bH_{reg-W'}^{\wedge \underline{\cL}} := \CC[\hh^{reg-W'}/W']^{\wedge \underline{\cL}} \otimes_{\CC[\hh/W]} \bH$, $\underline{\bH}_{reg-W'}^{\wedge \underline{\cL}} := \CC[\hh^{reg-W'}/W']^{\wedge \underline{\cL}} \otimes_{\CC[\hh/W]} \underline{\bH}$. The isomorphism $\Theta$ in Lemma \ref{lemma_be} can be restricted to an isomorphism

$$
\Theta: \bH_{reg-W'}^{\wedge \underline{\cL}} \buildrel \cong \over \longrightarrow Z(W, W', \underline{\bH}_{reg-W'}^{\wedge \underline{\cL}}).
$$

Now let $\widehat{\underline{\cL}}$ denote the formal neighborhood of $\underline{\cL}$ in $\hh^{reg-W'}/W'$, and $\widehat{\underline{\cL}}^{\times} := \widehat{\underline{\cL}} \setminus \underline{\cL}$. The algebra of functions on $\widehat{\underline{\cL}}^{\times}$ is a localization of $\CC[\hh^{reg-W'}/W']^{\wedge\underline{\cL}}$, and we can form the algebras $\bH_{\widehat{\underline{\cL}}^{\times}}$, $\underline{\bH}_{\widehat{\underline{\cL}}^{\times}}$ as localizations of the algebras $\bH_{reg-W'}^{\wedge\underline{\cL}}, \underline{\bH}_{reg-W'}^{\wedge\underline{\cL}}$, respectively. The isomorphism $\Theta$ can be further restricted to

$$
\Theta_{W'}: \bH_{\widehat{\underline{\cL}}^{\times}} \buildrel \cong \over \longrightarrow Z(W, W', \underline{\bH}_{\widehat{\underline{\cL}}^{\times}}).
$$
 
\subsection{Supports and symplectic leaves}\label{subsect:supports} Let us remark that there is a $W$-equivariant isomorphism 

\begin{equation}\label{eqn:iso_varieties} \hh^{reg-W'}/W' \times_{\hh/W} \hh \buildrel \cong \over \longrightarrow \bigsqcup_{w \in W/W'} w\hh^{reg-W'} \subseteq W/W' \times \hh
\end{equation}

\noindent that, for $x \in \hh^{reg-W'}$ and $w \in W$, sends $(W'x, wx)$ to $(wW', wx) \in W/W' \times \hh$. Let us denote by $\cX$ the variety $\bigsqcup_{w \in W/W'} w\hh^{reg-W'}$. So we can think of $\bH_{reg-W'}$ as $\bH(W, \cX)$, the rational Cherednik algebra associated to the action of $W$ on the variety $\cX$, see e.g. \cite[Section 2]{wilcox} for generalities on rational Cherednik algebras associated to the action of a complex reflection group on a smooth algebraic variety (not necessarily a vector space), in this paper we will only work with varieties that are disjoint unions of Zariski open sets inside a vector space. Note that $\cX \times \hh^{*} = T^{*}\cX$ is a symplectic algebraic variety, and the action of $W$ on $T^{*}\cX$ is by symplectomorphisms. So $(T^{*}\cX)/W$ is a Poisson variety. Moreover, it follows from the isomorphism (\ref{eqn:iso_varieties}) that $\cX/W = \hh^{reg-W'}/W'$ and $(T^{*}\cX)/W = (T^{*}\hh^{reg-W'})/W'$. As before, for a HC $\bH_{reg-W'}$-bimodule $\bB$, its support $\Supp(\bB) \subseteq (T^{*}\hh^{reg-W'})/W'$ is a union of symplectic leaves. 

We can describe the symplectic leaves inside $(T^{*}\hh^{reg-W'})/W'$ using the results in \cite[7.4]{BrGo}. We remark that \cite{BrGo} works with actions on a vector space, but the proofs work in our setting. Namely, let $W'' \subseteq W'$ be a parabolic subgroup. Let $\cL_{W''}^{W'} := \pi_{\hh^{reg-W'}}(\{x \in T^{*}\hh^{reg-W'} : W'_{x} = W''\})$, where $\pi_{\hh^{reg-W'}}: T^{*}\hh^{reg-W'} \rightarrow (T^{*}\hh^{reg-W'})/W'$ is the quotient by the $W'$-action. Note that $\cL^{W'}_{W''}$ depends only on the conjugacy class of $W''$ in $W'$. The symplectic leaves in $(T^{*}\hh^{reg-W'})/W'$ are precisely the $\cL^{W'}_{W''}$, where $W'' \subseteq W'$ is a parabolic subgroup. Note that $\overline{\cL_{W''}^{W'}} = \bigsqcup_{W'' \subseteq W''' \subseteq W'} \cL_{W'''}^{W'}$. It follows that the unique closed leaf inside $(T^{*}\hh^{reg-W'})/W'$ is $\cL^{W'}_{W'}$. We remark that $\cL^{W'}_{W'} = \hh^{W'}_{reg} \times (\hh^{*})^{W'}$, where $\hh^{W'}_{reg} := \{x \in \hh : W_{x} = W'\}$.  

Similarly, we can think of $\underline{\bH}_{reg-W'}$ as being the rational Cherednik algebra $\bH(W', \hh^{reg-W'})$ (recall, we are taking all variables $\hbar, \bc_1, \dots, \bc_r$ even if the defining relations do not involve all of them). The support of a HC $\underline{\bH}_{reg-W'}$-bimodule $\bB$ is again a union of symplectic leaves of $(T^{*}\hh^{reg-W'})/W'$. 
Let us describe a relation between supports of HC $\bH_{reg-W'}$ and $\underline{\bH}_{reg-W'}$-bimodules. Note that $\CC[\cX] = \bigoplus_{w \in W/W'} \CC[w\hh^{reg-W'}]$. We may think of the idempotent $e(W') \in Z(W, W', \underline{\bH}_{reg-W'}) \cong \bH_{reg-W'}$ introduced in Subsection \ref{subsect:BE} as being the primitive idempotent in $\CC[\cX] \subseteq \bH_{reg-W'}$ corresponding to the direct summand $\CC[\hh^{reg-W'}]$. It follows that for a HC $\bH_{reg-W'}$-bimodule $\bB$, $\Supp(\bB) = \Supp(e(W')\bB e(W'))$. \\
The description of the singular support of a HC bimodule has the following consequence for the support of the elements of $\bB/\fc\bB$. 

\begin{lemma}\label{lemma:support_elements}
Let $\bB$ be a HC $\bH_{reg-W'}$-bimodule. Consider $\bB/\fc\bB$ as a $\CC[\cX/W] \otimes \CC[\hh^{*}/W']$-module where, recall, $\cX = \bigsqcup_{w \in W/W'}w\hh^{reg-W'}$, and $\cX/W = \hh^{reg-W'}/W'$.  Then, for every nonzero $v \in \bB/\fc\bB$, its support $X_{v} \subseteq \hh^{reg-W'}/W' \times \hh^{*}/W'$ contains $\hh^{W'}_{reg} \times (\hh^{*})^{W'}$.
\end{lemma}

In view of the description of the symplectic leaves inside $(\cX \times \hh^{*})/W$, Lemma \ref{lemma:support_elements} is a consequence of the following result.

\begin{lemma}
Let $A$ be a commutative, Noetherian Poisson algebra, and let $M$ be a finitely generated Poisson $A$-module. Then, for every element $m \in M$, its set-theoretic support $X_{m} \subseteq \Spec(A)$ is a Poisson subvariety.
\end{lemma}
\begin{proof}
First of all, let $I \subseteq A$ be any ideal. For $k \geq 0$, let $M_{I^{k}} := \{ n \in M : I^{k}n = 0\}$. Note that $M_{I^{k}} \subseteq M_{I^{k+1}}$, so that $M(I) := \bigcup_{k \geq 0} M_{I^{k}}$ is a submodule of $M$. We claim that it is a Poisson submodule. Take $m \in M_{I^{k}}$ and $a \in A$. Let $a_{1}, \dots, a_{2k} \in I$, so that $a_{1}\cdots a_{2k}m = 0$. It follows that $0 = \{a, a_{1}\dots a_{2k}m\} = a_{1}\cdots a_{2k}\{a,m\} + \{a, a_{1}\cdots a_{2k}\}m$. Thanks to the Leibniz identity again, $\{a, a_{1}\cdots a_{2k}\} = a_{1}\cdots a_{k}\{a, a_{k+1}\cdots a_{2k}\} + a_{k+1}\cdots a_{2k}\{a, a_{1}\cdots a_{k}\} \in I^{k}$. Since $m \in M_{I^{k}}$, this implies that $a_{1}\cdots a_{2k}\{a, m\} = 0$. Thus, $\{A, m\}\subseteq M_{I^{2k}}$. So $M(I)$ is a Poisson submodule and thus $\supp(M(I)) \subseteq \Spec(A)$ is a Poisson subvariety. \\
Now specialize to the case where $I = \ann_{A}(m)$. Since $A$ is Noetherian and $M$ is finitely generated, $M(I) = M_{I^{k}}$ for some $k > 0$. Then, $I^{k} \subseteq \ann(M(I))$. On the other hand, since $m \in M(I)$ and $I = \ann_{A}(m)$, $\ann(M(I)) \subseteq I$. So $\sqrt{\ann(M(I))} = \sqrt{I}$ and the result follows.
\end{proof}

Let us remark that, thanks to the correspondence between supports of HC $\bH_{reg-W'}$- and $\underline{\bH}_{reg-W'}$-bimodules, we get from Lemma \ref{lemma:support_elements} the following result.

\begin{cor}\label{cor:support_important}
Let $\bB$ be a HC $\underline{\bH}_{reg-W'}$-bimodule. Consider $\bB/\fc\bB$ as a $\CC[\hh^{reg-W'}/W'] \otimes \CC[\hh^{*}/W']$-module. Then, for every nonzero $v \in \bB/\fc\bB$, its support $X_{v} \subseteq \hh^{reg-W'}/W' \times \hh^{*}/W$ contains $\hh^{W'}_{reg} \times (\hh^{*})^{W'}$.
\end{cor}

\subsection{Annihilators and liftings}\label{subsect:annihilators} We will describe the annihilator of a HC $\bH_{reg-W'}$-bimodule \emph{as a left $\CC[\hh^{reg-W'}/W']$-module}. In order to do so, we need the following finiteness result.

\begin{lemma}\label{lemma_fg}
Let $\bB$ be a HC $\bH_{reg-W'}$-bimodule. Then, $\bB$ is finitely generated over $\CC[\hh^{reg-W'}/W'][\fc] \otimes_{S(\fc)} \CC[\hh^{*}/W][\fc]^{\opp}$. Similarly, $\bB$ is finitely generated over $\CC[\hh^{*}/W][\fc] \otimes_{S(\fc)} \CC[\hh^{reg-W'}/W'][\fc]^{\opp}$, where the superscript $\opp$ means that the corresponding algebra acts on the right.
\end{lemma}
\begin{proof}
Since $\bB$ is HC, we have that $\bB/\fc\bB$ is a module over $\CC[\cX \times \hh^{*}]^{W}$, which is the center of the algebra $\bH_{reg-W'}/\fc\bH_{reg-W'}$. This latter algebra is finite over its center, so $\bB/\fc\bB$ is a finitely generated module over $\CC[\cX \times \hh^{*}]^{W}$. Now, the natural map $(\cX \times \hh^{*})/W \longrightarrow \cX/W \times \hh^{*}/W$ is finite, so $\bB/\fc\bB$ is a finitely generated module over $\CC[\cX]^{W} \otimes \CC[\hh^{*}]^{W}$. Let $\underline{v}_1, \dots, \underline{v}_m$ be generators of $\bB/\fc\bB$ under the action of $\CC[\cX]^{W} \otimes \CC[\hh^{*}]^{W}$. We can assume that these elements are homogeneous with respect to the grading on $\bB/\fc\bB$ inherited from the one on $\bB$. Let $v_1, \dots, v_m$ be homogeneous lifts of $\underline{v}_1, \dots, \underline{v}_m$. It is now standard to see that $v_1, \dots, v_m$ are generators of $\bB$ under the action of $\CC[\cX]^{W}[\fc] \otimes_{S(\fc)} \CC[\hh^{*}]^{W}[\fc]^{\opp}$.
\end{proof}

\begin{lemma}\label{lemma:ann}
Let $\bB$ be a HC $\bH_{reg-W'}$-bimodule. Assume that $\Supp(\bB) = \cL^{W'}_{W'}$. Then, as a (left or right) $\CC[\hh^{reg-W'}/W']$-module, $\bB$ is annihilated by a power of the ideal $I$ of functions vanishing on $\underline{\cL} \subseteq \hh^{reg-W'}/W'$ where, recall, $\underline{\cL} = \{x \in \hh : W_{x} = W'\}$. 
\end{lemma}
\begin{proof}
First, we show that any element in $\bB$ is annihilated by a large enough power of $I$. Recall that $\cL^{W'}_{W'} = \hh^{W'}_{reg} \times (\hh^{*})^{W'}$. In particular, $\hh^{W'}_{reg} \times 0 \subseteq \cL^{W'}_{W'}$. It follows by our assumption on $\Supp(\bB)$ that $I^{n} \subseteq \CC[\hh^{reg-W'}/W'] \subseteq \CC[\hh^{reg-W'}/W'] \otimes \CC[\hh^{*}/W] \subseteq \CC[\cX \times \hh^{*}]^{W}$ annihilates $\bB/\fc\bB$ for $n \gg 0$. So for any $i \in \ZZ$, $I^{n}\bB^{i} \subseteq \fc\bB^{i-1}$. Now the claim follows because the grading on $\bB$ is bounded below. \\
Now let $v_1, \dots, v_m$ be generators of $\bB$ as a $\CC[\hh^{reg-W'}/W'][\fc] \otimes_{S(\fc)} \CC[\hh^{*}/W][\fc]^{\opp}$-module, and let $N \gg 0$ be such that $I^{N}v_{i} = 0$ for all $i = 1, \dots, m$. It is easy to see that $I^{N}\bB = 0$. We are done.
\end{proof}

Let us discuss some consequences of Lemma \ref{lemma:ann}. Recall that we have a natural action of the group $N_{W'}(W)$ on the algebra $\bH_{reg-W'}$ by algebra automorphisms, in such a way that the action of $W' \subseteq N_{W'}(W)$ concides with the adjoint action. Recall also that we denote $\Xi = N_{W}(W')/W'$. The map $\eta_{W'}: \hh^{reg-W'}/W' \rightarrow \hh/W$ is \'etale and it restricts to a covering $\eta_{W'}: \underline{\cL} \rightarrow \eta_{W'}(\underline{\cL}) = \underline{\cL}/\Xi$ with Galois group $\Xi$. This implies that the formal neighborhood $(\hh/W)^{\wedge \eta_{W'}(\underline{\cL})}$ may be identified with the quotient by the action of $\Xi$ on the formal neighborhood $(\hh^{reg-W'}/W')^{\wedge\underline{\cL}}$. Now let $\bB$ be a $\Xi$-equivariant HC $\bH_{reg-W'}$-bimodule supported on $\cL^{W'}_{W'}$. Thanks to Lemma \ref{lemma:ann}, $\bB$ may be thought of as a quasi-coherent sheaf on an infinitesimal neighborhood of $\underline{\cL} \subseteq \hh^{reg-W'}/W'$. Thus, the space of invariants $\bB^{\Xi}$ is a quasi-coherent sheaf on an infinitesimal neighborhood of $\eta_{W'}(\underline{\cL}) \subseteq \hh/W$, and we may think of it as a quasi-coherent sheaf on $W\hh^{reg-W'}/W$. 

We claim that, moreover, $\bB = \CC[\hh^{reg-W'}/W'] \otimes_{\CC[\hh/W]} \bB^{\Xi}$. Since $\eta_{W'}$ is, in particular, \'etale when restricted to $\underline{\cL}$, the induced map $(\hh^{reg-W'}/W')^{\wedge\underline{\cL}} \rightarrow (W\hh^{reg-W'}/W)^{\wedge \eta_{W'}(\underline{\cL})}$ is the quotient by the free $\Xi$-action on the formal neighborhood of $\underline{\cL}$, and $\CC[W\hh^{reg-W'}/W]^{\wedge \eta_{W'}(\underline{\cL})}$ may be identified with the algebra of $\Xi$ invariants in $\CC[\hh^{reg-W'}/W']^{\wedge \underline{\cL}}$. So the desired equality will follow if we show that the right-hand side is equal to $\CC[\hh^{reg-W'}/W']^{\wedge \underline{\cL}} \otimes_{\CC[W\hh^{reg-W'}/W]^{\wedge \eta_{W'}(\underline{\cL})}} \bB^{\Xi}$. But this is clear by our description of the annihilator of $\bB$ (and of $\bB^{\Xi})$. 

\subsection{Main result}\label{subsect:main homogeneous} We are now ready to state our main result. Let $\underline{W}$ be a parabolic subgroup of $W$, and let $\bB$ be a $\Xi = N_{W}(\underline{W})/\underline{W}$-equivariant HC $\bH_{reg-\underline{W}}$-bimodule. We require that $\Supp(\bB) = \cL^{\underline{W}}_{\underline{W}}$, the minimal symplectic leaf in $(T^{*}\hh^{reg-\underline{W}})/\underline{W}$. Recall that we denote $\eta_{\underline{W}}: \hh^{reg-\underline{W}}/\underline{W} \rightarrow \hh/W$. From the previous subsection it follows that $\eta_{\underline{W}*}(\bB^{\Xi})$ is an $\bH$-bimodule satisfying $(\eta_{\underline{W}*}(\bB^{\Xi}))_{reg-\underline{W}} = \bB$. However, $\eta_{\underline{W}*}(\bB^{\Xi})$ need not be finitely generated over $\bH$ so it is not, in general, a HC $\bH$-bimodule. Similarly, if $W'$ is a parabolic subgroup of $W$ containing $\underline{W}$, $(\eta_{\underline{W}*}(\bB^{\Xi}))_{reg-W'}$ does not need to be a HC $\bH_{reg-W'}$-bimodule. However, we can further localize to the punctured formal neighborhood $\underline{\widehat{\cL}}_{W'}^{\times}$. The bimodule $(\eta_{\underline{W}*}(\bB^{\Xi}))_{\underline{\widehat{\cL}}_{W'}^{\times}}$ is now a HC $\bH_{\underline{\widehat{\cL}}_{W'}^{\times}}$-bimodule. 

\begin{theorem}\label{thm_main}
Let $\bB$ be a $\Xi$-equivariant HC $\bH_{reg-\underline{W}}$-bimodule. Assume that $\Supp(\bB) = \cL_{\underline{W}}^{\underline{W}}$ and that for all parabolic subgroups $W'$ with $\underline{W} \subseteq W'$ in corank 1, there is a HC $\bH_{\underline{\widehat{\cL}}_{W'}}$-bimodule $\bB_{W'}$ whose localization to $\widehat{\underline{\cL}}^{\times}_{W'}$ concides with $(\eta_{\underline{W}*}(\bB^{\Xi}))_{\widehat{\underline{\cL}}_{W'}^{\times}}$. Then, there exists a HC $\bH$-bimodule $\overline{\bB}$ such that $\bB = \overline{\bB}_{reg-\underline{W}}$. 
\end{theorem}

The proof of Theorem \ref{thm_main} is inspired by \cite[Section 3]{losev_hecke}, where a similar result is shown at the level of category $\cO$ (for the stratum corresponding to the dense symplectic leaf.) The strategy is as follows. We will define a bimodule that is coherent over an open subset in $\hh/W$ whose complement has codimension 2. Then, we can take global sections. This open subset will be the image in $\hh/W$ of

$$\hh^{sr-\underline{W}}:= \bigcup_{\substack{\underline{W} \subseteq W' \\ \text{in corank} \: 1 }} \hh^{reg-W'}.$$

 It is clear that the complement of $\hh^{sr-\underline{W}}$ in $\hh$ has codimension at least
  2. Moreover, $\hh^{\underline{W}}\cap\hh^{sr-\underline{W}}$ is an open subset of $\hh^{\underline{W}}$ whose complement has codimension at least 2. Indeed, we have
 
\begin{equation}\label{eqn:codim2}
\hh^{\underline{W}}\cap\hh^{sr-\underline{W}} = \hh^{\underline{W}} \setminus \bigcup_{\substack{s, s' \not \in \underline{W} \\ \Gamma_{s}\cap\hh^{\underline{W}} \neq \Gamma_{s'}\cap\hh^{\underline{W}}}} \Gamma_{s}\cap\Gamma_{s'}.
\end{equation}

The way to get a desired bimodule is as follows. First, for each parabolic subgroup $W'$ containing $\underline{W}$ in corank 1, we will construct a HC $\bH_{reg-W'}$-bimodule with the property that its lift to $\hh^{reg-\underline{W}}/\underline{W}$ coincides with $\bB$. Then we will get our bimodule by, roughly speaking, glueing the bimodules defined over $\hh^{reg-W'}/W'$. \\

\begin{proofmain}
{\it Part 1: Constructing HC $\bH_{reg-W'}$-bimodules}. For each parabolic subgroup $W'$ containing $\underline{W}$ in corank 1, let $\bB_{W'}$ be a HC $\bH_{\widehat{\underline{\cL}}_{W'}}$-bimodule that localizes to $(\eta_{\underline{W}*}(\bB^{\Xi}))_{\widehat{\underline{\cL}}^{\times}_{W'}}$, the existence of such a bimodule is guaranteed by the assumptions of the theorem. Note that we may assume that $\bB_{W'} \subseteq (\eta_{\underline{W}*}(\bB^{\Xi}))_{\widehat{\underline{\cL}}_{W'}^{\times}}$, if this is not the case we can just replace $\bB_{W'}$ by its quotient by the maximal sub-bimodule that is killed by the localization, we can find such a sub-bimodule because $\bB_{W'}$ is a finitely generated bimodule over the noetherian algebra $\bH_{\widehat{\underline{\cL}}_{W'}}$.
 
 On the other hand, let $\eta_{W'}: \hh^{reg-W'}/W' \rightarrow \hh/W$ be the natural projection, and consider $\eta_{W'}^{*}(\eta_{\underline{W}*}(\bB^{\Xi})) = \CC[\hh^{reg-W'}/W'] \otimes_{\CC[\hh/W]} \bB^{\Xi}$. The inclusion $\CC[\hh^{reg-W'}/W'] \hookrightarrow \CC[\widehat{\underline{\cL}}_{W'}^{\times}]$ induces a map $\eta_{W'}^{*}\eta_{\underline{W}*}(\bB^{\Xi}) \rightarrow \eta_{\underline{W}*}(\bB^{\Xi})_{\widehat{\underline{\cL}}^{\times}_{W'}}$ that we claim to be injective. Indeed, this follows because inside $(\hh\oplus\hh^{*})/W$ we have $\cL_{W'} \subseteq \overline{\cL_{\underline{W}}}$ and the singular support of every finitely generated $\bH$-sub-bimodule of $\eta_{\underline{W}*}(\bB^{\Xi})$ (which is the union of its HC sub-bimodules) contains $\cL_{\underline{W}}$. The claim is now a consequence of the fact that $\cL^{\underline{W}}_{\underline{W}}$ is the minimal symplectic leaf inside $T^{*}\cX_{\underline{W}}/W$, cf. Subsection \ref{subsect:supports}.  
 
 Define $\widetilde{\bB}_{W'} := (\eta_{W'}^{*}(\eta_{\underline{W}*}(\bB^{\Xi})))\cap \bB_{W'} \subseteq \bB^{\Xi}_{\widehat{\underline{\cL}}^{\times}_{W'}}$. Note that this is an $\bH_{reg-W'}$-bimodule. Let us see that it is finitely generated. By a suitable straightforward adaptation of Lemma \ref{lemma_fg}, $(\eta_{\underline{W}*}(\bB^{\Xi}))_{\widehat{\underline{\cL}}_{W'}^{\times}}$ is finitely generated over the algebra $\CC[\widehat{\underline{\cL}}_{W'}^{\times}][\fc]\otimes_{S(\fc)} \CC[\hh^{*}]^{W}[\fc]^{\opp}$. Note that $\bB_{W'}$ is a $\CC[\widehat{\underline{\cL}}_{W'}][\fc]\otimes_{S(\fc)}\CC[\hh^{*}]^{W}[\fc]^{\opp}$-lattice inside of $(\eta_{\underline{W}*}(\bB^{\Xi}))_{\widehat{\underline{\cL}}_{W'}^{\times}}$. So what we need to show is that $\eta_{W'}^{*}(\eta_{\underline{W}*}(\bB^{\Xi}))\cap\bL$ is finitely generated over $\CC[\hh^{reg-W'}/W'][\fc]\otimes_{S(\fc)}\CC[\hh^{*}/W][\fc]^{\opp}$ for \emph{some} lattice $\bL$. We can produce such a lattice as follows. Again thanks to Lemma \ref{lemma_fg}, $\bB$ is finitely generated over $\CC[\hh^{reg-\underline{W}}/\underline{W}][\fc]\otimes_{S(\fc)} \CC[\hh^{*}/W][\fc]^{\opp}$, so we have an epimorphism $\Upsilon: (\CC[\hh^{reg-\underline{W}}/\underline{W}][\fc] \otimes_{S(\fc)} \CC[\hh^{*}/W][\fc]^{\opp})^{\oplus n} \rightarrow \bB$, which in turn induces an epimorphism $\overline{\Upsilon}: (\CC[\widehat{\underline{\cL}}_{W'}^{\times}][\fc] \otimes_{S(\fc)} \CC[\hh^{*}/W][\fc]^{\opp})^{\oplus n} \rightarrow (\eta_{\underline{W}*}(\bB^{\Xi}))_{\widehat{\underline{\cL}}_{W'}^{\times}}$. We take $\bL$ to be the image of the restriction of $\overline{\Upsilon}$ to $(\CC[\widehat{\underline{\cL}}_{W'}][\fc] \otimes_{S(\fc)} \CC[\hh^{*}/W][\fc]^{\opp})^{\oplus n}$. This is clearly a lattice. Since $\CC[\hh^{reg-W'}/W']\cap \CC[\widehat{\underline{\cL}}_{W'}] = \CC[\hh^{reg-W'}/W']$ we have that  $\bL \cap \eta_{W'}^{*}(\eta_{\underline{W}*}(\bB^{\Xi}))$ coincides with the intersection of $\eta_{W'}^{*}(\eta_{\underline{W}*}(\bB^{\Xi}))$ with the image of the restriction of $\overline{\Upsilon}$ to $(\CC[\hh^{reg-W'}/W'][\fc] \otimes_{S(\fc)} \CC[\hh^{*}/W][\fc]^{\opp})^{\oplus n}$. So $\bL \cap \eta_{W'}^{*}(\eta_{\underline{W}*}(\bB^{\Xi}))$ is finitely generated over $\CC[\hh^{reg-W'}/W'][\fc]\otimes_{S(\fc)} \CC[\hh^{*}/W][\fc]^{\opp}$. Note that it follows that $\widetilde{\bB}_{W'}$ is a HC $\bH_{reg-W'}$-bimodule with $\Supp(\widetilde{\bB}_{W'}) = \overline{\cL^{W'}_{\underline{W}}}$. 
 
 It remains to show that $(\widetilde{\bB}_{W'})_{reg-\underline{W}} = \bB$. Since $\bB = (\eta_{\underline{W}*}(\bB^{\Xi}))_{reg-\underline{W}} = (\eta_{\underline{W}*}(\bB^{\Xi})_{reg-W'})_{reg-\underline{W}}$ it is enough to check that the lift of $\widetilde{\bB}_{W'}$ to $\hh^{reg-\underline{W}}/\underline{W}$ coincides with that of $(\eta_{\underline{W}*}(\bB^{\Xi}))_{reg-W'}$. By definition, $\widetilde{\bB}_{W'} \subseteq (\eta_{\underline{W}*}(\bB^{\Xi}))_{reg-W'} \subseteq (\eta_{\underline{W}*}(\bB^{\Xi}))_{\widehat{\underline{\cL}}_{W'}^{\times}}$. Now, for every $b \in (\eta_{\underline{W}*}(\bB^{\Xi}))_{\widehat{\underline{\cL}}_{W'}^{\times}}$ there exists $f \in \CC[\widehat{\underline{\cL}}_{W'}]$ vanishing on $\underline{\cL}_{W'}$ with $fb \in \bB_{W'}$. If, moreover, $b \in (\eta_{\underline{W}*}(\bB^{\Xi}))_{reg-W'}$ then $f \in \CC[\hh^{reg-W'}/W']$ and $fb \in \widetilde{\bB}_{W'}$. This implies the desired result. \\
 
 {\it Part 2: Glueing.} First we will define a sheaf on $\hh$, then we will take $W$-invariants to pass to $\hh/W$. For each parabolic subgroup $W'$ containing $\underline{W}$ in corank 1, let $\pi_{W'}: \hh^{reg-W'} \rightarrow \hh^{reg-W'}/W'$ be the projection, and $\iota_{W'}: \hh^{reg-W'} \rightarrow \hh$ the inclusion. So we can consider $\iota_{W'*}\pi_{W'}^{*}\widetilde{\bB}_{W'}$. We will take the intersection of these sheaves, so we need to find a sheaf containing all of them. Since $\hh^{reg-W'} \subseteq \hh^{sr-\underline{W}}$ for all $W'$, this will be a sheaf defined on $\hh^{sr-\underline{W}}$. So let $\pi: \hh \rightarrow \hh/W$ and $\iota:\hh^{sr-\underline{W}} \rightarrow \hh$ be the natural projection and inclusion, respectively. By the construction, viewing $\eta_{\underline{W}*}(\bB^{\Xi})$ as a quasicoherent sheaf on $\hh/W$, we may think of $\iota_{W'*}\pi_{W'}^{*}\widetilde{\bB}_{W'}$ as being contained inside of $\iota_{*}\pi^{*}\eta_{\underline{W}*}(\bB^{\Xi})$. So the intersection 
 
$$
\widetilde{\bB} := \bigcap_{\substack{\underline{W} \subseteq W' \\ \text{in corank} \; 1}} \iota_{W'*}\pi_{W'}^{*}\widetilde{\bB}_{W'}.
$$

\noindent makes sense and is a sheaf on $\hh^{sr-\underline{W}}$. 

Note that $W$ acts naturally on $\iota_{*}\pi^{*}\eta_{\underline{W}*}(\bB^{\Xi})$. Now notice that, for a parabolic subgroup $W' \subseteq W$ and $w \in W$, we have a canonical, graded isomorphism $\bH_{reg-W'} \cong \bH_{reg-wW'w^{-1}}$. Indeed, recall that $\bH_{reg-W'}$ is the rational Cherednik algebra for the action of $W$ on $\cX_{W'}$, a disjoint union of Zariski open subsets of $\hh$, cf. Subsection \ref{subsect:supports}. It is clear that $\cX_{W'} = \cX_{wW'w^{-1}}$ and the isomorphism between the algebras follows. So tracing back the construction, we see that we can pick our bimodules $\bB_{\widehat{\underline{\cL}}_{W'}}$ in such a way that, for $w \in W$, $w(\iota_{W'*}\pi_{W'}^{*}\widetilde{\bB}_{W'}) = \iota_{wW'w^{-1}*}\pi^{*}_{wW'w^{-1}}\widetilde{\bB}_{wW'w^{-1}}$. So $\widetilde{\bB}$ is $W$-stable. Finally, define 

$$
\widehat{\bB} := (\pi_{*}\widetilde{\bB})^{W},
$$

\noindent where $\pi: \hh \rightarrow \hh/W$ is the projection. We claim that $\widehat{\bB}$ is stable under the bimodule action of $\bH$. To see this first note that, by definition, $\widehat{\bB} = \pi_{*}\widetilde{\bB} \cap \jmath_{*}\bB^{\Xi}$, where $\jmath: \hh^{reg-\underline{W}} \rightarrow \hh/W$ is the projection. Each one of the bimodules on the right-hand side of the previous equality is stable under the (left or right) action of $\bH$. So $\widehat{\bB}$ is also $\bH$-stable.\\

Now set $\overline{\bB} := \Gamma(\pi(\hh^{sr-\underline{W}}), \widehat{\bB})$. We have that $\overline{\bB}$ is a $\bH$-bimodule. We claim that it is HC. First of all, since $\widehat{\bB} \subseteq \jmath_{*}\bB^{\Xi}$, we have that $\overline{\bB}$ is $\CC[\hbar]$-flat. It is also clear that $\overline{\bB}/\hbar\overline{\bB}$ is a module over $Z(\bH/\hbar\bH)$ and that $\overline{\bB}$ is graded. So, to finish the claim that $\overline{\bB}$ is HC, we need to show that it is finitely generated. We will show that, in fact, $\overline{\bB}/\fc\overline{\bB}$ is finitely generated over the algebra $\CC[\hh/W] \otimes \CC[\hh^{*}/W]$. The following is an easy consequence of \cite[Lemma 3.6]{wilcox}.

\begin{lemma}\label{lemma:wilcox}
Let $X$ be an affine Noetherian scheme, and let $U \subseteq X$ be an open subset of $X$ whose complement has codimension at least 2. Let $M$ be a coherent sheaf on $U$, and assume that the support of any global section $m \in \Gamma(U, M)$ contains an irreducible component of $U$. Then, $\Gamma(U, M)$ is finitely generated over $\CC[X]$.
\end{lemma}
\begin{proof}
By \cite[Lemma 3.6]{wilcox}, we get that $\Gamma(U, M)/I\Gamma(U, M)$ is a finitely generated $\CC[X]/I\CC[X]$-module, where $I$ is the nilradical of $\CC[X]$. The result follows.
\end{proof}

Note that we can look at $\widehat{\bB}/\fc\widehat{\bB}$ as a coherent sheaf on an infinitesimal neighborhood $U$ of $\pi(\hh^{sr-\underline{W}}\cap\hh^{\underline{W}}) \times \hh^{*}/W$, this follows from our assumptions on the singular support of $\bB$, $\Supp(\bB) = \cL^{\underline{W}}_{\underline{W}}$, the construction of $\widehat{\bB} \subseteq \jmath_{*}\bB^{\Xi}$ and Lemma \ref{lemma:ann}. This infinitesimal neighborhood may be regarded as an open set inside an infinitesimal neighborhood $X$ of $\pi(\hh^{sr-\underline{W}}) \times \hh^{*}/W$. Now, it follows from Lemma \ref{lemma:support_elements} that the support of any global section $m \in \Gamma(U, \widehat{\bB}/\fc\widehat{\bB})$ contains $\pi(\hh^{\underline{W}}_{reg})\times W(\hh^{*})^{\underline{W}}/W$. Then, it follows from Lemma \ref{lemma:wilcox} that $\widehat{\bB}/\fc\widehat{\bB}$ is finitely generated over $\CC[X]$. In particular, it is finitely generated over $\CC[\hh/W] \otimes \CC[\hh^{*}/W]$, this follows because the codimension of the complement of $\hh^{\underline{W}-sr}$ in $\hh$ is 2. Then, $\overline{\bB}$ is a HC $\bH$-bimodule. By construction, $\overline{\bB}_{reg-\underline{W}} = \bB$. This finishes the proof of Theorem \ref{thm_main}. 
\end{proofmain}

Let us remark one important feature of the bimodule $\overline{\bB}$ we have constructed: it has no sub-bimodules whose singular support is properly contained inside $\overline{\cL_{\underline{W}}}$. Indeed, this follows from Corollary \ref{cor:support_important} and the fact that $\widehat{\bB} \subseteq \bB^{\Xi}$.

\subsection{Specializing parameters.}\label{subsect:specpar} Let $c, c': \CC[\cS]^{W} \rightarrow \CC$ be parameters. Recall that $\R_{\hbar}(H_{c})$, $\R_{\hbar}(H_{c'})$ are quotients of $\bH$ and that, if $B$ is a HC $H_{c}\text{-}H_{c'}$-bimodule with a good filtration, then $\R_{\hbar}(B)$ is a HC $\bH$-bimodule. A similar result holds for HC $H_{c,reg-\underline{W}}\text{-}H_{c', reg-\underline{W}}$-bimodules. We remark that, since $\hh^{*}$ is in degree 0, the Rees construction commutes with localization: for an affine, open subset $U \subseteq \hh/W$, $\R_{\hbar}(B)|_{U} = \R_{\hbar}(B|_U)$. We also remark that the Bezrukavnikov-Etingof isomorphisms hold in the specialized setting. So we can take a $\Xi$-equivariant HC $H_{c, reg-\underline{W}}\text{-}H_{c', reg-\underline{W}}$-bimodule $B$ such that $\Supp(B) = \cL^{\underline{W}}_{\underline{W}}$ and for every parabolic subgroup $W'$ containing $\underline{W}$ in corank 1, the $H_{c}(W', \hh)_{\widehat{\underline{\cL}}^{\times}_{W'}}\text{-}H_{c'}(W', \hh)_{\widehat{\underline{\cL}}^{\times}_{W'}}$-bimodule $B^{\Xi}_{\widehat{\underline{\cL}}^{\times}_{W'}}$ is the localization of a HC $H_{c}(W', \hh)_{\widehat{\underline{\cL}}_{W'}}\text{-}H_{c'}(W', \hh)_{\widehat{\underline{\cL}}_{W'}}$-bimodule. By Theorem \ref{thm_main}, we can find a HC $\bH$-bimodule $\overline{\bB}$ that lifts to $\R_{\hbar}(B)$. Since $\overline{\bB}_{\widehat{\underline{\cL}}_{W'}}$ is a bimodule over $\R_{\hbar}(H_{c}(W', \hh)_{\widehat{\underline{\cL}}_{W'}})\text{-}\R_{\hbar}(H_{c'}(W', \hh)_{\widehat{\underline{\cL}}_{W'}})$, we see that the bimodule $\overline{\bB}$ factors through $\R_{\hbar}(H_{c})\text{-}\R_{\hbar}(H_{c'})$, so $\overline{\bB}/(\hbar - 1)\overline{\bB}$ is a HC $H_{c}\text{-}H_{c'}$-bimodule that lifts to $B$. We summarize this discussion in the following theorem, which is a specialized version of Theorem \ref{thm_main}.

\begin{theorem}\label{thm_mainparameters}
Let $\underline{W}$ be a parabolic subgroup of $W$, and let $B$ be a $\Xi$-equivariant HC $H_{c, reg-\underline{W}}\text{-}H_{c', reg-\underline{W}}$-bimodule,  where $c, c'\in \CC[\cS]$ are conjugation invariant functions. Assume that $\Supp(B) = \cL^{\underline{W}}_{\underline{W}}$ and that for all minimal parabolic subgroups $W' \subseteq W$ containing $\underline{W}$, the bimodule $(\eta_{\underline{W}*}(B^{\Xi}))_{\widehat{\underline{\cL}}^{\times}_{W'}}$ is the localization of a HC $H_{c}(W', \hh)_{\widehat{\underline{\cL}}_{W'}}\text{-}H_{c'}(W', \hh)_{\widehat{\underline{\cL}}_{W'}}$-bimodule. Then, there exists a HC $H_{c}\text{-}H_{c'}$-bimodule $\overline{B}$ such that $\overline{B}_{reg-\underline{W}} = B$. 
\end{theorem}

Now Theorem \ref{thm:veryverytechnical} is a consequence of Theorem \ref{thm_mainparameters}. \\

\begin{proofthmvery}
Recall the functor $\cG$ defined in Subsection \ref{subsection:induction}. Since $B$ is finite dimensional, $\Supp(\cG(\R_{\hbar}(B))) = \cL^{\underline{W}}_{\underline{W}}$. Our assumptions on $B$ imply, since $\cG$ is a fully faithful embedding, that $\cG(\R_{\hbar}(B))$ satisfies the conditions of Theorem \ref{thm_mainparameters}. So we can find a HC $\bH$-bimodule $\overline{\bB}$ with $\overline{\bB}_{reg-\underline{W}} = \cG(\R_{\hbar}(B))$. Note that since $\Supp(\cG(\R_{\hbar}(B))) = \cL^{\underline{W}}_{\underline{W}}$, $(\overline{\bB}_{reg-\underline{W}})^{\underline{\cL}} = \overline{\bB}_{reg-\underline{W}}$. Thus, by the construction of the restriction functor, $\overline{\bB}_{\dagger^{W}_{\underline{W}}} = \R_{\hbar}B$. It remains to put $\overline{B} := \overline{\bB}/(\hbar - 1)\overline{\bB}$.
\end{proofthmvery}

\section{Localization of HC bimodules.}\label{sect:localization}
Now we turn our attention to the study of Harish-Chandra bimodules with full support. It is natural to localize these bimodules to the regular locus $\hh^{reg}$ to obtain a bimodule over the algebra $\cD(\hh^{reg}\#W$, recall from Lemma \ref{lemma_loc} that left, right, and two-sided localization of a HC bimodule all coincide. In this section, we will see which $\cD(\hh^{reg})\#W$-bimodules can appear as the localization of irreducible HC bimodules. This will allow us to see, in particular, for which irreducible modules $M \in \cO_{c}$ the space of locally finite maps $\homf{\Delta_{c'}(\triv)}{M}$ can be nonzero. In the next section, we will use this result to prove Theorem \ref{thm:main1}.

We would like to remark that, after writing a preliminary version of this paper, we found out that most of the results in this section are already contained in some form in \cite{spencer}. There it is assumed that all parameters are regular (i.e., that $\cO_{c}$ is a semisimple category or, equivalently, that the algebra $H_{c}$ is simple) but, under minor modifications we provide here, the results are still valid in the general case, see Remark \ref{rmk:socleKZ}. 

\subsection{Bimodules over $\cD(\hh^{reg})\#W$}\label{subsect:4.1} Recall that, independently of the parameter $c$, $H_{c, \hh^{reg}/W}$ is isomorphic to $\cD(\hh^{reg})\#W$. So we start by studying bimodules over the latter algebra. Since $W$ acts freely on $\hh^{reg}$, the algebras $\cD(\hh^{reg})\#W$ and $\cD(\hh^{reg}/W)$ are Morita equivalent, an equivalence is given by $M \mapsto eM$ where $e = |W|^{-1}\sum_{w \in W}w$ is the trivial idempotent for $W$. Throughout this section, we denote $X := \hh^{reg}/W$. 

We will relate HC $H_{c}$-bimodules to spaces of differential operators on local systems on $X$, i.e., on $\cD(X)$-modules which are finitely generated over $\CC[X]$. So, first, we recall Grothendieck's definition of differential operators: if $M$ and $N$ are $\CC[X]$-modules, then the space of $\CC[X]$-differential operators on $M$ with values in $N$ is a subspace of $\Hom_{\CC}(M, N)$, defined via an increasing filtration $\Diff(M, N) = \bigcup_{n \geq 0} \Diff(M,N)_{n}$, where the components $\Diff(M, N)_{n}$ are inductively defined as follows:
\begin{equation*}
\Diff(M, N)_{-1} := 0, \; \; \; \Diff(M, N)_{n+1} := \{f \in \Hom_{\CC}(M, N) : [a,f] \in \Diff(M, N)_{n} \; \text{for all} \; a \in \CC[X]\}.
\end{equation*}

If $M, N$ are $\cD(X)$-modules, then $\Diff(M, N)$ is a $\cD(X)$-bimodule. We remark that if $N$ is a local system then we have a $\cD(X)$-bimodule isomorphism, $N \otimes_{\CC[X]} \cD(X) \cong \Diff(\CC[X], N)$, where the flat connection on $N \otimes_{\CC[X]}\cD(X)$ giving rise to the left $\cD(X)$-action is as in \cite[Proposition 1.2.9]{HTT}. An explicit isomorphism is given by $n \otimes_{\CC[X]} d \mapsto (f \mapsto d(f)n)$. Note that this implies that $\Diff(\CC[X], N)$ is finitely generated both as a right and as a left $\cD(X)$-module whenever $N$ is a local system. As a right $\cD$-module, an explicit set of generators is $n_1 \otimes_{\CC[X]}1, \dots, n_i \otimes_{\CC[X]}1$, where $n_1, \dots, n_i$ are generators of the $\CC[X]$-module $N$. This set also generates $N \otimes_{\CC[X]}\cD(X)$ as a left $\cD$-module. Since the algebra $\cD(X)$ is simple and noetherian, \cite[Theorem 10]{Br} implies the following. 

\begin{lemma}\label{lemma:brown}
Let $N$ be a nonzero local system. Then, $\Diff(\CC[X], N)$ is a progenerator in both categories of left and right $\cD(X)$-modules. In particular, $\Diff(\CC[X], N)\otimes_{\cD(X)} M \not= 0$ (resp. $M \otimes_{\cD(X)}\Diff(\CC[X], N) \not= 0$) for any nonzero left (resp. right) $\cD(X)$-module $M$.
\end{lemma}

Now assume that $N$ is an irreducible local system. We have a natural evaluation map of $\cD(X)$-modules 
$$
\Diff(\CC[X], N) \otimes_{\cD(X)} \CC[X] \rightarrow N, \; \phi \otimes f \mapsto \phi(f).
$$
\noindent Since $\Diff(\CC[X], N) \neq 0$, the evaluation map is not zero. Then, by the irreducibility of $N$, this map is surjective. We claim that it is also injective. To see this, note that we have an isomorphism $\Diff(\CC[X], N)\otimes_{\cD(X)} \CC[X] \cong N$ of $\CC[X]$-modules. This follows from the description of $\Diff(\CC[X], N)$ above. So we can view the evaluation map as an element of $\End_{\CC[X]}(N, N)$. Now $N$ is noetherian since it is finitely generated over $\CC[X]$. We have seen that the evaluation map is surjective. Hence, it must also be injective. This discussion has the following consequence.

\begin{prop}\label{irredbimod}
If $N$ is an irreducible local system over $X$, then the $\cD(X)$-bimodule $\Diff(\CC[X], N)$ is irreducible.
\end{prop}
\begin{proof}
 Assume $\Diff(\CC[X], N)$ has a nontrivial sub-bimodule $V$, and let $V'$ be the quotient bimodule. The result of \cite[Theorem 10]{Br} implies that {\it any} nonzero sub-bimodule or quotient module of $\Diff(\CC[X], N)$ must be a progenerator of $\cD(X)$-mod and mod-$\cD(X)$. Then, $\text{Tor}^{1}_{\cD(X)}(V', N) = 0$ and we have a short exact sequence

$$
0 \rightarrow V \otimes_{\cD(X)} N \rightarrow \Diff(\CC[X], N)\otimes_{\cD(X)} N \rightarrow V' \otimes_{\cD(X)} N \rightarrow 0.
$$

\noindent By the discussion above, the module  $\Diff(\CC[X], N)\otimes_{\cD(X)} N$ is irreducible. This forces one of $V \otimes_{\cD(X)} N$ or $V' \otimes_{\cD(X)} N$ to be $0$. But both $V, V'$ are progenerators of $\cD(X)$-mod and mod-$\cD(X)$. This is a contradiction.
\end{proof}


\subsection{Localization of HC bimodules.}\label{subsection:localization} We will apply the results of the previous subsection to give an explicit description of the localization of certain Harish-Chandra $H_{c}\text{-}H_{c'}$-bimodules. Namely, consider the Verma module $\Delta_{c'}(\triv)$. We will consider bimodules of the form $\homf{\Delta_{c'}(\triv)}{N}$, for $N \in \cO_{c}$ with full support. We first see that any irreducible HC $H_{c}\text{-}H_{c'}$-bimodule with full support is contained in such a bimodule. Throughout this section, we make use of the following observation.

\begin{rmk}
	We remark that $e\Delta_{c}(\triv)[\delta^{-1}] = \CC[X]$ with its natural left $\cD(X)$-action. This follows easily from an explicit description of the isomorphism $H_{c}[\delta^{-1}] \cong \cD(\hh^{reg})\#W$, which is given in terms of Dunkl-Opdam operators, see e.g. \cite[Proposition 4.5]{EG}. 
\end{rmk}

\begin{prop}\label{prop:subbimod}
Let $B$ be an irreducible HC $H_{c}\text{-}H_{c'}$-bimodule with full support. Then, there exists an irreducible object $T \in \cO_{c}$ such that $B$ is a sub-bimodule of $\homf{\Delta_{c'}(\triv)}{T}$.
\end{prop}
\begin{proof}
Since $V$ has full support, we see that $B[\delta^{-1}] \not= 0$, so $e(B[\delta^{-1}])e \not= 0$. Note that $e(B[\delta^{-1}])e$ is a bimodule over the algebra $\cD(X)$ that is noetherian as either a left or right $\cD(X)$-module. Then, it is a progenerator of $\cD(X)\text{-mod}$, so $e(B[\delta^{-1}])e\otimes_{\cD(X)} \CC[X] \neq 0$. In particular, this implies that $B \otimes_{H_{c'}} \Delta_{c'}(\triv) \neq 0$. Let $T$ be an irreducible quotient of the latter module. It is easy to see that we have a nonzero bimodule map $B \rightarrow \homf{\Delta_{c'}(\triv)}{T}$, and the proof follows.
\end{proof}

Note that for an irreducible module $N \in \cO_{c}$, we have that $e\homf{\Delta_{c'}(\triv)}{N}[\delta^{-1}]e$ is a $\cD(X)$-bimodule. We claim that this bimodule is isomorphic to $\Diff(\CC[X], N_{X})$ whenever the former bimodule is nonzero and $N_{X} := eN[\delta^{-1}]$. This claim follows from Proposition \ref{irredbimod} together with the following result.

\begin{lemma}\label{lemma:localization}
Let $N \in \cO_{c}$ be an irreducible module with full support. For any standard module $\Delta_{c'}(\tau)$, the bimodule $e\homf{\Delta_{c'}(\tau)}{N}[\delta^{-1}]e$ is isomorphic to a sub-bimodule of $\Diff(e\Delta_{c'}(\tau)[\delta^{-1}], eN[\delta^{-1}])$. 
\end{lemma}
\begin{proof}
Set $M := \Delta_{c'}(\tau)$. Let $f \in \homf{M}{N}$. Since for some $m$, $\delta^{m}$ is $W$-invariant, we have that $(\ad (e\delta^{m}))^{k}f = 0$, so for every $x \in M$,

$$
\sum_{j = 0}^{k}(-1)^{j}\left(^{k}_{j}\right)(e\delta^m)^{(k-j)}f((e\delta^m)^{j}x) = 0.
$$

Then, since $M$ is free as a $\CC[\hh]$-module we can extend $f$ to $eM[\delta^{-1}]$ by 

$$
f(\delta^{-m}x) = -(e\delta^m)^{-k}\sum_{j = 1}^{k}(-1)^{j}\left(^{k}_{j}\right)(e\delta^m)^{(k-j)}f((e\delta^m)^{(j-1)}x).
$$

To see that this actually defines an inclusion, assume that $f \not= 0$. Then, $f(x) \not= 0$ for some element $x \in M$. Since $N$ is torsion-free (see e.g. \cite[Proposition 5.21]{GGOR}), the element $f(x)$ is not a zero divisor. This implies that the image of $f$ in $\Diff(e\Delta_{c'}(\tau)[\delta^{-1}], eN[\delta^{-1}])$ is nonzero.
\end{proof}

\begin{cor}\label{cor:irredHCbimod}
Let $B$ be an irreducible HC $H_{c}\text{-}H_{c'}$-bimodule with full support. Then, there exists an irreducible local system $N$ on $X$ such that $eB[\delta^{-1}]e = \Diff(\CC[X], N)$.
\end{cor}
\begin{proof}
By Proposition \ref{prop:subbimod} and Lemma \ref{lemma:localization}, we have that $eB[\delta^{-1}]e \subseteq \Diff(\CC[X], N)$ for some irreducible local system $N$. But $\Diff(\CC[X], N)$ is an irreducible $\cD(X)$ bimodule, Proposition \ref{irredbimod}. We are done.
\end{proof}

\begin{cor}\label{cor:Hc_loc_fin_maps}
We have an isomorphism $H_{c} \cong \homf{\Delta_{c}(\triv)}{\Delta_{c}(\triv)}$.
\end{cor}
\begin{proof}
Reasoning as in the proof of Lemma \ref{lemma:localization}, we have that $\homf{\Delta_{c}(\triv)}{\Delta_{c}(\triv)}$ is a HC $H_{c}$-bimodule whose localization to $\hh^{reg}$ is an irreducible $\cD(\hh^{reg})\#W$-bimodule. So $\homf{\Delta_{c}(\triv)}{\Delta_{c}(\triv)}$ contains a unique irreducible bimodule with full support. It is easy to see that any subbimodule of $\homf{\Delta_{c}(\triv)}{\Delta_{c}(\triv)}$ has full support, so this bimodule has an irreducible socle. In particular, it is indecomposable.\\
On the other hand, we have a natural map $H_{c} \rightarrow \homf{\Delta_{c}(\triv)}{\Delta_{c}(\triv)}$, $x \mapsto (m \mapsto xm)$. Since the representation $\Delta_{c}(\triv)$ is faithful, this is an inclusion. Then, by Proposition \ref{prop:injective}, we have that $H_{c}$ must be isomorphic to a direct summand of $\homf{\Delta_{c}(\triv)}{\Delta_{c}(\triv)}$. By the previous paragraph, we must have $H_{c} \cong \homf{\Delta_{c}(\triv)}{\Delta_{c}(\triv)}$.
\end{proof}

\begin{rmk}
We remark that the isomorphism in Corollary \ref{cor:Hc_loc_fin_maps} is also an \emph{algebra} isomorphism with respect to the composition structure on $\homf{\Delta_{c}(\triv)}{\Delta_{c}(\triv)}$. This generalizes \cite[Proposition 8.10 (i)]{BEG}. 
\end{rmk}


\subsection{KZ functor.}\label{subsection:KZ} Since for any HC $H_{c}\text{-}H_{c'}$-bimodule $B$ and any module $M \in \cO_{c'}$, the module $B \otimes_{H_{c'}} M$ is a module in category $\cO_{c}$, it makes sense to ask what is the image of a module of the form $B \otimes_{H_{c'}} M$ under the KZ functor. In this subsection, we answer this question when $B$ has the form $\homf{\Delta_{c'}(\triv)}{M}$ for an irreducible module with full support $M \in \cO_{c}$. Namely, we have the following result.

\begin{lemma}\label{lemma:mainKZ}
Let $c, c': \cS \rightarrow \CC$ be conjugation invariant functions and consider the rational Cherednik algebras $H_{c}, H_{c'}$. Let $q, q'$ be the associated sets of parameters for the Hecke algebras $\cH_{q}, \cH_{q'}$, so that we have $\KZ_{c}: \cO_{c} \rightarrow \cH_{q}\modd$, $\KZ_{c'}: \cO_{c'} \rightarrow \cH_{q'}\modd$. Let $M \in \cO_{c}$ be an irreducible module with full support. Assume that $\homf{\Delta_{c'}(\triv)}{M} \not= 0$. Then, for every finite dimensional module $N \in \cH_{q'}\modd$, the $\pi_1(X)$-module $\KZ_{c}(M) \otimes_{\CC} N$ factors through $\cH_{q}$. 

\end{lemma}
\begin{proof}
We show that for every $\widetilde{N} \in \cO_{c'}$:

$$
\KZ_{c}(\homf{\Delta_{c'}(\triv)}{M}\otimes_{H_{c'}} \widetilde{N}) = \KZ_{c}(M) \otimes_{\CC} \KZ_{c'}(\widetilde{N}).
$$

\noindent Since $\cH_{q'}\modd$ is a quotient of $\cO_{c'}$ via the KZ functor, this implies the result. Now, by results in the previous subsection the localization to $X$ of $\homf{\Delta_{c'}(\triv)}{M} \otimes_{H_{c'}} N$ is $\Diff(\CC[X], M_{X})\otimes_{\cD(X)} N_X$. Since $M_X$ is a local system, $\Diff(\CC[X], M_X) \cong M_X \otimes_{\CC[X]} \cD(X)$. Then,

$$
(\homf{\Delta_{c'}(\triv)}{M} \otimes_{H_{c'}} \widetilde{N})_{X} = M_X \otimes_{\CC[X]} \cD(X) \otimes_{\cD(X)} \widetilde{N}_X = M_X \otimes_{\CC[X]} \widetilde{N}_X.
$$

\noindent By \cite[Proposition 4.7.8]{HTT}, $DR(M_X \otimes_{\CC[X]} \widetilde{N}_X)$ is precisely $DR(M_X) \otimes_{\CC} DR(\widetilde{N}_X)$, with diagonal action of the braid group $\pi_1(X)$. The lemma is proved.
\end{proof}

\begin{rmk}\label{rmk:socleKZ}
Note that $\Delta_{c'}(\triv)$ has an irreducible socle. This follows because $\Delta_{c'}(\triv)$ is free as a $\CC[\hh]$-module (so that every submodule is torsion free) and the localization $\Delta_{c'}(\triv)_{\hh^{reg}}$ is an irreducible $\cD(\hh^{reg})\#W$-module. Moreover, for $S := \Soc(\Delta_{c'}(\triv))$, we get $\KZ(S) = \KZ(\Delta_{c'}(\triv))$. Then, Lemma \ref{lemma:mainKZ} holds, with the same proof, if we substitute $\Delta_{c'}(\triv)$ by $S$. We will mostly use this form of the lemma.
\end{rmk}


\section{HC bimodules with full support.}\label{sect:fullsupport}

In this section, we prove Theorem \ref{thm:main1}. In order to do so, we introduce the action of a product of symmetric group (the Namikawa-Weyl group) on the space of parameters $\fp$ for the Cherednik algebra. It turns out that it is more convenient to define this action under a reparametrization of $H_{c}$, and this is what we do first, Subsection \ref{subsect:reparam}. Then, in Subsections \ref{subsect:nami}, \ref{subsect:namiequiv} and \ref{subsect:namiaction} we study the action of the Namikawa-Weyl group on the set of parameters and, more precisely, its effect under passing from $\fp$ to the Hecke parameters \lq$ q$\rq, this will give us (1) of Theorem \ref{thm:main1}. Finally, in Subsection \ref{subsect:subgroup} we construct the group $W_{c}$, and in Subsection \ref{subsect:proof full support} we finish the proof of Theorem \ref{thm:main1}.

\subsection{A reparametrization of $H_{c}$}\label{subsect:reparam} In the sequel, a reparametrization of the rational Cherednik algebra $H_{c}$ by parameters which are more amenabe with the KZ functor will be convenient. Here we explain such a parametrization, which is standard, see e.g. Section 3 in \cite{GGOR} or Section 2 in \cite{gordon-losev}.

Recall that by $\cA$ we denote the set of reflection hyperplanes in $\hh$, and for each reflection hyperplane $\Gamma$ we denote by $W_{\Gamma}$ its stablizer, this is a cyclic group of order $\ell_{\Gamma}$. Let $s_{\Gamma} \in W_{\Gamma}$ be a generator with $\det(s_{\Gamma}|_{R}) = \exp(2\pi\sqrt{-1}/\ell_{\Gamma}) =: \eta_{\Gamma}$. We may and will assume that $\alpha_{s} = \alpha_{{s}_{\Gamma}}$, $\alpha_{s}^{\vee} = \alpha_{s_{\Gamma}}^{\vee}$ for every $s \in W_{\Gamma}\cap\cS$. We will denote these elements by $\alpha_{\Gamma} \in \hh^{*}, \alpha_{\Gamma}^{\vee} \in \hh$, respectively.

Now, for each $i = 0, \dots, \ell_{\Gamma} - 1$ we have the idempotent
\[
\be_{i, \Gamma} := \frac{1}{\ell_{\Gamma}}\sum_{j = 0}^{\ell_{\Gamma}-1}\eta_{\Gamma}^{-ij}s_{\Gamma}^{j} \in \CC W_{\Gamma}
\]

\noindent so that, for example, $\be_{0, \Gamma} \in \CC W_{\Gamma}$ is the trivial idempotent. For each reflection hyperplane $\Gamma \in \cA$, pick a collection of numbers $k_{\Gamma, 0}, \dots, k_{\Gamma, \ell_{\Gamma} -1}$ such that $k_{\Gamma, i} = k_{\Gamma', i}$ for every $i = 0, \dots, \ell_{\Gamma} - 1 = \ell_{\Gamma'}- 1$ if $\Gamma, \Gamma'$ are in the same $W$-orbit. Then define the algebra $H_{k}$ by generators and relations similar to (\ref{eqn:relations}), with the last relation replaced by

\begin{equation}\label{eqn:krelations}
[y, x] = \langle y, x \rangle - \frac{1}{2}\sum_{\Gamma \in \cA}\langle y, \alpha_{\Gamma}\rangle\langle\alpha_{\Gamma}^{\vee}, x \rangle \sum_{i = 0}^{\ell_{\Gamma} - 1}(k_{\Gamma, i} - k_{\Gamma, i-1})\be_{i, \Gamma}
\end{equation}

\noindent where $k_{\Gamma, -1} := k_{\Gamma, \ell_{\Gamma}-1}$. We remark that $H_{k} = H_{c}$, where the parameter $c: \cS \rightarrow \CC$ is recovered from $k$ as follows. For each reflection $s \in \cS$, let $\Gamma_{s}$ be the reflection hyperplane of $s$. Then, 

\begin{equation}\label{eqn:cfromk}
c(s) = \frac{1}{2\ell_{\Gamma_{s}}}\sum_{j = 0}^{\ell_{\Gamma_{s}}-1}(k_{\Gamma_{s}, j} - k_{\Gamma_{s}, j-1})\lambda_{s}^{-j}
\end{equation}

Note that the $k_{\Gamma, i}$ are only defined up to a common summand. Here, \emph{we will always assume that $k_{\Gamma, 0} = 0$}. In this case, we can recover the parameter $k$ from the parameter $c$ via (\ref{eqn:h}). We will still denote the set of $k$-parameters by $\fp$. In particular, we still have the notion of integral parameters, a parameter $k$ is integral if and only if $k_{\Gamma, i}/\ell_{\Gamma} \in \ZZ$ for every $\Gamma \in \cA$, $i = 0, \dots, \ell_{\Gamma} - 1$. Of course, Remark \ref{rmk:integral} is still valid and we will make extensive use of it without further comment. 

\subsection{The Namikawa-Weyl group}\label{subsect:nami} For $\Gamma \in \cA$, let us denote by $\fp(W_{\Gamma})$ the parameter space for the rational Cherednik algebra associated to the action of $W_{\Gamma}$ on the 1-dimensional space $\hh_{W_{\Gamma}}$. Note that $\fp(W_{\Gamma})$ is a $(\ell_{\Gamma} - 1)$-dimensional space, an element $k_{\Gamma} \in \fp(W_{\Gamma})$ is a tuple $k_{\Gamma} = (k_{\Gamma, 0} = 0, k_{\Gamma, 1}, \dots, k_{\Gamma, \ell_{\Gamma} -1})$. Moreover, we have

\[
\fp = \bigoplus_{\Gamma \in \cA/W} \fp(W_{\Gamma})
\]

\begin{definition}
	We define the \emph{Namikawa-Weyl group} of $W$ to be 
	
	\[
	\nam := \prod_{\Gamma \in \cA/W} \fS_{\ell_{\Gamma}}
	\]
	
	\noindent were $\fS_{\ell_{\Gamma}}$ denotes the symmetric group on $\ell_{\Gamma}$-elements.
\end{definition}

The Namikawa-Weyl group acts on the space $\fp$ as follows. An element $\sigma \in \fS_{\ell_{\Gamma}}$ acts on the space $\fp(W_{\Gamma})$ and leaves the elements in $\fp(W_{\Gamma'})$ fixed if $\Gamma'$ is not $W$-conjugate to $\Gamma$. Thus, to describe the action, we may assume that $W$ is of the form $W_{\Gamma}$ for some hyperplane $\Gamma \in \cA$, i.e., $W$ is a cyclic group on $\ell := \ell_{\Gamma}$ letters and $\nam = \fS_{\ell}$. Let us denote by $s_{1}, \dots, s_{\ell - 1}$ the simple reflections in $\fS_{\ell}$, that is, $s_{i} = (i, i+1)$. Now, for $k = (k_{0} = 0, \dots, k_{\ell - 1}) \in \fp$ we have

\[
(s_{1}k)_{j} = \begin{cases} 0, & j = 0 \\ 2 - k_{1}, & j = 1 \\ 1 + k_{j} - k_{1}, & j \neq 0, 1 \end{cases} \hspace{1cm} (s_{i}k)_{j} = \begin{cases} k_{j}, & j \neq i, i - 1 \\ k_{i} - 1, & j = i - 1 \\ k_{i-1} + 1, & j = i \end{cases}
\]

\noindent where, on the right-hand side, $i = 2, \dots, \ell - 1$. The first part of the following result is now an easy consequence of results of Losev, see \cite[Theorems 5.3.1 and 6.2.2]{losev_iso}, while the second part follows immediately from the definitions. 

\begin{lemma}
	Assume $W$ is a cyclic group. Then, for any $k \in \fp$ and $\sigma \in \nam$, the spherical subalgebras $eH_{k}e$ and $eH_{\sigma(k)}e$ are isomorphic as filtered algebras. Thus, all of the following categories are equivalent in a way that preserves the support filtration.
	
	\[
	\begin{array}{cccc}
	\HC(eH_{k}e, eH_{k}e), & \HC(eH_{k}e, eH_{\sigma(k)}e), & \HC(eH_{\sigma(k)}e, eH_{\sigma(k)}e), & \HC(eH_{\sigma(k)}e, eH_{k}e)
	\end{array}
	\]
\end{lemma}

Since passing to the spherical subalgebra does not kill nonzero Harish-Chandra bimodules with full support, we get the following result. Recall that we denote by $\overline{\HC}(k, k')$ the quotient category of all HC $H_{k}\text{-}H_{k'}$-bimodules modulo the full subcategory consisting of those bimodules whose support is properly contained in $(\hh \oplus \hh^{*})/W$. 

\begin{cor}\label{cor:namicyclic}
	Let $W$ be a cyclic group, and $k \in \fp$. Then, for $\sigma \in \nam$, all categories $\overline{\HC}(k, k)$, $\overline{\HC}(k, \sigma(k))$, $\overline{\HC}(\sigma(k), \sigma(k))$ and $\overline{\HC}(\sigma(k), k)$ are equivalent. 
\end{cor}

\begin{rmk}
	The Namikawa-Weyl group $\nam$ coincides with the Namikawa-Weyl group associated to the symplectic quotient singularity $(\hh \oplus \hh^{*})/W$, and its action on $\fp$ can be identified with its action on $H^{2}(X, \CC)$, where $X$ is a $\QQ$-factorial terminalization of $(\hh \oplus \hh^{*})/W$, see e.g. \cite{BST}.
\end{rmk}

\subsection{Equivalences from the Namikawa-Weyl group}\label{subsect:namiequiv} In this section, we show that the conclusion of Corollary \ref{cor:namicyclic} is valid for \emph{any} complex reflection group, not just cyclic groups. The following lemma, together with Corollary \ref{cor:semisimple} will be key to the results of this section.

\begin{lemma}\label{lemma:semisimple}
Assume that the category $\overline{\HC}(k,k')$ is nonzero. Then, there exists a 1-dimensional character $\tau$ of $W$ such that $\homf{S}{\Delta_{k}(\tau)} \neq 0$ where, recall $S \in \cO_{k'}$ is the irreducible module that gets sent to the trivial representation under the $\KZ$ functor.
\end{lemma}
\begin{proof}
Assume that $\overline{\HC}(k,k') \neq 0$. Then, by Proposition \ref{prop:subbimod} there exists an irreducible module $N \in \cO_{k}$ with full support such that $\homf{S}{N} \neq 0$. Now, for each reflection hyperplane $\Gamma \in \cA$, consider the pointwise stabilizer $W_{\Gamma} \subseteq W$. This is a cyclic group. Note that, since $\homf{S}{N} \neq 0$, we have that $\homf{\Res^{W}_{W_{\Gamma}}(S)}{\Res^{W}_{W_{\Gamma}}(N)} \neq 0$, this follows from Lemma \ref{lemma:compatibility of res}. Since $\KZ$ commutes with restriction, we have that $S_{W_{\Gamma}}$ is the unique subquotient of $\Res^{W}_{W_{\Gamma}}(S)$ will full support, which implies that $\homf{S_{W_{\Gamma}}}{\Res^{W}_{W_{\Gamma}}(N)} \neq 0$. Now, in category $\cO$ for the rational Cherednik algebra of $W_{\Gamma}$, for every irreducible representation (= 1-dimensional character) $\tau$ of $W_{\Gamma}$, we have that either $L_{k}(\tau) = \Delta_{k}(\tau)$, or $L_{k}(\tau)$ is finite dimensional. Since $\Res^{W}_{W_{\Gamma}}(N)$ has full support, we conclude that there exists an irreducible representation $\tau_{\Gamma}$ of $W_{\Gamma}$ with $\homf{S_{W_{\Gamma}}}{\Delta_{c}(\tau_{\Gamma})} \neq 0$. We remark that we can take $\tau_{\Gamma} = \tau_{\Gamma'}$ if $\Gamma, \Gamma' \in \cA$ are conjugate, this follows from the conjugation invariance of $k$. 

Now recall that we have an isomorphism $\Hom(W, \CC^{\times}) \longrightarrow (\prod_{\Gamma \in \cA}\Hom(W_{\Gamma}, \CC^{\times}))/W$ that is given by restriction, cf. \cite[3.3.1]{rouquier}. So the $W$-equivariant choice of characters $\{\tau_{\Gamma} : \Gamma \in \cA\}$ determines a 1-dimensional character $\tau$ of $W$. We claim that $\homf{S}{\Delta_{k}(\tau)} \neq 0$. To see this, we will use Theorem \ref{thm:veryverytechnical}. Assume for the moment that for every reflection hyperplane $\Gamma \in \cA$, $\homf{S_{W_{\Gamma}}}{\Res^{W}_{W_{\Gamma}}(\Delta_{k}(\tau))} \neq 0$. Then, $\homf{S_{W_{\Gamma}}}{\Res^{W}_{W_{\Gamma}}(\Delta_{k}(\tau))}_{\dagger_{\{1\}}^{W_\Gamma}}$ is a nonzero subbimodule of the 1-dimensional bimodule $\Hom_{\CC}(\CC, \KZ_{c}(\Delta_{k}(\tau)))$, so the conditions of Theorem \ref{thm:veryverytechnical} are satisfied. Using this theorem we get a HC $H_{k}\text{-}H_{k'}$-bimodule $B$ that localizes to $\Hom_{\CC}(\KZ_{k'}(S), \KZ_{k}(\Delta_{c}(\tau)))$. Using Lemma \ref{lemma:mainKZ} it is easy to see that, up to subquotients with proper support, $B = \homf{S_{W'}}{\Delta_{k}(\tau)}$. So what we need to show now is that the space of locally finite maps $\homf{S_{W_{\Gamma}}}{\Res^{W}_{W_{\Gamma}}\Delta_{k}(\tau)}$ is nonzero for every $\Gamma \in \cA$. We proceed to do this.

The $H_{k}(W_{\Gamma})$-module $\Res^{W}_{W_{\Gamma}}(\Delta_{k}(\tau))$ has a standard filtration, see e.g. \cite[Proposition 1.9]{shan}. Since, by construction, the restriction of the representation $\tau$ to $W_{\Gamma}$ is $\tau_{\Gamma}$, Proposition 3.14(ii) in \cite{BE}, allows us to conclude that  $\Res^{W}_{W_{\Gamma}}(\Delta_{k}(\tau)) = \Delta_{k}(\tau_{\Gamma})$. So $\homf{S_{W_{\Gamma}}}{\Res^{W}_{W_{\Gamma}}\Delta_{k}(\tau)} \neq 0$. We are done. 
\end{proof}

\begin{cor}\label{cor:semisimple}
Assume that $\overline{\HC}(k,k') \neq 0$. Then, the categories $\overline{\HC}(k,k')$ and $\overline{\HC}(k',k')$ are equivalent. Moreover, they are equivalent to the category of representations of $W/N$ for some normal subgroup $N$ of $W$.
\end{cor}
\begin{proof}
Let $B$ be a HC $H_{k}\text{-}H_{k'}$-bimodule with full support. By the previous lemma, we may asumme that $B = \homf{S}{\Delta_{k'}(\tau)}$ for a 1-dimensional character $\tau$ of $W$, so that $eB[\delta^{-1}]e = \Diff(\CC[X], M)$. Here, $M := e\Delta_{k'}(\tau)[\delta^{-1}]$ is a rank 1 local system. Then, the tensor product functor $eB[\delta^{-1}]e \otimes_{\cD(X)} \bullet$ induces a self-equivalence in the category of $\cD(X)$-bimodules. Indeed, this follows because $eB[\delta^{-1}]e = M \otimes_{\CC[X]}\cD(X)$ and $M$ is a line bundle on $X$. This implies that $B \otimes_{H_{k'}} \bullet: \HC(k',k') \rightarrow \HC(k,k')$ induces an equivalence between $\overline{\HC}(k',k')$ and $\overline{\HC}(k,k')$. The last assertion was checked in Subsection \ref{subsection:induction}.
\end{proof}

The following is the main result of this subsection. It generalizes Corollary \ref{cor:namicyclic}, and it gives one direction of Theorem \ref{thm:main1}(1). 

\begin{lemma}\label{lemma:naminonzero}
	Assume that either
	\begin{enumerate}
		\item There exists $\sigma \in \nam$ such that $k' = \sigma(k)$, or
		\item There exists $k'' \in \fp_{\ZZ}$ such that $k' = k + k''$.
	\end{enumerate}
Then, the category $\overline{\HC}(k, k')$ is nonzero, and so all categories $\overline{\HC}(k, k)$, $\overline{\HC}(k, k')$, $\overline{\HC}(k', k')$ and $\overline{\HC}(k', k)$ are equivalent.
\end{lemma}
\begin{proof}
	Let us show (1). We use Theorem \ref{thm:veryverytechnical} with $\underline{W} = \{1\}$, so $\underline{W}$ sits inside $W'$ in corank 1 if and only if $W' = W_{\Gamma}$ for some reflection hyperplane $\Gamma$.  By Corollary \ref{cor:namicyclic}, the category $\overline{\HC}(H_{k}(W_{\Gamma}, R_{W_{\Gamma}}), H_{\sigma(k)}(W_{\Gamma}, R_{W_{\Gamma}}))$ is nonzero. So we can find an irreducible (= 1-dimensional) representation $\tau_{\Gamma}$ of $W_{\Gamma}$ such that $\homf{S_{\sigma(k)}(W_{\Gamma})}{\Delta_{k}(\tau_{\Gamma})}$ is nonzero, where $S_{\sigma(k)}(W_{\Gamma})$ denotes the socle of the polynomial representation $\Delta_{\sigma(k)}(\triv_{W_{\Gamma}})$. Note that, by the $W$-invariance of the parameter $k$, we may assume that $\tau_{\Gamma} = \tau_{\Gamma'}$ if $\Gamma, \Gamma'$ are in the same $W$-orbit. Now proceed as in the proof of Lemma \ref{lemma:semisimple}.
	
	Let us now proceed to (2). First, assume that $k'' = \overline{\chi}$ for some 1-dimensional character $\chi: W \rightarrow \CC^{\times}$, cf. Subsection \ref{subsect:lattice}. Recall that we have an isomorphism $eH_{k}e \cong e_{\chi}H_{k'}e_{\chi}$, and so $e_{\chi}H_{k'}e$ becomes a $eH_{k}e\text{-}eH_{k'}e$-bimodule, and we have the $H_{k}\text{-}H_{k'}$-bimodule $B_{k, k'} := H_{k}e\otimes_{eH_{k}e}e_{\chi}H_{k'}e \otimes_{eH_{k'}e} eH_{k'}$. According to \cite[Lemma 3.2]{losev_derived} it is HC and it has full support. Thus, $\overline{\HC}(k, k') \neq 0$. The general result now follows from the fact that the parameters of the form $\overline{\chi}$ form a basis of $\fp_{\ZZ}$, and Lemma \ref{lemma:brown}, which ensures that tensor products of the form $B_{k, k'} \otimes_{H_{k'}} B_{k', k''}$ are nonzero.
\end{proof}

\subsection{Action on the set of Hecke parameters}\label{subsect:namiaction} In this section, we show (1) of Theorem \ref{thm:main1}. The strategy will be to use Lemma \ref{lemma:mainKZ} together with a study on how the action of the Namikawa-Weyl group behaves under passing to the Hecke parameter $q = q(k)$. Since both the action of the Namikawa-Weyl group and the computation of the Hecke parameter $q(k)$ are done by restricting to stabilizers of reflection hyperplanes, it is enough to do this in the case when $W$ is a cyclic group. 

Let us denote by $\fS_{\{2, \dots, \ell\}}$ the symmetric group on the symbols $\{2, \dots, \ell\}$. Of course, $\fS_{\{2, \dots, \ell\}}$ is isomorphic to $\fS_{\ell - 1}$, an isomorphism $\fS_{\{2, \dots, \ell\}} \rightarrow \fS_{\ell - 1}$ is given by $\sigma \mapsto \widetilde{\sigma}$, $\widetilde{\sigma}(i) = \sigma(i + 1) - 1$. The following result can be easily checked by a direct computation.

\begin{lemma}\label{lemma:namiaction}
	Let $W = \ZZ/\ell\ZZ$, so that $\nam = \fS_{\ell}$. Let $k \in \fp$, and $q(k) = \{q(k)_{0} = 1, q(k)_{1}, \dots, q(k)_{\ell - 1}\}$ be the Hecke parameter associated to $k$ via (\ref{eqn:q}). For $\sigma \in \fS_{\ell}$, the parameter $q(\sigma(k))$ is determined by the following.
	
	\begin{itemize}
		\item $q(\sigma(k))_{0} = 1$.
		\item $q(\sigma(k))_{i} = q(k)_{\widetilde{\sigma}(i)}$, if $\sigma \in \fS_{\{2, \dots, \ell\}} \subseteq \fS_{\ell}$, $i = 1, \dots, \ell - 1$. 
		\item $q(\sigma(k))_{i} = q(k)_{1}^{-1}q(k)_{i}$, if $\sigma = (12)$, $i = 1, \dots, \ell - 1$.  
	\end{itemize} 
\end{lemma}

Let us remark that the group of characters $\Hom(W, \CC^{\times})$ acts on $\fp$ in such a way that we have an isomorphism $H_{k} \cong H_{\chi(k)}$ for $\chi \in \Hom(W, \CC^{\times})$ and $k \in \fp$. Moreover, this isomorphism preserves the subalgebras $\CC[\hh]^{W}, \CC[\hh^{*}]^{W}$ and therefore it also preserves categories of HC bimodules. In terms of the parameter $c$, the action is simply given by $(\chi c)(s) = \chi(s)c(s)$. Let us describe this action in terms of the parameter $k$. Recall that we have $\Hom(W, \CC^{\times}) \cong \prod_{\Gamma \in \cA/W}\Hom(W_{\Gamma}, \CC^{\times})$ and $\fp = \bigoplus_{\Gamma \in \cA/W}\fp(W_{\Gamma})$, so it is enough to describe this action when $W$ is a cyclic group, say $W = \ZZ/\ell\ZZ = \langle s : s^{\ell} = 1\rangle$. In this case, the group of characters is identified with $\ZZ/\ell\ZZ$, $j \mapsto (s \mapsto \eta^{j})$, $\eta := \exp(2\pi\sqrt{-1}/\ell)$. We denote the character $s \mapsto \eta^{j}$ by $\chi_{j}$. Then, we have for $k = (k_{0} = 0, k_{1}, \dots, k_{\ell - 1})$, $\chi_{j}(k)_{i} = k_{i - j} - k_{\ell - j}$, where the subscripts are taken modulo $\ell$. An isomorphism $H_{k} \rightarrow H_{\chi(k)}$ is given by $x \mapsto x, y \mapsto y, w \mapsto \chi(w)w$, $x \in \hh^{*}$, $y \in \hh$, $w \in W$.

\begin{lemma}\label{lemma:namintegral}
	Let $\chi: W \rightarrow \CC^{\times}$ be a 1-dimensional character, and let $k$ be a parameter. Then, there exist $\sigma \in \nam$ and $k' \in \fp_{\ZZ}$ such that $\chi(k) = \sigma(k) + k'$.
\end{lemma}
\begin{proof}
	It is again enough to show this when $W = \ZZ/\ell\ZZ$ is a symmetric group. Assume $\chi = \chi_{j}$ for some $j = 0, \dots \ell - 1$. A direct computation shows that we have $q(\chi_{j}(k))_{i} =  q(k)_{i - j}q(k)_{\ell - j}^{-1}$.
	Thanks to Lemma \ref{lemma:namiaction}, we can find an element $\sigma \in \nam$ such that $q(\varepsilon_{j}(k))_{i} = q(\sigma(k))_{i}$ for every $i = 0, \dots, \ell - 1$. This means that $\varepsilon_{j}(k) - \sigma(k) \in \fp_{\ZZ}$. We are done.
\end{proof}

We are now in a position to prove (1) of Theorem \ref{thm:main1}.

\begin{prop}\label{prop:thm1.1part1}
	Let $k, k' \in \fp$. The following are equivalent.
	\begin{enumerate}
		\item There exists $\sigma \in \nam$ such that $\sigma(k) - k' \in \fp_{\ZZ}$.
		\item The category $\overline{\HC}(k, k')$ is nonzero.
	\end{enumerate}
\end{prop}
\begin{proof}
(1) $\Rightarrow$ (2). Thanks to Lemma \ref{lemma:naminonzero}, we may find HC bimodules with full support $B_{1} \in \HC(k, \sigma(k))$, $B_{2} \in \HC(\sigma(k), k')$. Thanks to Lemma \ref{lemma:brown}, the bimodule $B_{1} \otimes_{H_{\sigma(k)}} B_{2}$ is nonzero and has full support. Thus, $\overline{\HC}(k, k')$ is nonzero. 

(2) $\Rightarrow$ (1). Thanks to Lemma \ref{lemma:semisimple}, we may find a 1-dimensional character $\tau$ such that the bimodule $\homf{\Delta_{k'}(\triv)}{\Delta_{k}(\tau)}$ is nonzero.  Under the equivalence $\varphi_{*}: \cO_{k} \rightarrow \cO_{\tau^{-1}(k)}$, coming from the isomorphism $\varphi: H_{k} \rightarrow H_{\tau^{-1}(k)}$ we have $\varphi_{*}(\Delta_{k}(\tau)) = \Delta_{\varepsilon^{-1}(k)}(\triv)$. Thus, this implication is a consequence of Lemma \ref{lemma:namintegral} and the following result.
\end{proof}

\begin{lemma}
Assume $\homf{\Delta_{k'}(\triv)}{\Delta_{k}(\triv)} \neq 0$. Then, there exists $\sigma \in \nam$ such that $\sigma(k') - k \in \fp_{\ZZ}$.
\end{lemma}
\begin{proof}
First of all, note that a parameter $k$ is integral if and only if $k|_{W_{\Gamma}} \in \fp_{\ZZ}(W_{\Gamma})$ for every reflection hyperplane $\Gamma \in \cA$. Also, since $\Res^{W}_{W_{\Gamma}}(\Delta_{?}(\triv)) = \Delta_{?}(\triv(W_{\Gamma}))$ where $? = k, k'$, cf. the proof of Lemma \ref{lemma:semisimple}, we have that $\homf{\Delta_{k'}(\triv(W_{\Gamma}))}{\Delta_{k}(W_{\Gamma})} \neq 0$ for every $\Gamma \in \cA$. Finally, since the action of the Namikawa-Weyl group is defined by restricting to stabilizers of reflection hyperplanes, we may assume that $W$ is a cyclic group $\ZZ/\ell\ZZ$.

So assume $\homf{\Delta_{k'}(\triv)}{\Delta_{k}(\triv)} \neq 0$ and $W$ is a cyclic group. Since $\KZ_{k}(\Delta_{k}(\triv)) = \CC$, the trivial representation of $\cH_{q(k)}$, Lemma \ref{lemma:mainKZ} implies that $\cH_{q(k)'}\modd \subseteq \cH_{q(k)}\modd$ as full subcategories of $\CC[t, t^{-1}]\modd$. But $\cH_{q(k)}, \cH_{q(k)'}$ are commutative algebras of the same dimension. This implies that $\cH_{q(k)} = \cH_{q(k')}$, in other words, the numbers $q(k)_{0} = 1, q(k)_{1}, \dots, q(k)_{\ell-1}$ and $q(k')_{0} = 1, q(k')_{1}, \dots, q(k')_{\ell - 1}$ coincide up to a permutation of the indices that fixes $0$. Now the lemma is an immediate consequence of Lemma \ref{lemma:namiaction} and the definition of having integral difference. 
\end{proof}

Note that an easy consequence of Proposition \ref{prop:thm1.1part1} is that the category $\overline{\HC}(k, k')$ is nonzero if and only if the same holds for the category $\overline{\HC}(k', k)$. Similarly to the proof of Corollary \ref{cor:semisimple}, we have the following result.

\begin{cor}\label{cor:allequiv}
	Assume that the category $\overline{\HC}(k, k')$ is nonzero. Then, all categories $\overline{\HC}(k, k)$, $\overline{\HC}(k, k')$, $\overline{\HC}(k', k)$ and $\overline{\HC}(k', k')$ are equivalent. Moreover, the categories $\overline{\HC}(k, k)$ and $\overline{\HC}(k', k')$ are equivalent as monoidal categories.
\end{cor}

\subsection{Subgroup $W_{c}$}\label{subsect:subgroup} For the rest of this section, it will be more convenient to return to the \lq$c$-parametrization\rq of the rational Cherednik algebra. Of course, we still have an action of the Namikawa-Weyl group $\nam$, and every result we have proved in Subsections \ref{subsect:nami}-\ref{subsect:namiaction} remains valid. 

 Recall that the category $\overline{\HC}(c,c)$ is equivalent to the category of representations of $W/N$ for some normal subgroup $N \subseteq W$. Here, we describe the group $N$. To motivate our description, we first look at the case where $W$ is a cyclic group. 

So assume $W = \ZZ/\ell\ZZ$, with generator $s$. The Hecke algebra $\cH_{q}$ is the quotient of the polynomial algebra $\CC[T]$ by the ideal generated by the polynomial $(T - 1)\prod_{i = 1}^{\ell - 1}(T - q_{i})$. We remark that $q_{i}$ is the scalar by which $T$ acts on $\KZ(\CC_{i})$, where $\CC_{i}$ is the irreducible representation of $W$ where $s$ acts by multiplication by $\exp(2\pi\sqrt{-1}i/\ell)$. Now, if $\homf{\Delta(\triv)}{\Delta(\CC_{i})}$ is nonzero then, thanks to Lemma \ref{lemma:mainKZ}, multiplication by $q_{i}$ induces a map $q \rightarrow q$, where $q$ denotes the multiset $q = \{q_{0} = 1, q_{1}, \dots, q_{\ell - 1}\}$. It is not hard to see that this map is actually a bijection, i.e. it preserves multiplicities. In particular, $q_{i}$ is an $\ell$-root of 1.

Set $\eta := \exp(2\pi\sqrt{-1}/\ell)$. Note that the group $W$ acts on the set of Hecke parameters, the element $s^{i}$ acts on a multiset $q' = \{q'_{0}, \dots, q'_{\ell -1}\}$ by multiplying each element by $\eta^{i}$. The stabilizer of $q$, the Hecke parameter associated to the Cherednik parameter $c$, is cyclic, so it is generated by $s^{m}$, where $m$ divides $\ell$, say $mp = \ell$. By definition, $W_{c} := \langle s^{p}\rangle$. Note that for generic $c$ we have that $m = \ell$ or, equivalently, $W_{c} = W$.

Let us generalize the definition of $W_{c}$ for the case where $W$ is any complex reflection group. Fix a reflection hyperplane $\Gamma \in \cA$. Let $\eta_{\Gamma} := \exp(2\pi\sqrt{-1}/\ell_{\Gamma})$. Now consider the set $X_{\Gamma} := \{i \in \{1, \dots, \ell_{\Gamma}\} : \eta_{\Gamma}^{i}q_{\Gamma, j} \in \{q_{\Gamma, 0} = 1, \dots, q_{\Gamma, \ell_{\Gamma} - 1}\} \; \text{with the same multiplicity as} \; q_{\Gamma, j} \; \text{for every} \; j = 0, \dots, \ell_{\Gamma} - 1\}$. For example, $\ell_{\Gamma} \in X_{\Gamma}$. Now set $m_{\Gamma} := \min X_{\Gamma}$. It is clear that $m_{\Gamma}$ is a divisor of $\ell_{\Gamma}$, say $m_{\Gamma}p_{\Gamma} = \ell_{\Gamma}$. We define $W_{c} := \langle s_{\Gamma}^{p_{\Gamma}} : \Gamma \in \cA \rangle \subseteq W$. By definition, this is a reflection group. Note that the conjugation invariance of $c$ implies that $W_{c}$ is a normal subgroup of $W$. 

Note that $W_{c} = \{1\}$ if and only if $m_{\Gamma} = 1$ for every reflection hyperplane $\Gamma \in \cA$. This happens if and only if $\{q_{\Gamma, 0} = 1, \dots, q_{\Gamma, \ell_{\Gamma} - 1}\} = \{1, \eta_{\Gamma}, \dots, \eta_{\Gamma}^{\ell_{\Gamma} - 1}\}$, that is, if and only if $c \in \nam(\fp_{\ZZ})$. On the other hand $W_{c} = W$ if and only if $m_{\Gamma} = \ell_{\Gamma}$ for every $\Gamma \in \cA$, and this is a generic condition. It is also clear that $W_{c} = W_{c'}$ provided $c - c' \in \fp_{\ZZ}$.

\begin{lemma}\label{lemma:namistable}
	Let $c \in \fp$ and let $\sigma \in \nam$. Then, $W_{c} = W_{\sigma(c)}$.
\end{lemma}
\begin{proof}
It is enough to check this when $W$ is a cyclic group, so let $W = \ZZ/\ell\ZZ = \langle s: s^{\ell} = 1\rangle$, and $\eta := \exp(2\pi\sqrt{-1}/\ell)$. Assume that $W_{c} = \langle s^{p}\rangle$ with $mp = \ell$. This means that there exist $Q_{0} = 1, Q_{1}, \dots, Q_{m-1} \in \CC^{\times}$ such that
\[
q(c) = \{Q_{j}\eta^{mi} : j = 0, \dots, m-1, i = 0, \dots, p-1\} 
\]

It is enough to check that $W_{c}= W_{\sigma_{i}(c)}$ for $i = 1, \dots, \ell - 1$, where, recall, $\sigma_{i} = (i, i+1) \in S_{\ell} = \nam(W)$. Thanks to Lemma \ref{lemma:namiaction}, $q(c) = q(\sigma_{i}(c))$ as a multi-set for $i = 2, \dots, \ell - 1$, so that in this case we have $W_{c} = W_{\sigma_{i}(c)}$. For $\sigma_{1}$, Lemma \ref{lemma:namiaction} implies that there exist $j_{0} \in \{0, \dots, m-1\}, i_{0} \in \{0, \dots, p-1\}$ such that
\[
q(\sigma_{1}(c)) = \{Q_{j_{0}}^{-1}Q_{j}\eta^{m(i - i_{0}) : j = 0, \dots, m-1, i = 0, \dots, p-1}\}
\]

Setting $Q'_{j} := Q_{j_{0}}^{-1}Q_{j}$, we get that $q(\sigma_{1}(c)) = \{Q'_{j}\eta^{mi} : j = 0, \dots, m-1, i = 0, \dots, p-1\}$, so $W_{\sigma_{1}(c)} = \langle s^{p'}\rangle$ with $p'$ dividing $p$. But since $c = \sigma_{1}\sigma_{1}(c)$, we also have that $p$ divides $p'$. Since $p, p' \in \{1, \dots, \ell\}$, this implies that $W_{c} = \langle s^{p}\rangle = W_{\sigma_{1}(c)}$. 
\end{proof}

\begin{theorem}\label{thm:main full support}
The category $\overline{\HC}(c,c)$ is equivalent to the category of representations of $W/W_{c}$.
\end{theorem} 

To prove Theorem \ref{thm:main full support}, we will check that in this case there exist an element $\sigma \in \nam$ and a parameter $c' \in \fp_{\ZZ} + \sigma(c)$ such that the algebra $H_{c'}(W)$ decomposes as $W \#_{W_{c}}H_{\underline{c}}(W_{c})$, for some parameter $\underline{c} \in \CC[\cS\cap W_{c}]^{W_{c}}$ which is naturally computed from $c'$. Since $\sigma(c) - c' \in \fp_{\ZZ}$, the categories $\overline{\HC}(c,c)$ and $\overline{\HC}(c',c')$ are equivalent, cf. Lemma \ref{lemma:naminonzero} and Corollary \ref{cor:semisimple}. The result will now follow if we check that $H_{\underline{c}}(W_{c})$ has a unique irreducible HC bimodule with full support. 

\subsection{Proof of Theorem \ref{thm:main full support}}\label{subsect:proof full support} We continue to use the notation introduced in Subsection \ref{subsect:subgroup}. 

Assume, for the moment, that the parameter $c$ is such that, for $\Gamma \in \cA$, $c(s_{\Gamma}^{i}) = 0$ unless $i = p_{\Gamma}, 2p_{\Gamma}, \dots, (m_{\Gamma} -1)p_{\Gamma}$. Then, it is clear from the relations (\ref{eqn:relations}) that the $H_{c}$-subalgebra generated by $\hh, \hh^{*}$ and $W_{c}$ is isomorphic to $H_{\underline{c}}(W_{c})$, where $\underline{c}$ simply denotes the restriction of the parameter $c$ to $W_{c}$. So $H_{c}$ is generated by $H_{\underline{c}}(W_{c})$ and $W$. Moreover, the subalgebra $H_{\underline{c}}(W_{c})$ is stable under the adjoint action of $W$. It follows that $H_{c} \cong H_{\underline{c}}(W_{c})\#_{W_{c}} W$, where the latter algebra is $H_{\underline{c}}(W_{c}) \otimes_{W_{c}} \CC W$ with product defined analogously to the smash-product algebra, using the action of $W$ on $H_{\underline{c}}(W_{c})$. Thus, HC $H_{c}$-bimodules with full support correspond to $W$-equivariant HC $H_{\underline{c}}(W_{c})$-bimodules with full support, where the action of $W_{c} \subseteq W$ coincides with that coming from the inclusion $W_{c} \subseteq H_{\underline{c}}(W_{c})$. 

Let us now examine the Hecke parameters $q_{\Gamma, i}$, still under the assumption that $c(s_{\Gamma}^{i}) = 0$ unless $i = p_{\Gamma}, 2p_{\Gamma}, \dots, (m_{\Gamma} - 1)p_{\Gamma}$. It follows easily from (\ref{eqn:h}) that $k_{\Gamma, i} = k_{\Gamma, i + m_{\Gamma}}$ for all $i$. But then it follows that:

$$
q_{\Gamma, i + m_{\Gamma}} = \exp(2\pi\sqrt{-1}(k_{\Gamma, i} - i - m_{\Gamma})/\ell_{\Gamma}) = \eta_{\Gamma}^{-m_{\Gamma}}q_{\Gamma, i}
$$

Note that, given numbers $Q_{\Gamma, 0} = 1, Q_{\Gamma, 1}, \dots, Q_{\Gamma, m_{\Gamma - 1}} \in \CC^{\times}$ we can always find a parameter $c \in \CC[\cS]^{W}$ with $c(s_{\Gamma}^{i}) = 0$ unless $i$ is a multiple of $p_{\Gamma}$ and such that $q_{\Gamma, i} = Q_{\Gamma, i}$. This implies the following.

\begin{lemma}\label{lemma:reduction of c}
Let $c \in \CC[\cS]^{W}$ be a parameter of the form considered in Subsection \ref{subsect:subgroup}. Then, there exists a parameter $c' \in \CC[\cS]^{W}$ such that for every $\Gamma \in \cA$, $c'(s_{\Gamma}^{i}) = 0$ unless $i = p_{\Gamma}, 2p_{\Gamma}, \dots, (m_{\Gamma} - 1)p_{\Gamma}$ and $\cH_{q(c)} = \cH_{q(c')}$, that is, there exists $\sigma \in \nam$ such that  $\sigma(c) - c' \in \fp_{\ZZ}$, and so the categories $\overline{\HC}(c,c)$ and $\overline{\HC}(c',c')$ are equivalent.
\end{lemma} 

\begin{prooffull}
Thanks to Lemma \ref{lemma:reduction of c} we may assume that $c(s_{\Gamma}^{i}) = 0$ unless $i$ is a multiple of $p_{\Gamma}$, i.e. that $H_{c} \cong H_{\underline{c}}(W_{c})\#_{W_{c}} W$. We claim now that $H_{\underline{c}}(W_{c})$ has a unique irreducible HC bimodule with full support. For $\Gamma \in \cA$, let $\underline{q}_{\Gamma, 0} = 1, \dots, \underline{q}_{\Gamma, m_{\Gamma} - 1}$ be the parameters for the Hecke algebra $\cH_{\underline{q}}(W_{c})$ associated to $\underline{c}$, and denote $\underline{\eta}_{\Gamma} := \exp(2\pi\sqrt{-1}/m_{\Gamma}) = \eta_{\Gamma}^{p_{\Gamma}}$. We also denote $m'_{\Gamma} := \min\{i \in \{1, \dots, m_{\Gamma}\} : \underline{\eta}_{\Gamma}^{i}\underline{q}_{\Gamma, j} \in \underline{q}_{\Gamma} \; \text{with the same multiplicity as} \; \underline{q}_{\Gamma, j}$  $\; \text{for every} \; j = 0, \dots, m_{\Gamma} - 1\}$. Thanks to Lemma \ref{lemma:mainKZ}, our claim will follow if we check the following.\\

{\it Claim: For every hyperplane $\Gamma \in \cA$, $m'_{\Gamma} = m_{\Gamma}$.}\\

We proceed by contradiction. Assume there exists $0 < i < m_{\Gamma}$ such that, for every $j = 0, \dots, m_{\Gamma} - 1$, $\underline{\eta}^{i}_{\Gamma}\underline{q}_{\Gamma, j}$ is in the multiset $\underline{q}_{\Gamma}$ with the same multiplicity as $\underline{q}_{\Gamma, j}$. Note that we have 

$$
\underline{q}_{\Gamma, j} = \exp\left(\frac{2\pi\sqrt{-1}(k_{i} - i)}{m_{\Gamma}}\right) = q_{\Gamma, j}^{p_{\Gamma}}
$$

Thus, $\underline{\eta}_{\Gamma}^{i}\underline{q}_{\Gamma, j} \in \underline{q}_{\Gamma}$ implies that $\eta_{\Gamma}^{i}q_{\Gamma, j} \in \{q_{\Gamma, 0} = 1, \dots, q_{\Gamma, m_{\Gamma - 1}}\}$, with the same multiplicity as $q_{\Gamma, j}$. But $q_{\Gamma} = \{q_{\Gamma, 0}, \dots, q_{\Gamma, m_{\Gamma -1}}, \eta^{m_{\Gamma}}q_{\Gamma, 0}, \dots, \eta^{m_{\Gamma}}q_{\Gamma, m_{\Gamma} - 1}, \dots, \eta^{(p_{\Gamma} - 1)m_{\Gamma}}q_{\Gamma, 0}, \dots, \eta^{(p_{\Gamma} - 1)m_{\Gamma}}q_{\Gamma, m_{\Gamma} - 1}\}$. Thus, we see that $\eta_{\Gamma}^{i}q_{\Gamma, j} \in q$ with the same multiplicity as $q_{\Gamma, j}$ for every $j = 0, \dots, \ell_{\Gamma} - 1$. This contradicts the choice of $m_{\Gamma}$. Thus, $H_{\underline{c}}(W_{c})$ has a unique irreducible HC bimodule with full support. Since $H_{c} = H_{\underline{c}}(W_{c}) \#_{W_{c}} W$, this proves Theorem \ref{thm:main full support}.
\end{prooffull}\\

Note that Theorem \ref{thm:main1} follows immediately from Proposition \ref{prop:thm1.1part1}, Corollary \ref{cor:allequiv} and Theorem \ref{thm:main full support}.

\section{Type A.}\label{sect:A}
\subsection{Preliminary results.}\label{subsection:prelimtypeA} We now turn our attention to type A, that is, $W = \fS_n$, with reflection representation $\hh = \{(x_1, \dots, x_n) \in \CC^{n} : \sum x_i = 0\}$. Throughout this section, we denote $H_{c}(n) := H_{c}(\fS_n)$. Similary, we denote $\cH_{q}(n) := \cH_{q}(\fS_{n})$, the Hecke algebra associated to $\fS_n$ with parameter $q \in \CC^{\times}$. In this subsection, we gather some results on the structure of the algebra $H_c$ and category $\cO_{c}$. 

It is known, cf. \cite[Example 3.25]{BE}, \cite[Theorem 5.8.1]{losev_completions}, that the algebra $H_{c}:= H_{c}(n)$ is simple unless $c = r/m$ with $r, m \in \ZZ$, $\gcd(r;m) = 1$ and $1 < m \leq n$. In this case, \cite[Theorem 5.8.1 (2)]{losev_completions}, the algebra $H_{c}$ has $\lfloor n/m\rfloor$ proper nonzero two-sided ideals that are	linearly ordered by inclusion, say $\cJ_{1} \subset \cJ_{2} \subset \cdots \subset \cJ_{\lfloor n/m\rfloor}$. Moreover, $\cJ_{i}^{2} = \cJ_{i}$ for any $i = 1, \dots, \lfloor n/m\rfloor $. We set $\cJ_0 := \{0\}$, $\cJ_{\lfloor n/m\rfloor +1} := H_{c}$.

The classification of two-sided ideals gives a characterization of the possible supports of HC bimodules. For $i = 1, \dots, \lfloor n/m\rfloor$ consider the subgroup $\fS_{m}^{\times i} \subseteq \fS_n$, and consider the set $\cX_{i} := \{x \in \hh\oplus\hh^{*} : W_{x} = \fS_{m}^{\times i}\}$. Let $\cL_{i}$ be the image of $\cX_{i}$ under the natural projection $\hh\oplus\hh^{*} \rightarrow (\hh\oplus\hh^{*})/\fS_n$. This is a symplectic leaf. The support of $H_{c}/\cJ_{i}$ is $\overline{\cL}_{i}$. Now recall Lemma \ref{lemma:annihilators}, that says that for a HC $H_{c}\text{-}H_{c'}$-bimodule $B$, $\Supp(B) = \Supp(H_{c}/\lann(B)) = \Supp(H_{c'}/\rann(B))$. This implies that $\HC(c,c') = 0$ unless $c, c'$ have the same denominator when expressed as an irreducible fraction. Thus, throughout this section we will assume that $c = r/m$, $c' = r'/m$, with $\gcd(r;m) = \gcd(r';m) = 1$ and $1 < m \leq n$. 

We now give a description of the supports of irreducible modules in $\cO_{c}$. Namely, for every $i = 1, \dots, \lfloor n/m\rfloor$, let $X'_i = \{(x_1, \dots, x_n) \in \CC^{n} :\sum x_i = 0, x_1 = x_2 = \cdots = x_m, x_{m+1} = \cdots = x_{2m}, \dots, x_{(i-1)m+1} = \cdots = x_{im}\},$ and let $X_i$ be the union of the $\fS_n$ translates of $X'_i$, so that $\CC^{n-1} \supset X_1 \supset \cdots \supset X_{\lfloor n/m\rfloor}$. Then, \cite[Example 3.25]{BE}, \cite[Theorem 3.9]{wilcox}, any module in category $\cO_{c}$ is supported on one of the $X_i$. Denote by $\cO_{c}^{i}$ the full subcategory of $\cO_{c}$ consisting of all modules whose support is contained in $X_i$. Note that this is a Serre subcategory of $\cO_{c}$. Let us explain a description of the category $\cO_{c}^{i}/\cO_{c}^{i+1}$ obtained in \cite{wilcox}. Let $p := n - im$, $q := \exp(2\pi\sqrt{-1}c)$ and consider the Hecke algebra $\cH_{q}(p)$. Then \cite[Theorem 1.8]{wilcox} tells us that the category $\cO_{c}^{i}/\cO_{c}^{i+1}$ is equivalent to the category of finite dimensional modules over the algebra $\CC\fS_{i}\otimes \cH_{q}(p)$.

Let us recall how \cite[Theorem 1.8]{wilcox} is proved, as this will be important for our arguments. So let $i$ and $p$ be as in the previous paragraph. Consider the subgroup $\fS_{m}^{\times i} \subseteq \fS_{n}$. Let $\underline{\hh} := \{x \in \hh : \fS_{m}^{\times i} \subseteq \text{Stab}_{\fS_n}(x) \} (= X'_{i})$ and $\underline{\hh}^{reg} := \{x \in \hh : \text{Stab}_{\fS_n}(x) = \fS_{m}^{\times i}\}$. Then, Wilcox proves that we have a localization functor, $\Loc^i: \cO^{i} \rightarrow \cD(\underline{\hh}^{reg})\#(\fS_{i} \times \fS_{p})\text{-mod}$, $M \mapsto \CC[\underline{\hh}^{reg}] \otimes_{\CC[\hh]} M$ that factors through $\cO^{i}/\cO^{i+1}$ and that identifies this quotient category with a subcategory of the category of $(\fS_{i} \times \fS_{p})$-equivariant $\cD(\underline{\hh}^{reg})$-modules with regular singularities. Then, he checks that under the Riemann-Hilbert correspondence that identifies the latter category with the category of finite dimensional representations of $\pi_1(\underline{\hh}^{reg}/(\fS_{i}\times\fS_{p}))$, the image of $\cO_{c}^{i}/\cO_{c}^{i+1}$ gets identified with the subcategory $(\CC\fS_{i}\otimes \cH_{q}(p))$-mod of  $\pi_1(\underline{\hh}^{reg}/(\fS_{i}\times\fS_{p}))$-rep where, recall, $q = \exp(2\pi\sqrt{-1}c)$. We denote by $\KZ^i: \cO^{i} \rightarrow  (\CC\fS_{i}\otimes \cH_{q}(p))$-mod the composition of the localization functor $\Loc^i$ with the Riemann-Hilbert correspondence. 

This construction has the following consequence for HC bimodules. Let $S \in \cO_{c}$ be the irreducible module supported on $X_i$ that gets sent to the trivial $\CC\fS_{i}\otimes\cH_{q}(p)$-module under $\KZ^i$ so that, in particular, $\Loc^i(S) = \CC[\underline{\hh}^{reg}]$. Then,  the proofs in Subsection \ref{subsection:KZ} can be carried out in this setting and we see that, whenever $T$ is a simple module with $\homf{S}{T} \neq 0$, and $N$ is another simple module in $\cO^{i}$, then $\KZ^i(\homf{S}{T}\otimes_{H_c} N) = \KZ^i(S)\otimes_{\CC} \KZ^i(T)$. 

\begin{lemma}\label{lemma:4.1}
Let $S \in \cO^{i}$ be the irreducible module satisfying $\KZ^i(S) = \CC$, the trivial $\CC\fS_{i}\otimes\cH_{q}(p)$-module. Let $T \in \cO_{c}$ (necessarily supported on $X_{i}$) be a simple module satisfying $\homf{S}{T} \neq 0$. Then, for every $M \in \CC\fS_{i}\otimes\cH_{q}(p)\modd$, the $\pi_1(\underline{\hh}^{reg}/(\fS_i \times \fS_p))$-module $\KZ^i(T)\otimes_{\CC} M$ factors through $\CC\fS_{i}\otimes \cH_q(p)$.
\end{lemma}

The previous lemma gives an upper bound on the number of irreducible objects in the category $\HC_{\cL_{i}}(H_c)$, namely, since $c \not\in \ZZ$ then, using Lemma \ref{lemma:mainKZ} we see that, if $T \in \cO^{i}_{c}$ is such that $\homf{S}{T} \neq 0$, then $\KZ^{i}(T)$ has to be of the form $\lambda \otimes \CC$, where $\lambda$ is an irreducible representation of $\fS_i$ and $\CC$ stands for the trivial $\cH_q(p)$ representation.

\begin{prop}\label{prop:upperbound}
The number of irreducible HC bimodules whose support coincides with the closure of the symplectic leaf $\mathcal{L}_{i}$ is no more than $p(i)$, the number of partitions of $i$. 
\end{prop} 

We will see in Subsection \ref{subsect:allbimodA} that the bound obtained in Proposition \ref{prop:upperbound} is sharp.

\subsection{Semisimplicity of $\HC_{\cL_{i}}(H_c)$.}\label{subsection:semisimple} We have just obtained an upper bound on the number of irreducible objects of the category $\HC_{\cL_{i}}(H_c)$. In this subsection, we check that this category is semisimple. The proof is based on restriction functors, Subsection \ref{subsection:induction}. Recall that this is a functor $\bullet_{\dagger}: \HC_{\overline{\cL}_{i}}(H_{c}) \rightarrow \HC^{\Xi}_{0}(\underline{H}_{c})$ that identifies the quotient category $\HC_{\cL_{i}}(H_{c})$ with a full subcategory of the category of finite dimensional $\Xi$-equivariant HC $\underline{H}_{c}$ bimodules that is closed under taking subquotients where, recall, $\Xi = N_{W}(\underline{W})/\underline{W}$. Then, we start with a few remarks on finite dimensional bimodules.  

Recall that the algebra $H_{c}(n)$ has a finite dimensional module if and only if $c = r/n$, with $\gcd(r;n) = 1$, \cite[Theorem 1.2]{BEG2}. The unique irreducible finite dimensional $H_{c}$-module is $L_{c}(\triv)$ if $c > 0$; and it is $L_{c}(\sign)$ if $c < 0$. Moreover, the category of finite dimensional $H_{c}$-modules is semisimple, this follows either from the results of the previous subsection or from the fact that irreducibles in $\cO_{c}$ do not admit self-extensions, see e.g. \cite[Proposition 1.12]{BEG2}. Note that it follows that $H_{c}$ has a unique irreducible finite dimensional bimodule, and that this bimodule does not have self-extensions. 

We remark that a finite dimensional bimodule must be HC: given a finite dimensional bimodule $M$, for any element $x \in \CC[\hh]^{W}\cup\CC[\hh^{*}]^{W}$, there exists $n \gg 0$ such that $x^{n}M = 0 = Mx^{n}$. This follows from the existence of a grading on $M$ as a left $H_{c}$-module compatible with a grading on $H_{c}$ given by $\deg(\hh) = 1, \deg(\hh^{*}) = -1, \deg(W) = 0$ and a similar grading on $M$ as a right $H_{c}$-module. These gradings exist because the grading on $H_{c}$ is inner see e.g. \cite[Subsection 3.1]{GGOR}.

Now, assume that $c = r/m$, $m \leq n$, $\gcd(r;n) = 1$. Let $i \leq \lfloor n/m\rfloor$. Consider the subgroup $\underline{W} := \fS_{m}^{\times i} \subseteq \fS_n$. Note that $\Xi = N_{W}(\underline{W})/\underline{W}$ may be identified with $\fS_{i} \times \fS_{n - mi}$. Recall, Subsection \ref{subsection:induction}, that we have a restriction functor $\bullet_{\dagger}: \HC_{\overline{\cL}_{i}}(H_{c}) \rightarrow \HC^{\Xi}_{0}(\underline{H}_{c})$ that identifies the quotient category $\HC_{\cL_{i}}(H_{c})$ with a full subcategory of $\HC^{\Xi}_{0}(\underline{H}_{c})$ closed under taking subquotients. 

Let us apply the previous observations to the study of $\HC_{\cL_i}(H_{c})$. Note that the algebra $\underline{H}_{c}$ is isomorphic to $H_{r/m}(\CC^{m-1}, \fS_m)^{\otimes i}$. By the results above, this algebra has a unique finite dimensional bimodule, that does not have self-extensions. The subgroup $\fS_i \subseteq \Xi$ acts on $\underline{H}_{c}$ by permuting the tensor factors, and the subgroup $\fS_{n - mi}$ acts trivially. Then, the category $\HC^{\Xi}_{0}(\underline{H}_{c})$ is equivalent to the category of representations of the group $\fS_i \times \fS_{n - mi}$. Note, however, that $\fS_{n - mi}$ acts trivially on the image of the restriction functor $\bullet_{\dagger}$: this follows from the fact that simple HC bimodules supported on $\mathcal{L}_{i}$ are contained in bimodules of the form $\homf{S}{M}$, where $\KZ^{i}(S) = \CC$ and $\cH_{q}(n - mi)$ acts trivially on $\KZ^{i}(M)$, cf. Lemma \ref{lemma:4.1}. Then, we get the following result.

\begin{prop}\label{prop:semisimple}
The category $\HC_{\cL_{i}}(H_{c})$ is semisimple. Moreover, it is equivalent to the category of representations of $\fS_{i}/N$ for some normal subgroup $N \subseteq \fS_{i}$.
\end{prop}

In the next two subsections we are going to see that, in fact, $N = \{1\}$. First, we are going to do it for the case $c = r/m$, where $m$ divides $n$ and $i = n/m$. It turns out that the case $2m = n$ is important in our argument for the general case, that we explain in Subsection \ref{subsect:allbimodA}

\subsection{Bimodules with minimal support.}\label{subsect:min_supp} We give a complete description of the category of HC $H_{c}(n)$-bimodules with minimal support when the parameter $c$ has the form $c = r/m$, for $m$ a divisor of $n$. In particular, we show that in this case the normal subgroup $N$ in Proposition \ref{prop:semisimple} is trivial. Throughout this subsection, we denote $k := n/m$. For convenience, we assume that $c > 0$, we will deal with the case $c < 0$ at the end of this subsection. The following is our main result.

\begin{prop}\label{prop:min_supp}
Let $c := r/m$, where $\gcd(r;m) = 1$ and $m$ is a divisor of $n$. Let $k :=  n/m$. Then, the category $\HC_{\cL_{k}}(H_{c}(n))$ of HC $H_{c}$-bimodules with minimal support is equivalent, as a monoidal category, to the category of representations of $\fS_{k}$.
\end{prop}

The proof of Proposition \ref{prop:min_supp} will be done by induction on $r$. The proof for the case $r = 1$ is based on a symmetry result obtained in \cite{CEE}, see also \cite{EGL}. There is a symmetry of parameters for the simple quotients of spherical rational Cherednik algebras. Namely, for positive integers $n, N$ (not necessarily coprime) consider the Cherednik algebras $H_{N/n}(n)$ and $H_{n/N}(N)$, with maximal ideals $\cJ_{max}$ and $\cJ'_{max}$, respectively. Both parameters are spherical so $e\cJ_{max}e$, $e'\cJ'_{max}e'$ are the maximal ideals of the spherical Cherednik algebras $eH_{N/n}(n)e$ and $e'H_{n/N}(N)e'$, respectively. Here, $e \in\CC\fS_n$ and $e' \in \CC\fS_N$ denote the trivial idempotents in their respective group algebras. Then, by \cite[Proposition 9.5]{CEE}, \cite[Proposition 7.7]{EGL}, we have an isomorphism between the algebras $eH_{N/n}(n)e/e\cJ_{max}e$ and $e'H_{n/N}(N)e'/e'\cJ'_{max}e'$, mapping (the images of) the subalgebras $\CC[\hh_{n}]^{\fS_n}, \CC[\hh_{n}^{*}]^{\fS_n}$ to (the images of) the subalgebras $\CC[\hh_{N}]^{\fS_N}$, $\CC[\hh_{N}^{*}]^{\fS_N}$, respectively. Since HC $eH_{c}e$-bimodules with minimal support are precisely the ones whose annihilator is the maximal ideal in $eH_{c}e$, we have the following easy consequence of the above mentioned results.

\begin{prop}\label{prop:CEEiso}
The isomorphism $eH_{N/n}(n)e/e\cJ_{max}e \cong e'H_{n/N}(N)e'/e'\cJ'_{max}e'$ induces a tensor equivalence between the categories of minimally supported HC $eHe_{N/n}(n)$-bimodules and minimally supported HC $eHe_{n/N}(N)$-bimodules. 
\end{prop}

Now, the parameter $c = r/m > 0$, with $\gcd(r;m) = 1$ and $mk = n$ for $k \in \ZZ_{> 0}$, is spherical for the rational Cherednik algebra associated to $\fS_{n}$. Then, Proposition \ref{prop:CEEiso} has the following consequence.

\begin{cor}\label{cor:CEEiso}
The categories of minimally supported HC $H_{r/m}(n)$-bimodules and minimally supported HC $H_{m/r}(rk)$-bimodules are equivalent as monoidal categories.
\end{cor}

Now the case $r = 1$ of Proposition \ref{prop:min_supp} is an easy consequence of Corollary \ref{cor:CEEiso} and \cite[Theorem 8.5]{BEG}, that asserts that the category of HC $H_m(k)$-bimodules is equivalent, as a monoidal category, to the category of representations of $\fS_k$. To complete the proof of Proposition \ref{prop:min_supp} we use an inductive argument for which we will need the theory of shift functors for rational Cherednik algebras of type A, see e.g. \cite[Section 3]{gordon}. Namely, consider the $eH_{c+1}(n)e\text{-}eH_{c}(n)e$-bimodule $Q_{c}^{c+1} := eH_{c+1}(n)e_{\sign}\delta$. Here, $e_{\sign}$ denotes the sign idempotent, $e_{\sign} = \frac{1}{n!}\sum_{\sigma \in \fS_n}\sign(\sigma)\sigma$. The bimodule $Q_{c}^{c+1}$ is HC, this follows from \cite[Theorem 1.7]{GGS}. The functor $F: eH_{c}e(n)\text{-mod} \rightarrow eH_{c+1}e(n)\text{-mod}$ given by $F(M) = Q_{c}^{c+1}\otimes_{eH_{c}e} M$ is then an equivalence of categories, \cite[Corollary 4.3]{BE}. A quasi-inverse functor is given by tensoring with the $(eH_{c+1}(n)e, eH_{c}(n)e)$-bimodule $P_{c}^{c+1} := \delta^{-1}e_{\sign}H_{c+1}(n)e$, see Section 3 in \cite{gordon} (we remark that \cite{gordon} assumes that $c \not\in \frac{1}{2} + \ZZ$, an assumption that was later removed in \cite[Corollary 4.3]{BE}). The bimodule $P_{c}^{c+1}$ is also HC. It then follows that we have an equivalence of monoidal categories $F: \HC(eH_{c}(n)e) \rightarrow \HC(eH_{c+1}(n)e)$, $F(V) = Q_{c}^{c+1}\otimes_{eH_{c}(n)e} V \otimes_{eH_{c}(n)e} P_{c}^{c+1}$. Clearly, this equivalence preserves the filtrations of the categories of HC bimodules by the support. 

We now proceed to finish the proof of Proposition \ref{prop:min_supp}. So let $r, m, n, k$ be as in the statement of that proposition. We work over spherical subalgebras, and we make the following inductive assumption: \\

{\it For every $0 < r' < r$ and every $m', k' \in \ZZ_{> 0}$ with $\gcd(r', m') = 1$, the category $\HC_{\cL_{k'}}(eH_{r'/m'}(m'k')e)$ is equivalent, as a monoidal category, to the category of representations of $\fS_{k'}$.} \\

Clearly, Proposition \ref{prop:CEEiso}, together with \cite[Theorem 8.5]{BEG}, give the base of induction. Now, using Proposition \ref{prop:CEEiso} again, we have that the categories $\HC_{\cL_{k}}(eH_{r/m}(n)e)$ and $\HC_{\cL_{k}}(e'H_{m/r}e'(kr))$ are equivalent as monoidal categories. Using shift functors, we get a tensor equivalence between $\HC_{\cL_{k}}(eH_{r/m}(n)e)$ and $\HC_{\cL_{q}}(e'H_{r'/r}(kr)e')$, where $0 < r' < r$. By our inductive assumption, this is tensor equivalent to $\fS_{k}\!\repp$. Proposition \ref{prop:min_supp} now follows by sphericity, since we are assuming our parameter $c$ is positive. 

Let us give a description of the irreducible objects in $\HC_{\cL_{k}}(H_{r/m}(n))$. First of all, the irreducible modules in $\cO_{r/m}^{k}$ have the form $L(m \lambda)$, where $\lambda$ is a partition of $k$. The irreducible module that gets sent to the trivial $\fS_{k}$-representation under $\KZ^{k}$ is $L(m\triv_{k}) = L(\triv)$, where $\triv_{k}$ stands for the trivial partition of $k$ and $\triv = m\triv_k$ is the trivial partition of $n$. The localization $\Loc^{k}$ of the bimodule $\homf{L(\triv)}{L(m\lambda)}$ is irreducible, so each one of $\homf{L(\triv)}{L(m\lambda)}$ is either irreducible or $0$. But a HC bimodule $B \in \HC_{\cL_{k}}(H_{r/m}(n))$ is a Noetherian bimodule over the simple algebra $H_{c}/\cJ_{k}$, so it is a progenerator in the category of left and right modules over this algebra, cf. \cite[Theorem 10]{Br}. In particular, $B \otimes_{H_{c}} L(\triv) \neq 0$, so any irreducible HC bimodule with minimal support embeds in a bimodule of the form $\homf{L(\triv)}{L(m\lambda)}$. By counting, it follows that $\{\homf{L(\triv)}{L(m\lambda)} : \lambda \; \text{is a partition of} \; q\}$ is a complete list of irreducible HC $H_{r/m}(n)$-bimodules with minimal support. An explicit tensor equivalence is given as follows, $B = \homf{L(\triv)}{L(m\lambda)} \mapsto \KZ^{k}(B \otimes_{H_{r/m}(n)} L(\triv))$. That this functor intertwines tensor products follows by an analog of Lemma \ref{lemma:mainKZ}, using the functor $\KZ^{k}$ instead of $\KZ$.

\begin{rmk}\label{rmk:all_locfin_notzero}
We remark that, while $\{\homf{L(\triv)}{L(m\lambda)} : \lambda \vdash k\}$ forms a complete and irredundant list of irreducible HC bimodules with minimal support, we have that $\homf{L(m\mu)}{L(m\lambda)} \neq 0$ where $\lambda, \mu$ are any partitions of $k$. This follows from, for example, Theorem 8.16 in \cite{BEG}, which gives a description of $\homf{L(m\mu)}{L(m\lambda)}$ as a direct sum of bimodules of the form $\homf{L(\triv)}{L(m\xi)}$. 
\end{rmk}

To finish this subsection, let us explain what happens when we have $c = r/m < 0$, with $\gcd(r;m) = 1$ and $m$ divides $n$, say $km = n$. In this case, the category $\HC_{\cL_{k}}(n)$ is also equivalent to the category of representations of $\fS_{k}$. This follows because there is an equivalence $\HC_{\cL_{k}}(H_{r/m}(n)) \cong \HC_{\cL_{k}}(H_{-r/m}(n))$ induced by an isomorphism $H_{c}(n) \rightarrow H_{-c}(n)$, mapping $\hh^* \ni x \mapsto x$, $\hh \ni y \mapsto y$, $\fS_n \ni \sigma \mapsto \sign(\sigma)\sigma$.

\subsection{Irreducible HC bimodules}\label{subsect:allbimodA}  We use the results of the previous subsection and Section \ref{sect_reduction} to give a classification of all irreducible HC $H_{c}(n)$-bimodules where, as above, we assume that $c$ has the form $c = r/m > 0$, with $1 < m \leq n$ and $\gcd(r;m) = 1$. The following is the main result of this subsection.

\begin{theorem}\label{thm:maintypeA}
Let $c = r/m > 0$, with $1 < m \leq n$, $\gcd(r;m) = 1$, and let $i = 1, \dots, \lfloor n/m \rfloor$. Then, the category $\HC_{\cL_{i}}(H_{c}(n))$ is equivalent to the category of representations of $\fS_{i}$.
\end{theorem}

Before proceeding to the proof of Theorem \ref{thm:maintypeA} we describe the objects in the category $\HC^{\Xi}_{0}(\underline{H}_{c})$. Recall that this category is equivalent to the category of representations of $\Xi = \fS_{n - mi} \times \fS_{i}$, this follows because the algebra $\underline{H}_{c}$ has a unique irreducible finite dimensional bimodule (that does not admit non-trivial self-extensions). This bimodule is $\underline{B} := \Hom_{\CC}(L(\triv_{\underline{W}}), L(\triv_{\underline{W}}))$, this is a consequence of the fact that $L(\triv_{\underline{W}})$ is the unique irreducible finite dimensional module over the algebra $\underline{H}_{c}$. Moreover, since $c = r/m$ and $\underline{W} = \fS_{m}^{\times i}$, we have that $\underline{H}_{c} = H_{c}(m)^{\otimes i}$, and $\underline{B} = B^{\otimes i}$, where $B$ is the unique irreducible finite dimensional bimodule over $H_{c}(m)$, so $\underline{B}$ admits a $\Xi$-equivariant structure, where $\fS_{i}$ permutes the tensor factors and $\fS_{n - mi}$ acts trivially. Under the equivalence $\HC^{\Xi}_{0}(\underline{H}_{c}) \rightarrow (\fS_{i} \times \fS_{n - mi})\!\repp$, $\underline{B}$ corresponds to the trivial representation. So we have the following result.

\begin{lemma}
The irreducible objects in $\HC^{\Xi}_{0}(\underline{H}_{c})$ have the form $\underline{B} \otimes \xi$, where $\xi$ runs over the set of irreducible representations of $\fS_{i} \times \fS_{n - mi}$, which acts diagonally. The irreducibles where $\fS_{n - mi}$ acts trivially correspond precisely to those representations $\underline{B} \otimes \xi$ where $\xi$ factors through $\fS_{i}$.
\end{lemma}

\begin{proofmainA}
We need to check that, if $\xi$ is an irreducible representation of $\fS_{i}$, then the equivariant bimodule $\underline{B} \otimes \xi$ belongs to the image of $\bullet_{\dagger^{W}_{\underline{W}}}$. By Theorem \ref{thm:veryverytechnical}, for every parabolic subgroup $W'$ containing $\underline{W}$ in corank 1 we need to produce a HC $H_{c}(W')$-bimodule $B'$ with $B'_{\dagger^{W'}_{\underline{W}}} = \underline{B} \otimes \xi$, with the restricted $N_{W'}(\underline{W})/\underline{W}$-equivariant structure. The subgroups $W'$ have three different types. Either $W' \cong \fS_{m}^{\times(i - 2)} \times \fS_{2m}$, $W' \cong \fS_{m}^{\times (i - 1)} \times \fS_{m + 1}$ or $W' \cong \fS_{m}^{\times i} \times \fS_{2}$. \\

{\bf Case 1.} $W' \cong \fS_{m}^{\times(i - 2)}\times \fS_{2m}$. So that $N_{W'}(\underline{W})/\underline{W} \cong \fS_{2}$ acting on $\underline{H}_{c} = H_{r/m}(m)^{\otimes i}$ by permuting two of the tensor factors. Thanks to the results of Subsection \ref{subsect:min_supp} the functor $\bullet_{\dagger^{W'}_{\underline{W}}}: \HC_{\cL}(H_{c}(W')) \rightarrow \HC^{\fS_{2}}_{0}(\underline{H}_{c})$ is essentially surjective, were $\HC_{\cL}(H_{c}(W'))$ denotes the category of minimally supported HC $H_{c}(W')$-bimodules. So we can certainly find a bimodule $B'$ with $B'_{\dagger^{W'}_{\underline{W}}} = \underline{B} \otimes \xi$.  \\

{\bf Case 2.} $W' \cong \fS_{m}^{\times (i - 1)} \times \fS_{m+1}$, so that $N_{W'}(\underline{W})/\underline{W} \cong \{1\}$. Thus, what we have to check here is that $\underline{B}$ belongs to the image of the functor $\bullet_{\dagger^{W'}_{\underline{W}}}$. But this follows because the image of $\bullet_{\dagger^{W'}_{\underline{W}}}$ is closed under sub-bimodules. \\

{\bf Case 3.} $W' \cong \fS_{m}^{i} \times \fS_{2}$, so that $N_{W'}(\underline{W})/\underline{W} \cong \fS_{2}$, acting trivially on $\underline{H}_{c}$. Thanks to our assumptions on $\xi$, $\fS_{2}$ acts trivially on $\xi$. It also acts trivially on $\underline{B}$. So we need to check that $\underline{B}$, with trivial action of $\fS_{2}$, belongs to the image of $\bullet_{\dagger^{W'}_{\underline{W}}}$. Upon the identification $\HC^{\fS_{2}}_{0}(\underline{H}_{c}) \rightarrow \fS_{2}\!\repp$, $\underline{B}$ corresponds to the trivial representation. The image of the restriction functor is closed under tensor products and sub-bimodules. Since the trivial representation of $\fS_{2}$ is contained in $S^{\otimes 2}$ for any representation $S$ of $\fS_{2}$, the result follows.
\end{proofmainA}

\subsection{Case $c = r/n$.}\label{subsect:c=r/n} In this subsection, we completely characterize the category of HC bimodules over the algebra $H_{r/n}(n)$, where $\gcd(r;n) = 1$. By Theorem \ref{thm:main full support}, this algebra has a unique irreducible HC bimodule with full support, namely, the unique nonzero proper ideal $\cJ \subsetneq H_c$. By Subsection \ref{subsection:semisimple}, this bimodule does not have self-extensions. On the other hand, this algebra has a unique irreducible finite dimensional bimodule, namely $M := H_c/\cJ$, that does not admit self-extensions. The bimodules $M, \cJ$ form a complete list of irreducible HC bimodules. We now investigate extensions between them. For the rest of this section we denote simply by \lq$\Ext$\rq \; the extension group $\Ext^{1}_{H_c\text{-bimod}}$. 
It is clear that $\Ext(M, \cJ) \neq 0$, as $H_c$ is a nonsplit extension of $\cJ$ by $M$. On the other hand, \cite[Subsection 7.6]{BL} constructs a nonsplit extension $D$ of $M$ by $\cJ$. Our goal now is to show that $M, \cJ, H_c$ and $D$ form a complete list of indecomposable HC $H_c$-bimodules. This is a consequence of the following result. 

\begin{prop}\label{prop:ext=0}
The following is true:
\begin{enumerate}
\item[(i)] $\Ext(H_c, M) = 0$.
\item[(ii)] $\Ext(M, H_c) = 0$.
\item[(iii)] $\dim(\Ext(M, \cJ)) = 1$.
\item[(iv)] $\Ext(H_c, \cJ) = 0$.
\item[(v)] $\Ext(\cJ, H_c) = 0$. 
\item[(vi)] $\Ext(M, D) = 0$.
\item[(vii)] $\Ext(D, \cJ) = 0$.
\item[(viii)] $\Ext(D, M) = 0$.
\item[(ix)] $\Ext(\cJ, D) = 0$.
\item[(x)] $\dim(\Ext(\cJ, M)) = 1$.
\end{enumerate}
\end{prop}
\begin{proof}
We show that (i) holds more generally. Namely, we have the following result.

\begin{lemma}\label{minsupp}
Let $H_c$ be any rational Cherednik algebra of type A (we do not put restrictions on the parameter $c$), and let $M$ be an irreducible Harish-Chandra $H_c$-bimodule with minimal support. Then, $\Ext(H_c, M) = 0$.
\end{lemma}
\begin{proof}
We know that $\Ext^{\bullet}(H_c, M) = \HH^{\bullet}(H_c, M)$, where $\HH^{\bullet}$ denotes Hochschild cohomology, so we need to compute $\HH^{1}(H_c, M)$. It is well known that this is the space of outer derivations (i.e. the space of derivations modulo the space of inner derivations). Now, let $\delta: H_c \rightarrow M$ be a derivation. Since $\cJ^2 = \cJ$, Subsection \ref{subsection:prelimtypeA}, the Leibniz rule implies that $\delta(\cJ) = 0$, so $\delta$ factors through the quotient algebra $H_c/\!\cJ$. This implies that $\HH^{1}(H_c, M) = \HH^{1}(H_c/\!\cJ, M)$ (note that $M$ is an $H_c/\!\cJ$-bimodule since RAnn$(M) =$LAnn$(M) = \cJ$, so this last Hochschild cohomology does make sense). Now, both $H_c/\!\cJ$ and $M$ are irreducible HC bimodules with minimal support. Recall, Subsection \ref{subsection:semisimple}, that the category of HC bimodules with minimal support is semisimple. Then, $\Ext(H_c/\!\cJ, M) = 0$, which implies that $\HH^{1}(H_c/\!\cJ, M) = \Ext_{H_c/\!\cJ\text{-bimod}}(H_c/\!\cJ, M) = 0$.  
\end{proof}

Then (i) is a special case of Lemma \ref{minsupp}. Note that (ii) and (v) are consequences of Proposition \ref{prop:injective}. Now (iii) is a consequence of (ii): we have a long exact sequence
$$
0 \rightarrow \Hom(M, \cJ) \rightarrow \Hom(M, H_c) \rightarrow \Hom(M, M) \rightarrow \Ext(M, \cJ) \rightarrow \Ext(M, H_c) \rightarrow \cdots
$$

Now, both $\Hom(M, H_c)$ and $\Ext(M, H_c)$ are $0$, so $\Hom(M, M) \rightarrow \Ext(M, \cJ)$ must be an isomorphism and the claim follows. Again using long exact sequences, we can see that $\dim(\Ext(H_{c}, \cJ)) = -\dim(\Hom(\cJ,\cJ)) + \dim(\Ext(M, \cJ)) = 0$, so (iv) is proved.  Statements (vi), (ix) are consequences of the following.

\begin{prop}\label{prop:doublewallcrossing_injective}
Assume $c = r/n$, with $\gcd(r;n) = 1$.The bimodule $D$ is injective in the category of HC $H_{c}(n)$-bimodules.
\end{prop}
\begin{proof}
We remark that we have functors $F: \HC(c,c) \rightarrow \cO_{c}, B \mapsto B \otimes_{H_{c}}\Delta_{c}(\triv)$, $G: \cO_{c} \rightarrow \HC(c,c), M \mapsto \homf{\Delta_{c}(\triv)}{M}$. Since $c > 0$, $\Delta_{c}(\triv)$ is projective in $\cO_{c}$, so thanks to \cite[Lemma 3.9]{losev_derived} the functor $F$ is exact. Note also that $F$ is left adjoint to $G$, so $G$ maps injective objects to injective objects. Thus, the lemma will follow if we find an injective $N \in \cO_{c}$ with $D = \homf{\Delta_{c}(\triv)}{N}$. Thanks to \cite[Theorem 1.3]{BEG2}, $\Delta_{c}(\triv)$ has a unique proper nonzero submodule, say $I$. We claim that $\homf{\Delta(\triv)}{I} = \cJ$, this follows because $I = \cJ\Delta_{c}(\triv)$ and Corollary \ref{cor:Hc_loc_fin_maps}. We also have that $\homf{\Delta_{c}(\triv)}{L_{c}(\triv)} = \Hom_{\CC}(L_{c}(\triv), L_{c}(\triv)) = H_{c}/\cJ = M$. 

The \emph{costandard} module $\nabla_{c}(\triv)$ is injective in category $\cO_{c}$, it has a unique proper submodule isomorphic to $L_{c}(\triv)$ and $\nabla_{c}(\triv)/L_{c}(\triv) \cong I$, all of these properties follow from the construction of $\nabla_{c}(\triv)$, see e.g. \cite[Subsection 2.3]{GGOR}. So $G(\nabla_{c}(\triv))$ is injective and contains $\homf{\Delta_{c}(\triv)}{L_{c}(\triv)} = M$. It follows that we have an injection $D \hookrightarrow G(\nabla_{c}(\triv))$. Note, however, that we have an exact sequence $0 \rightarrow G(L_{c}(\triv)) \rightarrow G(\nabla_{c}(\triv)) \rightarrow G(I)$. Thanks to the previous paragraph, we conclude that the composition length of $G(\nabla_{c}(\triv))$ is $\leq 2$. So $D \cong \homf{\Delta_{c}(\triv)}{\nabla_{c}(\triv)}$ and is therefore injective.
\end{proof}

\begin{rmk}
It is worth noticing that the category $\HC(H_{c}(W))$ has enough injectives for \emph{any} complex reflection group $W$ and parameter $c$. Indeed, let $P_{c}$ be a progenerator of the category $\cO_{c}$. Thanks to Lemma 3.9 in \cite{losev_derived} the functor $F: \HC(H_{c}(W)) \rightarrow \cO_{c}$, $B \mapsto B \otimes_{H_{c}}P_{c}$ is exact. Moreover, $F$ admits a right adjoint $G: \cO_{c} \rightarrow \HC(H_{c})$, $M \mapsto \homf{P_{c}}{M}$. So $G$ has to map injectives to injectives. Thanks to \cite[Lemma 3.10]{losev_derived}, every irreducible HC $H_{c}$-bimodule is contained in one of the form $G(M)$ for some $M \in \cO_{c}$. This implies that there are enough injectives in $\HC(H_{c})$. When $W = \fS_{n}$ and $c > 0$, we can replace $P_{c}$ by $\Delta_{c}(\triv)$. This follows because $\Delta_{c}(\triv)$ is projective and the results in Subsection \ref{subsection:semisimple}, that imply that every irreducible HC bimodule is contained in one of the form $\homf{\Delta_{c}(\triv)}{M}$. 
\end{rmk}

Now we show that $\Ext(D, \cJ) = 0$. Assume we have a short exact sequence 

  \begin{equation}\label{xx}
  0 \rightarrow \cJ \rightarrow X \buildrel \pi \over \rightarrow D \rightarrow 0
  \end{equation}
  
  Consider the induced exact sequence $0 \rightarrow \cJ \rightarrow \pi^{-1}(M) \rightarrow M \rightarrow 0$. So either $\pi^{-1}(M) = H_c$ or $\pi^{-1}(M) = \cJ \oplus M$. If $\pi^{-1}(M) = H_c$, then the exact sequence $0 \rightarrow \pi^{-1}(M) \rightarrow X \rightarrow \cJ \rightarrow 0$ gives $X = H_c \oplus \cJ$, cf. (iv), which contradicts the existence of the exact sequence (\ref{xx}). Then, we must have $\pi^{-1}(M) = \cJ \oplus M$. Using again the exact sequence $0 \rightarrow \pi^{-1}(M) \rightarrow X \rightarrow J \rightarrow 0$, we get that $X = \cJ \oplus V$, where $V$ is an extension of $M$ by $\cJ$. Then (\ref{xx}) forces $V = D$ and the sequence splits. 
 
The proof of (viii) is similar: say that we have a short exact sequence

\begin{equation}\label{xxx}
0 \rightarrow M \rightarrow X \rightarrow D
\end{equation}

So we see that $\Soc(X) = M \oplus M$ and we have an exact sequence

$$
0 \rightarrow M \oplus M \rightarrow X \rightarrow \cJ \rightarrow 0
$$

An extension of $M \oplus M$ by $\cJ$ must be of the form $B \oplus M$, where $B$ is an extension of $M$ by $\cJ$. Using the short exact sequence (\ref{xxx}) we see that $X \cong D \oplus M$. Finally, (x) is an easy consequence of the previous statements.
\end{proof}

Note that the previous proposition implies that both $H_{c}$ and $D$ are injective-projective in the category $\HC(c,c)$. The injective hull of $\cJ$ coincides with the projective cover of $M$, which is $H_{c}$, while $D$ is both the injective hull of $M$ and the projective cover of $\cJ$. It follows, in particular, that the homological dimension of $\HC(c,c)$ is infinite.

\begin{rmk}
The results of this subsection are also valid for parameters of the form $c = r/m$, with $\gcd(r;m) = 1$ and $\lfloor n/2 \rfloor < m \leq n$. Indeed, here we also have two irreducible HC bimodules $M$ and $\cJ$, where $\cJ$ has full support and $M$ has minimal support. Proposition \ref{prop:ext=0} is valid with the same proof, so in this case the category $\HC(c,c)$ is equivalent to the category of representations of the quiver

$$
\xymatrix{\bullet \ar@/^/[rr]^{\alpha} & & \bullet \ar@/^/[ll]^{\beta}}
$$

\noindent with relations $\alpha\beta = \beta\alpha = 0$.
\end{rmk}

\subsection{Two-parametric case.}\label{subsect:2paramTypeA} We study the category $\HC(c,c')$ when the parameters $c, c'$ are distinct. First, we remark that if $c$ is a regular parameter (this means that $c$ is not of the form $r/m$ with $0 < m \leq n$) then $\HC(c,c') = 0$ unless $c' = \pm c + m$ with $m \in \ZZ$ and, in this case, $\HC(c,c')$ has been completely described, Theorem \ref{thm:main1}. So we may assume that both parameters $c, c'$ are singular. Since for irreducible modules $M \in \cO_{c'}$, $N \in \cO_{c}$, $\homf{M}{N} \neq 0$ only when $\supp(M) = \supp(N)$, the description of supports of irreducible modules given in Subsection \ref{subsection:prelimtypeA} implies that a necessary condition for $\HC(c,c')$ to be nonzero is that $c$ and $c'$ have the same denominator when expressed as irreducible fractions. Then, throughout this subsection we assume that $c = r/m$, $c' = r'/m$, $\gcd(r;m) = \gcd(r';m) = 1$, $1 < m \leq n$.

Recall that, for $i = 1, \dots, \lfloor n/m \rfloor$ we have the functor $\KZ_{c'}^{i}: \cO^{i}_{c'} \rightarrow (\CC\fS_{i} \otimes \cH_{q'}(\fS_{n - mi}))\text{-mod}$, where $q' = \exp(2\pi\sqrt{-1}c')$. Let $N \in \cO^{i}_{c'}$ be the irreducible module with $\KZ_{c'}^{i}(N) = \triv$. Then, similarly to Section \ref{sect:localization}, we have that every irreducible HC $H_{c}(n)\text{-}H_{c'}(n)$-bimodule supported on the closure of the symplectic leaf $\cL_{i}$ is contained in a bimodule of the form $\homf{N}{M}$ for an irreducible module $M \in \cO_{c}^{i}$ and, moreover, that whenever $\homf{N}{M}$ is nonzero then, for every module $L \in (\CC\fS_{i} \otimes \cH_{q}(n))\text{-mod}$ the $B_{i} \times B_{n - mi}$-module $\KZ^{i}(M) \otimes_{\CC} L$ factors through the algebra $\CC\fS_{i} \otimes \cH_{q}(n - mi)$. The following result is then completely analogous to Proposition \ref{prop:upperbound}. 

\begin{prop}\label{prop:twoparam}
Let $i \in \{1, \dots, \lfloor n/m \rfloor\}$ and assume that $n - mi \neq 0$. Then, $\HC_{\cL_i}(c,c') = 0$ unless $c - c' \in \ZZ$ or $c + c' \in \ZZ$.
\end{prop}

Now assume that $c' = c + k$, with $c > 0$ and $k \in \ZZ_{> 0}$. Then, using shift functors we have an equivalence of categories $\HC(c,c) \cong \HC(c,c') \cong \HC(c',c')$ preserving the filtration by supports so that, in particular, they descend to equivalences $\HC_{\mathcal{L}_i}(c,c) \cong \HC_{\mathcal{L}_i}(c,c') \cong \HC_{\mathcal{L}_i}(c',c')$. Since we have an isomorphism $H_{c'}(n) \rightarrow H_{-c'}(n)$ fixing the subalgebras $\CC[\hh]^{\fS_n}, \CC[\hh^{*}]^{\fS_n}$, we also have an equivalence $\HC(c,-c-k) \cong \HC(c,c)$ preserving the filtration by supports. Similar results hold if $c < 0$ and $k \in \ZZ_{< 0}$.

Assume now that $c = c'+1$, with $-1 < c' < 0$. In this case, the shift functor is not an equivalence. However, it does induce a derived equivalence $_{c}P_{c'} \otimes^{L}_{H_{c'}} \bullet: D^{b}(\cO_{c'}) \rightarrow D^{b}(\cO_{c})$, see e.g. \cite[Section 5]{gordon-losev}. It follows that, if we denote by $D^{b}_{\HC}(c',c)$ the subcategory of $D^{b}(H_{c}\text{-}H_{c'}\text{-bimod})$ consisting of complexes with HC homology, then we have a derived equivalence $_{c}P_{c'} \otimes^{L}_{H_{c}} \bullet: D^{b}_{\HC}(c',c) \rightarrow D^{b}_{\HC}(c,c)$. Thus, the categories $\HC(c,c)$ and $\HC(c,c')$ have the same number of irreducibles. We remark here that Proposition \ref{prop:semisimple} is valid in the two-parametric setting, with the same proof. Hence, $\HC_{\cL_{i}}(c,c') \cong \fS_{i}\!\repp$ for $i = 1, \dots, \lfloor n/m \rfloor$. The same holds for the category $\HC_{\cL_{i}}(c',c)$. 

Let us finalize with the case that is not covered by Proposition \ref{prop:twoparam}, that is, minimally supported HC $H_{r/m}(n)\text{-}H_{r'/m}(n)$-bimodules where $m$ divides $n$, say $\ell m = n$. Note that Proposition \ref{prop:twoparam} is not valid anymore. As an easy example,  if $c = r/n, c' = r'/n$, $\gcd(r; n) = \gcd(r';n) = 1$, then $\HC_{0}(H_{c}, H_{c'}) \neq 0$, so it does not matter whether $c + c'$ or $c - c'$ are integers. 

So assume $m$ divides $n$, say $n = m\ell$, $1 < m \leq n$. Let $c = r/m, c'=r'/m$ as irreducible fractions. We have the following result.

\begin{prop}\label{prop:finally}
	The category $\HC_{\cL_{\ell}}(H_{c}(n), H_{c'}(n))$ is equivalent to the category of representations of $\fS_{\ell}$.
\end{prop}

Note that in Proposition \ref{prop:finally} we do not impose any other conditions on $c$ and $c'$. We just require that they are expressed as irreducible fractions with the same denominator which is a factor of $n$, with quotient $\ell$. 

\begin{proof}
	We proceed in several steps.
	
	{\it Step 1.} We remark that, using the isomorphisms $H_{c}(n) \rightarrow H_{-c}(n)$, we may assume that both $c, c'$ are positive. Moreover, using shift functors, we may assume that $0 < c, c' < 1$.  So we have $0 < r, r' < m \leq n$. Since both $c$ and $c'$ are positive, they are spherical. So we can work over the spherical subalgebras $eH_{c}(n)e, eH_{c'}(n)e$. To simplify the notation, we will denote the spherical subalgebras by $A_{c}(n) := eH_{c}(n)e$, $A_{c'}(n) := eH_{c'}(n)$.
	
	{\it Step 2.} Let us introduce the following notation. For a positive integer $N$, set $R_{N} = \{(z_{1}, \dots, z_{N}) \in \CC^{N} : \sum_{i = 1}^{N} z_{i} = 0\}$. This is, of course, the reflection representation of $S_{N}$. Let $x_{1}, \dots, x_{N}$ be the coordinate functions on $\CC^{N}$, and for a positive integer $k$ let $p_{k, N}(x) = x_{1}^{r} + \cdots + x_{N}^{r}$. So the invariant algebra $\CC[R_{N}]^{S_{N}}$ is generated by $p_{2, N}(x), \dots, p_{N, N}(x)$. Similarly, the invariant algebra $\CC[R_{N}^{*}]^{S_{N}}$ is generated by $p_{2,N}(y), \cdots, p_{N, N}(y)$. 
	
	{\it Step 3.} For $c > 0$, let us denote by $\oA_{c}(n)$ the quotient of the spherical subalgebra $A_{c}(n)$ by its unique maximal ideal. Recall the isomorphism $\varphi_{N, M}: \oA_{M/N}(N) \rightarrow \oA_{N/M}(M)$ that we have already used in Subsection \ref{subsect:min_supp}. It is known that $\varphi_{N, M}$ maps $p_{k, N}(x) \mapsto (N/M)p_{k, M}(x)$, while $p_{k, N}(y) \mapsto (M/N)^{k-1}p_{k, M}(y)$. Both of these assertions follow from \cite[Section 8]{CEE}, see also \cite[Section 7]{EGL}. 
	
	{\it Step 4.} In Steps 4-8 we are going to produce an equivalence from the category of minimally supported bimodules $\HC_{\cL_{\ell}}(A_{c}(n), A_{c'}(n))$ to $\HC(A_{N_{1}}(\ell), A_{N_{2}}(\ell))$ for some integers $N_{1}, N_{2} \in \ZZ_{> 0}$, Proposition \ref{prop:finally} follows from here. First of all note that, since we are taking bimodules with minimal support, $\HC_{\cL_{\ell}}(A_{c}(n), A_{c'}(n)) = \HC_{\overline{\cL}_{\ell}}(A_{c}(n), A_{c'}(n))$, so we are going to think of objects in $\HC_{\cL_{\ell}}(A_{c}(n), A_{c'}(n))$ as honest bimodules. So let $B \in \HC_{\cL_{\ell}}(A_{r/m}(n), A_{r'/m}(n))$. Since $B$ has minimal support it is, in particular, an $\oA_{r/m}(n)\text{-}\oA_{r'/m}(n)$-bimodule. Using the isomorphisms $\varphi_{N, \ell r}, \varphi_{N, \ell r'}$, we may think of $B$ as an $\oA_{m/r}(\ell r)\text{-}\oA_{m/r'}(\ell r')$-bimodule, equivalently, as an $A_{m/r}(\ell r)\text{-}A_{m/r'}(\ell r')$-bimodule whose left (resp. right) annihilator coincides with the maximal ideal of $A_{m/r}(\ell r)$ (resp. of $A_{m/r'}(\ell r')$). By Step 3, the following operators act locally nilpotently on $B$:
	\begin{gather}
	a_{k}(x): b \mapsto (m/r)p_{k, \ell r}(x) b - (m/r')bp_{k, \ell r'}(x) \label{eqn:adx}\\
	d_{k}(y) : b \mapsto (r/m)^{k-1}p_{k, \ell r}(y) b - (r'/m)^{k-1}bp_{k, \ell r'}(y) \label{eqn:ady}
	\end{gather}
	
	{\it Step 5.} Now let $m = rk_{1} + m_{1}$, with $0 \leq m_{1} < r$, $k_{1} \in \ZZ$. So we have the shift $A_{m_{1}/r}(\ell r)\text{-}A_{m/r}(\ell r)$-bimodule, say $P_{m_{1}/r, m/r}(\ell r)$. Consider $B' := P_{m_{1}/r, m_{r}}(\ell r) \otimes_{A_{m/r}(\ell r)} B$, which is an $A_{m_{1}/r}(\ell r)\text{-}A_{m/r'}(\ell r')$-bimodule . We claim that the operators (\ref{eqn:adx}), (\ref{eqn:ady}) act locally nilpotently on $B'$. For (\ref{eqn:adx}), this follows because $a_{k}(x)(b_{1} \otimes b_{2}) = \ad((m/r)p_{k, \ell r}(x))(b_{1}) \otimes b_{2} - b_{1} \otimes a_{k}(x)(b_{2})$ and the operator $\ad((m/r)p_{k, \ell r}(x))$ acts locally nilpotently on $P_{m_{1}/r, m_{r}}$. The reasoning for (\ref{eqn:ady}) is the same. Now let $m' = r'k_{1}' + m_{1}'$ be division with remainder, and consider the shift $A_{m/r'}(\ell r')\text{-}A_{m_{1}'/r'}(\ell r')$-bimodule $P_{m/r', m_{1}'/r'}(\ell r')$. Let $B_{1} := B' \otimes_{A_{m/r'}(\ell r')} P_{m/r', m_{1}'/r'}(\ell r')$. This is an $A_{m_{1}/r}(\ell r)\text{-}A_{m_{1}'/r'}(\ell r')$-bimodule. It is clear that its left (resp. right) annihilator is the maximal ideal in $A_{m_{1}/r}(\ell r)$ (resp. $A_{m_{1}'/r'}(\ell r')$), that it is finitely generated as a left or right module, and that the operators (\ref{eqn:adx}), (\ref{eqn:ady}) act locally nilpotently on $B_{1}$. 
	
	{\it Step 6.} Now we use the isomorphisms $\varphi_{\ell r, \ell m_{1}}$ and $\varphi_{\ell r', \ell  m_{1}'}$ to view $B_{1}$ as an $A_{r/m_{1}}(\ell m_{1})\text{-}A_{r'/m_{1}'}(\ell m_{1}')$-bimodule. Note that the operators that act locally nilpotently on $B_{1}$ now are
	\begin{gather}
	a_{k}^{1}(x): b \mapsto (m/m_{1})p_{k, \ell m_{1}}(x) b - (m/m_{1}')bp_{k, \ell m_{1}'}(x) \label{eqn:adx1}\\
	d_{k}^{1}(y) : b \mapsto (m_{1}/m)^{k-1}p_{k, \ell m_{1}}(y) b - (m_{1}'/m)^{k-1}bp_{k, \ell m_{1}'}(y) \label{eqn:ady1}
	\end{gather}
	
	And we repeat the same procedure of multiplying by shift bimodules on the left and right, to get an $A_{r_{1}/m_{1}}(\ell m_{1})\text{-} A_{r_{1}'/m_{1}'}(\ell m_{1}')$-bimodule $B_{2}$, which is finitely generated as a left or right module, and on which the operators (\ref{eqn:adx1}), (\ref{eqn:ady1}) act locally nilpotently. 
	
	{\it Step 7.} Continuing with this procedure, since $\gcd(r,m) = 1 = \gcd(r', m)$, the Euclidean algorithm tells us that we are going to get to an $A_{N_{1}}(\ell)\text{-}A_{N_{2}}(\ell)$-bimodule $\tilde{B}$, where $N_{1}, N_{2}$ are integers, $\tilde{B}$ is finitely generated as either a left or right module, and the operators
	\begin{gather}
	b \mapsto mp_{k, \ell}(x)b - mbp_{k, \ell}(x) \label{eqn:adxx} \\
	b \mapsto (1/m)^{k-1}p_{k, \ell}(y)b - (1/m)^{k-1}bp_{k, \ell}(y) \label{eqn:adyy}
	\end{gather}
	
	\noindent act locally nilpotently. From (\ref{eqn:adxx}), it follows that $\CC[\hh_{\ell}]^{S_{\ell}}$ acts locally nilpotently on $\tilde{B}$. From (\ref{eqn:adyy}) it follows that $\CC[\hh^{*}_{\ell}]^{S_{\ell}}$ acts locally nilpotently on $\tilde{B}$, too. Thus, $\tilde{B} \in \HC(A_{N_{1}}(\ell), A_{N_{2}}(\ell))$.
	
	{\it Step 8.} It is clear that everything we have done in Steps 4-7 can be reversed. So we get a category equivalence $\HC_{\cL_{\ell}}(A_{c}(n), A_{c'}(n)) \cong \HC(A_{N_{1}}(\ell), A_{N_{2}}(\ell))$. Since all our parameters are positive, hence spherical, we get $\HC_{\cL_{\ell}}(H_{c}(n), H_{c'}(n)) \cong \HC(H_{N_{1}}(\ell), H_{N_{2}}(\ell))$. The latter category is equivalent to the category of representations of $\fS_{\ell}$. We are done.
\end{proof}

Let us summarize our results in the two-parametric setting in the following theorem.

\begin{theorem}\label{thm:main2paramA}
	Let $c = r/m, c' = r'/m$ be such that $\gcd(r;m) = 1 = \gcd(r';m')$, and $1 < m \leq n$. Then.
	\begin{enumerate}
		\item Let $i = 0, \dots, \lfloor n/m\rfloor$. Then, $\HC_{\cL_{i}}(H_{c}(n), H_{c'}(n)) = 0$ unless $c - c' \in \ZZ$ or $c + c' \in \ZZ$. If $c - c' \in \ZZ$ or $c + c' \in \ZZ$, then $\HC_{\cL_{i}}(H_{c}(n), H_{c'}(n))$ is equivalent to the category of representations of $\fS_{i}$.
		\item If $n - mi = 0$, then $\HC_{\cL_{i}}(H_{c}(n), H_{c'}(n))$ is equivalent to the category of representations of $\fS_{i}$, without further restrictions on the parameters $c, c'$.
	\end{enumerate}
\end{theorem}

Note that, in particular, if $c = r/m, c' = r'/m$ with $m\ell = n$ but neither $c - c'$ nor $c + c'$ is an integer, then $\HC(H_{c}(n), H_{c'}(n)) = \HC_{\cL_{\ell}}(H_{c}(n), H_{c'}(n))$.

\end{document}